\documentclass[draft]{york-thesis}
\usepackage{hyperref}
\usepackage{amsfonts}
\usepackage{amssymb}
\usepackage{amsmath}
\usepackage{amsthm}
\usepackage{enumerate}
\usepackage{mathrsfs}
\usepackage{tikz}

\newcommand{\mathscripty}{\mathscr}
\newcommand{\cstar}[1]{\mathcal{#1}}

\newcommand{\rs}{\mathord{\upharpoonright}}
\newcommand{\sm}{\setminus}

\newcommand{\e}{\epsilon}

\newcommand{\NN}{\mathbb{N}}
\newcommand{\ZZ}{\mathbb{Z}}
\newcommand{\QQ}{\mathbb{Q}}

\newcommand{\er}{\mathbb{R}}
\newcommand{\ce}{\mathbb{C}}

\newcommand{\PP}{\mathbb{P}}

\newcommand{\SA}{\mathscripty{A}}
\newcommand{\SB}{\mathscripty{B}}
\newcommand{\SC}{\mathscripty{C}}
\newcommand{\SD}{\mathscripty{D}}
\newcommand{\SE}{\mathscripty{E}}
\newcommand{\SF}{\mathscripty{F}}
\newcommand{\SG}{\mathscripty{G}}
\newcommand{\SH}{\mathscripty{H}}
\newcommand{\SI}{\mathscripty{I}}
\newcommand{\SJ}{\mathscripty{J}}

\newcommand{\SL}{\mathscripty{L}}
\newcommand{\SM}{\mathscripty{M}}
\newcommand{\SN}{\mathscripty{N}}
\newcommand{\SP}{\mathscripty{P}}
\newcommand{\SQ}{\mathscripty{Q}}
\newcommand{\SR}{\mathscripty{R}}

\newcommand{\SU}{\mathscripty{U}}

\newcommand{\SX}{\mathscripty{X}}
\newcommand{\SY}{\mathscripty{Y}}
\newcommand{\SZ}{\mathscripty{Z}}

\newcommand{\Cmeas}{\mathrm{C}}

\newcommand{\ZFC}{\mathrm{ZFC}}

\newcommand{\CH}{\mathrm{CH}}
\newcommand{\MA}{\mathrm{MA}}

\newcommand{\Cstar}{\mathrm{C}^*}
\newcommand{\OCA}{\mathrm{OCA}}

\newcommand{\B}{\cstar{B}}

\newcommand{\M}{\SM}
\renewcommand{\P}{\SP}

\newcommand{\U}{\SU}

\newcommand{\en}{\mathbb N}

\newtheorem{theorem}{Theorem}[section]
\newtheorem*{theorem*}{Theorem}
\newtheorem{proposition}[theorem]{Proposition}
\newtheorem*{proposition*}{Proposition}
\newtheorem{lemma}[theorem]{Lemma}
\newtheorem*{lemma*}{Lemma}
\newtheorem{corollary}[theorem]{Corollary}
\newtheorem*{corollary*}{Corollary}
\newtheorem{fact}[theorem]{Fact}
\newtheorem*{fact*}{Fact}
\theoremstyle{definition}
\newtheorem{defin}[theorem]{Definition}
\newtheorem*{defin*}{Definition}
\newtheorem{claim}[theorem]{Claim}
\newtheorem*{claim*}{Claim}
\newtheorem{conjecture}[theorem]{Conjecture}
\newtheorem*{conjecture*}{Conjecture}
\newtheorem{theoremi}{Theorem}
\newtheorem{question}{Question}

\newtheorem{conjecturei}[theoremi]{Conjecture}
\newtheorem{questioni}[theoremi]{Question}
\newtheorem*{question*}{Question}

\theoremstyle{remark}

\newtheorem*{example*}{Example}
\newtheorem{remark}[theorem]{Remark}
\newtheorem*{remark*}{Remark}
\newtheorem*{notation}{Notation}

\newtheorem*{note*}{Note}

\newcommand{\set}[2]{\left\{#1\mathrel{}\middle|\mathrel{}#2\right\}}
\newcommand{\seq}[2]{\left\langle #1\mathrel{}\middle|\mathrel{}#2\right\rangle}

\newcommand{\norm}[1]{\left\lVert #1 \right\rVert}

\newcommand{\mc}[1]{\mathcal{#1}}
\newcommand{\abs}[1]{\left\vert#1\right\vert}
\newcommand{\ball}[1]{#1_{\le 1}}

\DeclareMathOperator{\II}{II}
\DeclareMathOperator{\Triv}{Triv}
\DeclareMathOperator{\Th}{Th}
\DeclareMathOperator{\Id}{Id}
\DeclareMathOperator{\bd}{bd}
\DeclareMathOperator{\Homeo}{Homeo}
\DeclareMathOperator{\ran}{ran}
\DeclareMathOperator{\img}{img}

\DeclareMathOperator{\supp}{supp}

\DeclareMathOperator{\Fin}{Fin}

\DeclareMathOperator{\Aut}{Aut}

\DeclareMathOperator{\Map}{Map}
\DeclareMathOperator{\Skel}{Skel}
\DeclareMathOperator{\Fn}{Fn}

\DeclareMathOperator{\GL}{GL}

\usepackage{enumitem}

\addcontentsline{toc}{section}{Abstract}
\newcommand{\iso}{\Lambda}
\department{Mathematics and Statistics}
\date{April 2017}%

\author{Alessandro Vignati}

\title{Logic and $\Cstar$-algebras: Set theoretical dichotomies in the theory of continuous quotients}%
\begin{document}

\thispagestyle{empty}
\maketitlepage
\pagenumbering{roman}

\begin{abstract}
Given a nonunital $\Cstar$-algebra $A$ one constructs its corona algebra $\mathcal M(A)/A$. This is the noncommutative analog of the \v{C}ech-Stone remainder of a topological space. We analyze the two faces of these algebras: the first one is given assuming $\CH$, and the other one arises when Forcing Axioms are assumed. In their first face, corona $\Cstar$-algebras have a large group of automorphisms that includes nondefinable ones. The second face is the Forcing Axiom one; here the automorphism group of a corona $\Cstar$-algebra is as rigid as possible, including only definable elements. 
 \end{abstract}

 \tableofcontents

\chapter{Introduction}\label{ch:intro}
\pagenumbering{arabic}

This thesis focuses on the interactions between logic (set theory and model theory) and operator algebras (in particular $\Cstar$-algebras). These are established high-profile areas of pure mathematics in their own right, with connections across mathematics. 

$\Cstar$-algebras are Banach  self-adjoint subalgebras of $\mathcal B(H)$, the algebra of bounded operators on a complex Hilbert space $H$. Via the Gelfand transform, abelian $\Cstar$-algebras arise as algebras of continuous functions on locally compact Hausdorff spaces, leading to the guiding philosophy that $\Cstar$-algebras provide non-commutative analogues of topological spaces.  Connections have been established with many important branches of mathematics, including dynamical systems, topology, algebra, geometry, geometric group theory, and number theory. Central notions in mathematics, such as amenability, can be described through operator algebras. The Elliott classification programme - a major research goal of the last 25 years - aims to completely understand simple amenable $\Cstar$-algebras analogously to Connes' Fields Medal-winning work on amenable von Neumann factors.

Dealing with forms of reasoning, logic provides the foundation of mathematics and links it to philosophy and computer science. 

At the core of this dissertation are two key branches of logic: model theory, which studies intrinsic properties that can be expressed with first order statements, and set theory, which focuses on the study of higher cardinalities and the axiomatisation of mathematics.

Connections between logic and operator algebras have a long history with recently renewed impetus in the last two decades. The set of connections between these two areas can be seen as the convex combination of four points: 
\begin{itemize}
\item the application of combinatorial set theoretical principles to construct pathological nonseparable objects;
\item the application of descriptive set theory to classification; 
\item the study on how set theoretical axioms influence the structure of automorphisms groups of corona $\Cstar$-algebras;
\item the development of continuous model theory for operator algebras. 
\end{itemize} This thesis focuses on the last two points.

The applications of set theory to automorphisms of corona $\Cstar$-algebras originated from two sources: 
\begin{enumerate}[label=(\roman*)]
\item\label{intro.enum1} Brown-Douglas-Fillmore's search for a K-theory reversing automorphism of the Calkin algebra $\mathcal C(H)$\footnote{Given a separable Hilbert space $H$, the Calkin algebra is the quotient of $\mathcal B(H)$ by the ideal of compact operators $\mathcal K(H)$.};
\item\label{intro.enum2} the study of homeomorphisms of \v{C}ech-Stone remainders of topological spaces.
\end{enumerate}

\ref{intro.enum1}
Aiming to classify the normal elements of the Calkin algebra, the foundational paper \cite{BDF.Ext} developed extension theory and defined analytic K-homology for algebras. Given an essentially normal $T\in\mathcal B(H)$ ($T$ is essentially normal if $\dot T\in\mathcal C(H)$, which is the image of $T$ under the quotient map, is normal) $T$ can be classified by its essential spectrum $\sigma_e(T)$ and a free group indexed by $\sigma(T)\setminus\sigma_e(T)$, where $\sigma(T)$ is the spectrum of $T$. If $T$ and $S$ are normal operators in $\mathcal B(H)$, then $\dot T$ and $\dot S$ are unitarily equivalent if and only if there is an automorphism of $\mathcal C(H)$ mapping $\dot T$ to $\dot S$,  and this happens if and only $\sigma_e(T)=\sigma_e(S)$. Brown Douglas and Fillmore asked whether, for essentially normal $S$ and $T$, being unitarily equivalent in the Calkin algebra is equivalent to the existence of an automorphism of $\mathcal C(H)$ mapping $\dot T$ to $\dot S$. A positive answer to this question is equivalent to the existence of a $K$-theory reversing automorphism of $\mathcal C(H)$. As every such automorphism would be outer, it was then asked whether it is possible to have outer automorphisms at all. Set theoretic axioms entered play: results of Phillips and Weaver (see \cite{Phillips-Weaver}, or \cite[Theorem 1.1]{Farah.C} for a simpler proof) and Farah (\cite{Farah.C}) show that the question cannot be answered in the usual axiomatization of mathematics ($\ZFC$). 

\ref{intro.enum2}
Given a locally compact space $X$ it is possible to construct a universal compact space $\beta X$ in which $X$ is densely embedded. This space is known as the \v{C}ech-Stone compactification of $X$ and it is, in many senses, the largest compactification of $X$. The study of homeomorphisms of the remainder space $\beta X\setminus X$ has a long history, that began with the following question: is every homeomorphisms of $\beta\NN\setminus\NN$ induced by a function $f\colon \NN\to\NN$ which is (up to finite sets) a bijection? The work of W. Rudin (\cite{Rudin}) and Shelah (\cite{Shelah.PF}) shows that the answer to this question is independent from $\ZFC$. On one hand, Rudin proved that under the Continuum Hypothesis (henceforth $\CH$) it is possible to construct a nondefinable homeomorphism of $\beta\NN\setminus\NN$. On the other hand, Shelah used forcing to exhibit a very rigid model of set theory, in the sense that in this model every homeomorphism of $\beta\NN\setminus\NN$ is trivial (i.e., induced by an ``almost permutation'' as above). Shelah's argument was refined in \cite{Shelah-Steprans.PFAA} and \cite{Velickovic.OCAA}, where it was proved that all homeomorphisms are trivial  if one assumes Forcing Axioms, which are generalizations of the Baire Category Theorem negating $\CH$. The eye-opening monograph \cite{Farah.AQ} follows the guiding philosophy that the structure of the  automorphisms group of discrete quotients depends on which set theoretical axioms are in play. It extended the results obtained for $\beta\NN\setminus\NN$ to quotients of more general Boolean algebras. The work of Dow, Hart, and Yu, attempts to solve similar questions when $X=\er$.

The goal of this thesis is to generalize the results in \ref{intro.enum1} and \ref{intro.enum2} to a more general setting: the one of corona $\Cstar$-algebras. In the same way the compact operators are related to $\mathcal B(H)$ and $\mathcal C(H)$, one can associate to a nonunital $\Cstar$-algebra $A$ its multiplier algebra $\mathcal M(A)$ and its corona $\mathcal M(A)/A$. $\mathcal M(A)$ and $\mathcal M(A)/A$ are the noncommutative analogues of the \v{C}ech-Stone compactification and the \v{C}ech-Stone remainder of a locally compact space. In an attempt to generalize the independence results obtained for the Calkin algebra, one has to generalize the notion of inner automorphism, as there are noncommutative algebras for which  it is possible to construct outer automorphisms of the corona in $\ZFC$ (e.g., if $A=c_0(\mathcal O_2)$). The strict topology on $\mathcal M(A)$, which is Polish on bounded sets if $A$ is $\sigma$-unital, justifies the following definition of triviality: an automorphisms $\Lambda$ of $\mathcal M(A)/A$ is said trivial if its graph
\[
\{(a,b)\mid \Lambda(\pi(a))=\pi(b)\}
\]
is Borel in the strict topology. One can see that inner automorphisms are trivial, and that every trivial automorphism is absolute between models of set theory.

In the abelian case, where $A=C_0(X)$ for a locally compact $X$, automorphisms of $\mathcal M(A)/A$ correspond bijectively to homeomorphisms of $\beta X\setminus X$, linking the study of automorphisms of corona algebras to topology. If $X\neq\NN$, the notion of  ``permutation up to finite sets'' is replaced by ``permutation up to compact subsets of $X$''. 

The following was stated in \cite{Coskey-Farah}:
\begin{conjecturei}[Coskey-Farah]\label{conj:CFconj}
Let $A$ be a $\sigma$-unital nonunital $\Cstar$-algebra. The existence of nontrivial automorphisms of $\mathcal M(A)/A$ is independent from $\ZFC$.
\end{conjecturei}

Various instances of the conjecture have been established. A common factor is the assumption of some low dimensionality hypothesis on $A$. The first result not relying on these assumptions in any way is the following corollary of Theorem~\ref{thm:CHMani}:
\begin{theoremi}
Let $X$ be a locally compact noncompact metrizable manifold. The existence of a  nontrivial homeomorphism of $\beta X\setminus X$ is consistent with $\ZFC$.
\end{theoremi}
The consistency of the existence of a nontrivial homeomorphism of $\beta X\setminus X$ was an open problem even in the most natural cases, e.g., when $X=\er^n$ for $n\geq 2$.

Focusing on non necessarily commutative coronas we prove new instances of the conjecture, using Theorem~\ref{thm:FA.liftingtheorem}. An appealing consequence of Theorem~\ref{thm:Borel} is the following:
\begin{theoremi}\label{thmi:1}
Let $A$ be a unital nuclear separable $\Cstar$-algebra and $B$ its stabilization $A\otimes\mathcal K(H)$. Then the existence of nontrivial automorphisms of $\mathcal M(B)/B$ is independent from $\ZFC$.
\end{theoremi}

The consequences of Theorem~\ref{thm:FA.liftingtheorem} go beyond the study of automorphisms of coronas. As already noted for discrete structures (see \cite{Farah.AQ}), the existence or nonexistence of embeddings between different quotient structures may be independent from $\ZFC$. However, for some coronas the existence of certain embeddings is absolute between models of set theory. For instance, it can be proved that the algebra $B=\prod M_n/\bigoplus M_n$ embeds in the Calkin algebra without any set theoretical assumption. It is natural to ask whether modified versions of $B$ can be embedded into the Calkin algebra in $\ZFC$. If $\SI\subseteq\mathcal P(\NN)$ is any ideal the algebra $\prod M_n/\bigoplus_{\SI} M_n$ under $\CH$ embeds into $B$, and consequently into the Calkin algebra. The following corollary of Theorem~\ref{thm:noinjection} shows that this not always the case.
\begin{theoremi}\label{thmi:2}
Let $\SI\subseteq\mathcal P(\NN)$ be a meager dense ideal. It is consistent with $\ZFC$ that $\prod A_n/\bigoplus _{\SI}A_n$ does not embed in the Calkin algebra for any choice of $A_n$ unital nonzero $\Cstar$-algebras.
\end{theoremi}

It is not known whether it is consistent with $\ZFC$ that algebras of the form $\prod A_n/\bigoplus_{\SI} A_n$ embed into the Calkin algebra (independently from the choice of $A_n$). A positive answer to the following question would imply it.
\begin{questioni}\label{q:intro1}
Is it consistent with $\ZFC$ that every $\Cstar$-algebra of density character $\aleph_1$ embeds into the Calkin algebra?
\end{questioni}
The difficulties of answering this question are of model theoretical nature.

The recent formalization of continuous model theory relies on the work in \cite{BYBHU}. This approach was applied  to operator algebras  in \cite{FHS.I}, \cite{FHS.II}, and \cite{FHS.III}, where model theory for $\Cstar$-algebras and tracial von Neumann algebras was developed. A major recent progress  in this direction is the large scale 142 page monograph \cite{bourbaki} (in which I was a major contributor), which explores the model theory of amenable $\Cstar$-algebras. This work puts amenable $\Cstar$-algebras under a model theory lens. It shows that many properties related to the Elliott classification programme (such as amenability itself, Toms and Winter's strong self absorption (\cite{TW.SSA}), and several notions of dimension) can be approached in a model theoretical way. The main and most ambitious goal of the applications of model theory to $\Cstar$-algebras is to construct novel and exotic examples of nuclear $\Cstar$-algebras. The existence of these new objects is crucial to many important open problems in $\Cstar$-algebras (such as the Toms-Winter conjecture, or the existence of a nuclear stably finite $\Cstar$-algebra which is not quasidiagonal). The potential consequences of these original ways of thinking are profound. For instance, recent work of Goldbring and Sinclair, provides insights on longstanding problems related to quasidiagonality and the UCT (for an overview on these concepts, see \cite{WinterAbel}).

In this setting, we study the concept of  countable saturation, a model theoretical property shared by ultrapowers and reduced products. This notion is tied to $\CH$: if $C$ is a $\Cstar$-algebra which is countably saturated and $B$ is an algebra of character density $\aleph_1$ whose separable subalgebras embed into $C$, then under $\CH$ it is possible to show that $B$ itself embeds into $C$. The core difficulty in answering Question~\ref{q:intro1} is that the Calkin algebra fails to be countably saturated. We analyze weaker versions of the concept of countable saturation which are shared by coronas of $\sigma$-unital algebras. The weaker of these layers of saturation, known as countable degree-$1$ saturation, was  considered in a different setting by Kirchberg under the name ``$\epsilon$-test" (\cite{Kirch.CentralSeq}). Such weakenings provide a uniform setting for properties shared by coronas of $\sigma$-unital algebra, such as the following (see \cite{farah2011countable}): being AA-CRISP, sub-$\sigma$-Stonean, satisfying the conclusion of Kasparov's Technical Theorem, and so on.

We expand the class of algebras with this property (see Corollary~\ref{cor:vnAsigmafinite}):
\begin{theoremi}
Let $M$ be a von Neumann algebra with an infinite $\sigma$-finite trace $\tau$ and $\mathcal K_\tau$ be the ideal of finite trace elements. Then $M/\mathcal K_\tau$ is countably degree-1 saturated.
\end{theoremi} 

Another connection between countable saturation and $\CH$ is given by the fact that, under $\CH$, coronas of $\sigma$-algebras which are countably saturated have nontrivial automorphisms. One may ask whether the full power of countable saturation is needed for this result. A negative answer to the following question would provide solutions to one side of Conjecture~\ref{conj:CFconj}:
\begin{questioni}\label{q:intro2}
Under $\CH$, is there an infinite-dimensional countably degree-1 saturated algebra of density character $\aleph_1$ with only $\mathfrak c$-many automorphisms?
\end{questioni}

For a survey on the different layers of saturation see \cite{farah2011countable}. For a detailed description of the applications of continuous model theory to $\Cstar$-algebras see \cite{bourbaki}.

\section{Structure of the thesis}
 
In Chapter~\ref{ch:prel} we introduce notation and  preliminary notions.

Chapter~\ref{ch:CH} is focused on model theory and the consequences of $\CH$. In \S\ref{s:CH.MTforcstar} we introduce the concept of saturation and its different layers. We prove that quotients of certain von Neumann algebras (Theorem~\ref{thm:FactorCtbleSat})  and some abelian algebras (Theorems \ref{thm:CHSat.Abelian.Conds} and \ref{thm:CHSat.Abelian.Conds2}) have certain degrees of saturation. The results of this section come from joint work with Eagle contained in \cite{EV.Sat}. In \S\ref{ss:CH.consofCH} we analyze how the existence of a plethora of embedding between corona $\Cstar$-algebras can be proved, combining $\CH$ and the saturation of certain $\Cstar$-algebras. In \S\ref{s:CH.SCmani}, and specifically in Theorem~\ref{thm:CHMani}, we prove that $\CH$ implies the existence of many nondefinable homeomorphisms for the \v{C}ech-Stone remainder of a manifold, and consequently of automorphisms of the associated corona $\Cstar$-algebra. The results of this section are contained in \cite{V.Nontrivial}.

 In Chapter~\ref{ch:Ulam} we introduce and develop a strong concept of stability for maps between $\Cstar$-algebras known as Ulam stability. The main results of this chapter are Theorem~\ref{thm:US.USFinDim} (showing that approximate maps from a finite-dimensional $\Cstar$-algebra to any $\Cstar$-algebra are close to $^*$-homorphisms by a factor independent from the domain and the codomain) and Corollary~\ref{cor:US.AF2}, that proves the same stability result holds if  the domain is a unital AF algebra and the codomain is a von Neumann algebra. The results of this chapter are contained in \cite{MKAV.UC}, a joint work with McKenney.
 
 Chapter~\ref{ch:FA} focuses on the consequences of Forcing Axioms on the structure of automorphisms of coronas. First, we prove a powerful lifting result - Theorem~\ref{thm:FA.liftingtheorem}. Stating and proving this takes all of \S\ref{s:FA.Lift1} and \S\ref{s:FA.Lift2}. Then, in \S\ref{s:FA.consofLT} we prove  Theorem~\ref{thm:Borel}. This shows that Forcing Axioms imply that all automorphisms of the corona of a separable nuclear $\Cstar$-algebra carrying an approximate identity of projections are trivial. In \S\ref{s:FA.Emb} we use again Theorem~\ref{thm:FA.liftingtheorem} to show that many of the embeddings constructed from $\CH$ in \S\ref{ss:CH.consofCH} cannot exist in the presence of Forcing Axioms. Part of these results are obtained in joint work, yet unpublished, with McKenney.

Lastly, Chapter~\ref{ch:QandA} contains a list of open questions and hints for future developments of the research presented.
\section{Acknowledgments}

I would like to thank the Fields Institute for letting me use their common spaces in the last 5 years. Also, I would also like to thank Ilijas Farah for funding most of my travels to conferences, meetings and workshops with his grant. Similarly, I want to thank the Association of Symbolic Logic, and the organizers of conferences and meetings for the support. I particularly want to thank Institut Mittag-Leffler for having me as a Postdoctoral Fellow for the programme on Classification of Operator Algebras: Complexity, Rigidity, and Dynamics in Winter 2016, and the organizers of the Focusing Programme on Classification of $\Cstar$-algebras in M\"unster for having me in Summer 2015. Finally, I have been supported by the Susan Mann Scholarship at York University, and I would like to thank the donors for funding the scholarship.

I would like to thank the organizers and participants of the Toronto Set Theory seminar and the Operator Algebras Seminar for  many interesting and fruitful conversations.

I wish to thank Prof. George Elliott for the many, sometimes long,  conversations. When I came to Toronto, I didn't know what a $\Cstar$-algebra was. It is his great merit if I managed to scratch the surface of this interesting area of mathematics.

Finally, I would like to thank my advisor, Prof. Ilijas Farah. Whenever I had a wrong, bad, or terrible idea (a high probability event) he never told me so, but always made me realize it by myself. For this reason and many others he was the best mentor I could've asked for.

\chapter{Preliminaries and Notation}\label{ch:prel}

The notation  in  use is quite standard. $\NN$ (or $\omega$) represents the set of natural numbers (including $0$), $\ZZ$ the integers, $\QQ$ the rationals, $\er$ the reals and $\ce$ the complex numbers. $\omega_1$ is the first uncountable ordinals, with $\mathfrak c$ being the cardinality of the Continuum (i.e., $\mathfrak c=|\er|$). Cardinals are denoted by the $\aleph_0,\aleph_1$, and so on. If $X$ is a set, $\mathcal P(X)$ is its power set. If $X$ is a locally compact topological space, $\beta X$ denotes its \v{C}ech-Stone compactification.

We use the variables $A$, $B$, .. to denote $\Cstar$-algebras, and the variables $f$, $g$, ... for functions, while $\phi,\psi$.. usually denote maps. This notation is used everywhere but for Chapter~\ref{ch:FA}, where the amount of notation needed forces us to use $\cstar{A}$, $\cstar{B}$, ... for $\Cstar$-algebras and leave the variables $A$, $B$, .. to denote subsets of $\NN$.

If $f\colon X\to Y$ is a function and $Z\subseteq X$, $f[Z]$ is the pointwise image of $Z$. Finally, the symbol $\exists^\infty n$ reads ``there are infinitely many $n$'', while $\forall^\infty n$ is for $\exists n_0\forall n\geq n_0$.

The usual approach to set theory is carried over in \cite{kunen:settheory}, where Chapter VII represents the standard introduction to forcing. A good introductory approach to Forcing Axioms can be found in \cite{Moore.PFA}, while specific results are contained in \cite{Farah.AQ} and \cite{Todorcevic.PPIT}. A standard reference for descriptive set-theoretic results is \cite{Kechris.CDST}.

For results in $\Cstar$-algebras and von Neumann algebras we will often refer to \cite{Blackadar.OA} or \cite{Pedersen.CAAG}. \cite{Runde.LA} represents the standard text for whoever is interested in amenability and related topics. Another good approach to $\Cstar$-algebras can be found in \cite{David:CstarEx}.

\section{Set Theory}\label{s:Prel.ST}

Set Theory is the study of the infinite and of the axiomatization of mathematics. The interest in modern set theory can be traced back to the work of Cantor, Russell, Peano, Dedekind, Zermelo, Fraenkel, Hilbert, von Neumann and G\"odel among others. The usual axiomatization of mathematics nowadays used is known as the Zermelo-Fraenkel system of axioms, together with the Axiom of Choice AC. This scheme of axioms is known as $\ZFC$. By G\"odel incompleteness theorem, if $\ZFC$ is consistent, it cannot prove it. Nevertheless, $\ZFC$ is the setting in which modern mathematics is developed.

\subsection{Descriptive set theory}\label{ss:descr}
A topological space is said to be \emph{Polish} if it is separable and completely metrizable. As all compact metrizable spaces are Polish, so is $\mathcal P(\NN)$ when identified with $2^{\NN}$ endowed with the product topology. If $X$ is Polish and $Y\subset X$, we say that $Y$ is \emph{meager} if it is a countable union of closed and nowhere dense sets. $Y$ is said \emph{Baire} if it has meager symmetric difference with an open set and \emph{analytic} if it is the continuous image of  a Borel subset of a Polish space. If $X$ and $Y$ are Polish and $f\colon X\to Y$ we say that $f$ is \emph{Baire-measurable} if the inverse image of every open set is Baire, and $\Cmeas$-\emph{measurable} if it is measurable with respect to the $\sigma$-algebra generated by analytic sets. $\Cmeas$-measurable functions are, in particular, Baire-measurable (see~\cite[Theorem~21.6]{Kechris.CDST}).
The proof of the following can be found in \cite[Lemma 1.3.17]{BJ.SetTheory}.

\begin{theorem} \label{thm:STPrel.talagrand}
  
Let $Y_n$ be finite sets.  A set $G\subseteq\prod Y_n$ is comeager if and only if there is a partition $\seq{E_i}{i\in\NN}$ of $\NN$ into intervals, and a sequence $t_i\in \prod_{n\in E_i} Y_n=Z_i$, such that 
\[
\exists^{\infty}i(y\restriction Z_i=t_i)\Rightarrow y\in G.
\]
\end{theorem}

If $X,Y$ are sets and $Z\subseteq X\times Y$ a function $f\colon X\to Y$ is said to \emph{uniformize} $Z$ if for every $x\in X$
\[
\exists y(x,y)\in Z\Rightarrow (x,f(x))\in Z.
\]
By the Axiom of Choice, it is always possible to find uniformizing functions. The goal of what follows is to find well-behaved uniformizing functions. This is known as the Jankov-von Neumann Theorem (see \cite[Theorem~18.1]{Kechris.CDST}).

\begin{theorem}\label{thm:STPrel.JVN}
Let $X,Y$ be Polish and $Z\subseteq X\times Y$ be analytic. Then $Z$ has a  $\Cmeas$-measurable uniformization.
\end{theorem}
In general is not possible to uniformize Borel sets with a Borel function, but this is the case when the vertical sections of $Z$ are well behaved. For a proof of the following, see \cite[Theorem~8.6]{Kechris.CDST}.

\begin{theorem}\label{thm:STPrel.BorelUnif}
  Let $X,Y$ be Polish and $Z\subseteq X\times Y$ be Borel. Suppose further that for all $x\in X$ we have $\{y\mid (x,y)\in Z\}$ is either  empty or nonmeager. Then $Z$ has a  Borel function uniformization.
\end{theorem}

\subsection{Ideals in $\mathcal P(\NN)$}\label{ss:ST.idealsonomega}

 A subset $\SI\subseteq\mathcal P(\NN)$ is \emph{hereditary} if $X\in\SI$ and $Y\subseteq X$ implies $Y\in\SI$, and it is an \emph{ideal} if it is hereditary and closed under finite unions. The easiest example of an ideal is the one of finite sets, $\Fin$. If for every infinite $X\subseteq\NN$ there is an infinite $Y\subseteq X$ with $Y\in\SI$, the ideal is said \emph{dense}.

A family $\SF\subseteq\mathcal P(\NN)$ of infinite sets is \emph{almost disjoint} (a.d.) if for every distinct $X,Y\in\SF$ we have that $X\cap Y$ is finite.  An a.d. family is \emph{treelike} if there is a bijection $f \colon \NN\to 2^{<\omega}$ such that for every $X\in\SF$, $f[X]$ is a branch through $2^\NN$, i.e., a pairwise comparable subset of $2^{<\omega}$.  An ideal $\SI\subseteq\mathcal P(\NN)$ is \emph{ccc /}$\Fin$ if $\SI$ meets every uncountable, a.d. family $\SF\subseteq\mathcal P(\NN)$. (Note that this is slightly stronger then asking for the quotient $\mathbb P(\NN)/\SI$ to be ccc as a poset, in the terminology of \S\ref{ss:FA}, see \cite{Farah.AQ}).

The following are applications of of Theorem \ref{thm:STPrel.talagrand}.

\begin{proposition}\label{prop:JT2}
  If $\SJ$ is a meager ideal on $\mathcal P(\NN)$, if and only if there is a partition $\NN = \bigcup\set{E_n}{n\in\NN}$ into finite  intervals such that for any infinite set $L$, $\bigcup\set{E_n}{n\in L} \not\in\SJ$.
\end{proposition}

\begin{proposition}\label{prop:intersectionofnonmeager}
Let $\SI,\SJ\subseteq\mathcal P(\NN)$ be hereditary and nonmeager. Then so is $\SI\cap\SJ$. Moreover, if $\SI_n$ is a sequence of hereditary nonmeager sets such that $\Fin\subseteq\SI_n$ for each $n$, then $\bigcap\SI_n$ is hereditary and nonmeager.
\end{proposition}
Note that no nontrivial analytic ideal containing the finite sets can be nonmeager (\cite[Lemma~3.3.2]{Farah.AQ}), so there are only countably many analytic nonmeager ideals (one for each finite set).

\subsection{The Continuum Hypothesis}\label{ss:CH}
The Continuum Hypothesis ($\CH$) is the statement that every uncountable subset of the real line $\er$ has the same cardinality of the real line itself. This statement can be rephrased by asking that $|\mathcal P(\NN)|=\omega_1$, or that there is a well-order of the reals whose initial segments are countable.

The problem of whether $\CH$ is true or false was listed as the first of the famous list Hilbert presented to the mathematical community in 1900, even before a formal axiomatization of mathematics was completed. In 1940 G\"odel showed that $\CH$ holds in the constructible universe $L$, therefore proving that $\CH$ cannot be disproved in $\ZFC$.  Later in the early `60 Cohen introduced the groundbreaking technique of forcing and used it to prove the independence of $\CH$ from $\ZFC$ (see \cite{Cohen.CH1} and \cite{Cohen.CH2}). This shows that Hilbert's first problem cannot be solved inside $\ZFC$. 


Some important consequences of $\CH$ are the following:
\begin{itemize}
\item there are nontrivial automorphisms of $\ell_\infty/c_0$ (\cite{Rudin});
\item  if $X$ is a $0$-dimensional locally compact noncompact Polish space then $\beta X\setminus X$ is homeomorphic to $\beta\NN\setminus \NN$ (Parovi\v{c}enko's Theorem);
\item every compact Hausdorff space of density $\mathfrak c$ is a surjective image of $\beta\NN\setminus\NN$ (\cite{Paro.Universal}, or see \cite{KP.STCC});
\item  every connected compact Hausdorff space of density $\mathfrak c$ is a surjective image of $\beta[0,1)\setminus[0,1)$ (\cite{DowHart.Universal});
\item the Calkin algebra $\mathcal C(H)$ has outer automorphisms (\cite{Phillips-Weaver}).
\end{itemize}
 Each of these statements, but one, needs $\CH$, as it was proved to be independent from $\ZFC$ \footnote{That $\beta[0,1)\setminus [0,1)$ surjects onto every continuum of density $\mathfrak c$ is not known to be independent from $\ZFC$.}.

The assumption of $\CH$ has an impact on the cardinality and the structure of the automorphisms group of quotient structures. In both the discrete case (such as certain quotients of Boolean algebras, see~\cite{Farah.AQ}) and the continuous one (such as corona $\Cstar$-algebras, see~\ref{s:multicorona}), it is either proved or conjectured that $\CH$ gives a huge amount of automorphisms, inferring consequently the existence of nondefinable ones.

\subsection{Forcing and Forcing Axioms}\label{ss:FA}

The method of forcing was introduced by Cohen to prove the independence of $\CH$ from $\ZFC$. The general idea of forcing consists of starting with a model of $\ZFC$ to build a second model, constructed from a generic object. The technique of forcing is  capable of modifying the truth value of certain high-complexity statements from the first model (the ground model) to the second one (the forcing extension). The initial goal of the development of forcing was to generate a counterexample to $\CH$, and more sophisticated forcings have been constructed to generate (or create obstruction to the existence of) morphisms between different mathematical structures.
Although forcing is not capable of modifying the truth value of statements of relatively low complexity (see \cite[Chapter 25]{jech:settheory}), its development led to the proof of many celebrated consistency results. For an introductory approach to forcing see \cite{kunen:settheory} or \cite{Shelah.PF}.
  
Forcing Axioms were introduced as an alternative to $\CH$ and as generalizations of the Baire Category Theorem. They assert that the universe of sets has a strong degree of closure when generic objects are formed by sufficiently non pathological forcings. Different Forcing Axioms arise once the exact definition of non pathological is given. The first Forcing Axiom to be stated was Martin's Axiom and the last, and provably the strongest, is Martin's Maximum MM, (\cite{FMS.MM1}). For a great overview on Forcing Axioms see \cite{Moore.PFA}.

In this thesis we will always use consequences of MM: Martin's Axiom at level $\aleph_1$, $\MA_{\aleph_1}$, and $\OCA_\infty$, a strengthening of  Todor{\v{c}}evi{\'c}'s Open Coloring Axiom $\OCA$. It should be noted that both these axioms follow from Shelah's Proper Forcing Axiom PFA, itself a consequence of MM. Also, while both MM and PFA need a supercompact cardinal to be proven consistent, both $\OCA_\infty$ and $\MA_{\aleph_1}$ are provable to be consistent from the consistency of $\ZFC$ without any additional cardinal axioms. Notably, $\OCA$ holds in Woodin's canonical model for the failure of $\CH$ (\cite{Larson.OCA}).

\subsubsection{Martin's Axiom}

Martin's Axiom is a generalization of the Baire Category Theorem isolated by Martin from the work of Solovay and Tennenbaum on Souslin's Hypothesis.

Let $\PP$ be a partially ordered set (poset, or sometimes, forcing) with a largest element.  Two elements of $\PP$ are called \emph{incompatible} if there is no element of $\PP$ below both of them. A set of pairwise incompatible elements is an \emph{antichain}. If all antichains of $\PP$ are countable, $\PP$ is said to have the \emph{countable chain condition} (ccc). A set $D\subseteq \PP$ is said \emph{dense} if $\forall p\in\PP\exists q\in D$ with $q\leq p$. A \emph{filter} $G\subseteq\PP$ is an upward closed downward directed set. 

Martin's Axiom at the cardinal $\kappa$ (written $\MA_\kappa$) asserts that if $\PP$ is a ccc poset, given a family of dense $D_\alpha\subseteq\PP$ ($\alpha < \kappa$), there is a filter $G\subseteq\PP$ such that $G\cap D_\alpha\neq\emptyset$ for every $\alpha < \kappa$.

$\MA_{\aleph_0}$ is a restatement of the Baire Category Theorem, which is a theorem of $\ZFC$. Also, the negation of $\MA_{\mathfrak c}$ follows from $\ZFC$,  therefore $\MA_{\aleph_1}$ contradicts $\CH$. 

\subsubsection{The Open Coloring Axiom}

The axiom $\OCA$, sometimes denoted TA, was introduced by  Todor{\v{c}}evi{\'c} in \cite{Todorcevic.PPIT}. It is a modification of a coloring axiom introduced by Abraham, Rubin and Shelah in \cite{ARS} and generalizes Baumgartner's Axiom BA. We now introduce $\OCA_\infty$, an infinitary version of $\OCA$ introduced by Farah in \cite{Farah.CNOC}, and itself a consequence of PFA.

If $X$ is a set, $[X]^2$ denotes the set of unordered pairs of elements of $X$.  $\OCA_\infty$ is the following statement.  For every separable metrizable space $X$ and every sequence of partitions $[X]^2=K_0^n\cup K_1^n$, if every $K_0^n$ is open in the product topology on $[X]^2$ and $K_0^{n+1}\subseteq K_0^n$ for every $n$, then either
\begin{enumerate}
  \item there are $X_n$ ($n\in \NN$) such that $X=\bigcup_n X_n$ and $[X_n]^2\subseteq K_1^n$ for every $n$, or
  \item there is an uncountable $Z\subseteq 2^\NN$ and a continuous injection $f\colon Z\to X$ such that for all $x\neq y\in Z$ we have
  \[
    \{f(x),f(y)\}\in K_0^{\Delta(x,y)}
  \]
  where $\Delta(x,y)=\min\set{n}{x(n)\neq y(n)}$.
\end{enumerate}
$\OCA$ is the restriction of $\OCA_{\infty}$ to the case where $K_0^n=K_{0}^{n+1}$ for every $n$. It is not known whether the two are equivalent, but $\OCA$ is sufficient to contradict $\CH$. Whether $\OCA$ implies $\mathfrak c=\omega_2$ is an open question but, if one assumes $\OCA$ and its initial formulation as given in \cite{ARS}, then $\mathfrak c=\omega_2$ (\cite{Moore.OCA}).

\subsubsection{The $\Delta$-system Lemma}\label{sss:cccDeltasystem}

We now state a very useful result, known as the $\Delta$-system Lemma. The $\Delta$-system Lemma is usually used to prove that some forcing is ccc, by arguing on the possible properties that an uncountable antichain must have  (see for example \cite[Lemma VII.5.4 or Lemma VII.6.10]{kunen:settheory}). We will do so in the proof of Lemma~\ref{lemma:oca->alternative}, see Claim~\ref{claim:provingisccc}.
 
 \begin{defin}
 A family of sets $\SA$ is a $\Delta$-system if there is $r$ such that $a\neq b\in\SA\Rightarrow a\cap b=r$.
 \end{defin}

The existence of $\Delta$-systems in large collections of sets follows from the following, known as the $\Delta$-system~Lemma (see \cite[Theorem~II.1.5]{kunen:settheory}).
\begin{lemma}\label{lem:DeltaSystem}
If $\SA$ is an uncountable family of finite sets then there is an uncountable subfamily $\SB\subseteq\SA$ which is a $\Delta$-system.
\end{lemma}
\subsubsection{Cardinal invariants}\label{sss:cardinv}

In case $\CH$ fails, one can characterize the properties of the different cardinals between $\aleph_0$ and $\mathfrak c$. In general cardinal invariants characterize the minimal cardinality of  sets satisfying certain conditions. 
There are many cardinal invariants that can be defined, and the theory of cardinal invariants is wide and complex (see \cite{BJS.Cichon}, or \cite{Monk.CI}). Recently, the use of cardinal invariants in the Calkin algebra has been carried over, most notably in \cite{Zamora}. 

We will use only two cardinal invariants: the bounding number $\mathfrak b$ and the dominating number $\mathfrak d$. They relate to subsets of $\NN^\NN$ when considered with the order
\[
f_1\leq^* f_2 \iff\forall^\infty n (f_1(n)\leq f_2(n)).
\]
 They are defined as follows:
\[
\mathfrak b=\min\{ |X|\colon X \text{ is unbounded in } (\NN^\NN,\leq^*)\},
\]
\[
\mathfrak d=\min \{|X|\colon X \text{ is cofinal in } (\NN^\NN,\leq^*)\}.
\]
It is clear that $\omega_1\leq\mathfrak b\leq\mathfrak d\leq\mathfrak c=|\NN^\NN|$. We will use that $\CH$ implies $\mathfrak d=\omega_1$ in \S\ref{s:CH.SCmani}, and that $\OCA$ pushes $\mathfrak b$ above $\omega_1$ in Chapter~\ref{ch:FA}.

\section{$\Cstar$-algebras}\label{s:Cstar}

An abstract $\Cstar$-algebra $A$ is a complex Banach algebra together with an isometric involution $^*\colon A\to A$ with the property that $(a^*)^*=a$, $(ab)^*=b^*a^*$, $(\lambda a)^*=\bar \lambda a^*$  and $\norm{a}^2=\norm{a^*a}$ for all $\lambda\in\ce$ and $a\in A$ ($\norm{a}^2=\norm{aa^*}$ is known as the $\Cstar$-equality). $\Cstar$-algebras were introduced as $B^*$-algebras by Rickart in 1946. Later, Segal, referred to $\Cstar$-algebras as $^*$-closed Banach subalgebras of $\mathcal B(H)$, the algebra of bounded linear operators on a complex Hilbert space $H$. Such objects are known as concrete $\Cstar$-algebras. Given an abstract $\Cstar$-algebra $A$, the Gelfand-Naimark-Segal construction (\cite[II.6.4]{Blackadar.OA}) shows that to a positive linear functional of norm $1$ on $A$ one can canonically associate a representation of $A$ into $\mathcal B(H)$. By considering the direct sum of every possible such representation,  the Gelfand-Naimark Theorem establishes that every abstract $\Cstar$-algebra is isomorphic to a concrete one. We will therefore not distinguish between abstract and concrete $\Cstar$-algebras.

An interesting class of $\Cstar$-algebras is the one of abelian algebras. The Gelfand transform (\cite[II.2.2]{Blackadar.OA}) shows that every abelian $\Cstar$-algebra $A$ is isomorphic to $C_0(X)$, the algebra of complex valued continuous functions on a locally compact space $X$ vanishing at $\infty$. Operations are performed pointwise, the adjoint is given by the conjugate function, and the norm is supremum norm. $A$ has a unit if and only if $X$ is compact. For this reason $\Cstar$-algebras are often seen as noncommutative topological spaces.

In a $\Cstar$-algebra $A$ one can isolate certain sets of elements described by their algebraic properties: the \emph{self-adjoints} (for which $a=a^*$), the \emph{positives} (if there is $b$ such that $a=bb^*$),  the \emph{projections} ($a=a^2=a^*$) and the \emph{unitaries} ($aa^*=1=a^*a$). The self-adjoints carry an order given by $a\leq b$ if and only if $b-a$ is positive. If $A$ is  a $\Cstar$-algebra, $A_{\leq 1}$, $A_1$, $A^+$, and $\U(A)$ denote the unit ball of $A$, its boundary, the set of positive elements, and of unitaries in $A$ respectively. An element of $A$ commuting with every other element of $A$ is said \emph{central}. The set of all central elements is the center of $A$, denotes by $Z(A)$.

As in the category of $\Cstar$-algebras morphisms are $^*$-homomorphisms ($^*$-preserving Banach algebra homomorphisms), for a  subalgebra we will always mean a $\Cstar$-subalgebra.  An injective $^*$-homomorphism is said an \emph{embedding}. If $A$ is unital, a $^*$-homomorphism $\phi\colon A\to B$ does not have to be unital, even if $B$ is. On the other hand, the image of the unit is always a projection. In case $\phi(1_A)=1_B$, we will talk of unital $^*$-homomorphisms and unital embeddings.

\subsection{Examples of $\Cstar$-algebras}

The easiest example of a $\Cstar$-algebra is $\mathcal B(H)$. If $H$ is finite dimensional, then $H\cong\ce^n$ for some $n$, and in this case $\mathcal B(H)=M_n(\ce)$ is a $\Cstar$-algebra with the usual operations, the $\ell^2$-norm, and the involution given by the transpose conjugation. If a $\Cstar$-algebras is finite-dimensional (as a vector space), then it is isomorphic to a finite direct sum of matrix algebras (see \cite[II.8.3.2.(iv)]{Blackadar.OA}). 

Other examples of $\Cstar$-algebras, as already noted, arise from a locally compact space $X$ by considering $C_0(X)$.  An example of a noncommutative nonunital $\Cstar$-algebra arises  by considering an infinite dimensional Hilbert space $H$ and constructing $\mathcal K(H)$, the algebra of all compact operators on $H$, i.e., the norm closure of the algebra of operators with finite-dimensional range.

There are several ways to build interesting $\Cstar$-algebras from the ones we have already described.
\begin{itemize}
\item  the unitization: if $A$ is a nonunital $\Cstar$-algebra we can construct the smallest unital $\Cstar$-algebra containing $A$, denoted by $\tilde A$.  It is isomorphic, as a Banach space, to $A\oplus\ce$. It corresponds to the one-point compactification of a locally compact space.
\item the direct sum: if $A$ and $B$ are $\Cstar$-algebras so is $A\oplus B$, with pointwise operations and the norm given by $\norm{a\oplus b}=\max\norm{a},\norm{b}$. If $I$ is a net and $A_i$, $i\in I$ are $\Cstar$-algebras so is the algebra $\bigoplus_{i\in I} A_i$, the algebra of all sequences $(a_i)_{i\in I}$ with $a_i\in A_i$ and $\lim_{i\in I}\norm{a_i}=0$. This can be seen as the closure of the algebra of ``eventually'' zero sequences.
\item the direct product: if $A_i$ are $\Cstar$-algebras, for $i\in I$, so is $\prod_{i\in I} A_i$, the algebra of all bounded sequences $(a_i)_{i\in I}$, with coordinatewise operations. The norm is the supremum norm.
\item $C(X,A)$: if $X$ is a locally compact space and $A$ is a $\Cstar$-algebra, $C_0(X,A)$ is the $\Cstar$-algebra of all continuous $f\colon X\to A$ vanishing at $\infty$, with pointwise operations and the supremum norm. $C_0(X,A)$ is unital if and only if $X$ is compact and $A$ is unital.
\item inductive limits: if $A_n$ are $\Cstar$-algebras and $\phi_n\colon A_n\to A_{n+1}$ are $^*$-homomorphisms, it is possible to define the limit object $A=\lim (A_n,\phi_n)$, with the obvious operations and norm. If each $\phi_n$ is the inclusion map, then $\bigcup A_n $ is dense in $A$. A similar definition can be made if one allows the index set to be any net. Important limit algebras we will use are the  unital UHF algebras (limit of full matrix algebras, where each embedding is unital) and AF algebras (limits of finite-dimensional algebras). 
\end{itemize}
Other interesting constructions of new $\Cstar$-algebras from old ones are tensor products and quotients.

\subsection{Ideals}\label{ss:idealscstar}
 If a subalgebra $I\subseteq A$ has the additional property that for all $b\in I, a\in A$ we have $ab,ba\in I$ then $I$ is said an \emph{ideal}. If $I\subseteq A$ is an ideal, the quotient $A/I$ is a $\Cstar$-algebra and the quotient map $\pi\colon A\to A/I$ is a surjective $^*$-homomorphism. An algebra with no nontrivial ideals is said \emph{simple}.
 
 Easy examples of ideals are obtained if $C=A\bigoplus B$. In this case, both $A\oplus 0$ are $0\oplus B$ are ideals in $C$. Another, more interesting, example of ideal arises if one considers a locally compact $X$ and a closed set $Y\subseteq X$. The set of functions which are equal to $0$ on $Y$ is an ideal of $C_0(X)$.
 
Particular, very important for our work, cases of ideals are the so called essential ones. An ideal $I\subseteq A$ is \emph{essential} if whenever $J $ is a nonzero ideal in $A$ then $I\cap J\neq\{0\}$. An example of an essential ideal is $\mathcal K(H)$ when seen as an ideal of $\mathcal B(H)$ (in fact, $\mathcal K(H)$ is the only nontrivial ideal of $\mathcal B(H)$). We will call the quotient  $\mathcal B(H)/\mathcal K(H)$ the \emph{Calkin} algebra, denoted as $\mathcal C(H)$. Another example is obtained  from a locally compact noncompact space $X$. In this case the algebra $C_0(X)$ is an essential ideal of $C_b(X)$, the algebra of all bounded continuous function from $X$ to $\ce$. Equally, if $A$ is a $\Cstar$-algebra, then $C_0(X,A)$ is an essential ideal of $C_b(X,A)$.
 
Fix now a sequence of $\Cstar$-algebras $A_1,\ldots,A_n,\ldots$ and an ideal $\SI\subseteq\mathcal P(\NN)$. Defining
 \[
 \bigoplus_\SI A_n=\{(a_n)\in\prod A_n\mid \lim\sup_{n\in\SI}\norm{a_n}=0\},
 \]
 where 
 \[
 \lim\sup_{n\in\SI}\norm{a_n}=\inf_{X\in\SI}\sup_{n\notin X}\norm{a_n},
 \]
we have that $\bigoplus_\SI A_n$ is an essential ideal of $\prod A_n$.

We now list few properties of essential ideals. For their proofs, see~\cite{Skou:NotesMult}.
\begin{proposition}
Let $A$ be a $\Cstar$-algebra. Then
\begin{itemize}
\item if $A$ is nonunital, then $A$ is an essential ideal of its unitization.
\item $I\subseteq A$ is an essential ideal if and only if  it is an ideal and whenever $b\in A$ and $bI=Ib=0$ then $b=0$.
\item if $I\subseteq A$ is an essential ideal and both $I$ and $A$ are unital then $1_I=1_A$ and so $I=A$.
\end{itemize}
\end{proposition}
\subsection{Approximate identities}\label{ss:CstarPrel.approxid}
 
As we saw, not every $\Cstar$-algebra is unital. On the other hand it is always possible to find a net which behaves like a unit. To be more precise, let $A$ be a $\Cstar$-algebra and suppose that $\Lambda$ is a net. A set $\{a_\lambda\}_{\lambda\in\Lambda}$ is said an \emph{approximate identity} for $A$ if
 \begin{itemize}
 \item $a_\lambda\in A_{\leq 1}^+$ for all $\lambda$,
 \item if $\lambda<\mu$ then $a_\lambda\leq a_\mu$,
 \item for all $a\in A$ we have $\lim_\Lambda \norm{a_\lambda a-a}+\norm{a-aa_\lambda}=0$.
 \end{itemize}
 
Note that in the definition of approximate identity we require $\{a_\lambda\}_{\lambda\in\Lambda}$ to be bounded and increasing. In Banach algebras terminology, such an approximate identity is both a left and a right approximate identity. As hinted, every $\Cstar$-algebra has an approximate identity of this form. This is not the case for certain Banach algebras which fail to have a bounded approximate identity with is both left and right (\cite{Runde.LA}).

If $\Lambda$ can be chosen to be countable, $A$ is said to be $\sigma$-unital. Every separable $\Cstar$-algebra is $\sigma$-unital. We will be interested in particular approximate identities: the ones made of projections, where each $a_\lambda$ is itself a projection, and the \emph{quasicentral} ones. Suppose that $I\subseteq A$ is an ideal and let $\{a_\lambda\}$ be an approximate identity for $I$. Then $\{a_\lambda\}$ is said \emph{quasicentral with respect to }$A$ if whenever $a\in A$ then $\lim_\Lambda\norm{a_\lambda a-aa_\lambda}=0$.
 
 Quasicentral approximate identities always exist:

\begin{theorem}[{\cite[Theorem~2.1]{pedersencorona}}]
Let $A$ be a $\Cstar$-algebra and $I$ be an ideal of $A$. Then $I$ has an approximate identity $\{a_\lambda\}_{\lambda\in\Lambda}$ which is quasicentral w.r.t. $A$. Moreover for every $\{b_\lambda\}_{\lambda\in\Lambda'}$ which is an approximate identity for $I$, $\{a_\lambda\}$ can be found in the convex hull of $\{b_\lambda\}_{\lambda\in\Lambda'}$.
\end{theorem}

\subsubsection{The strict topology and two useful lemmas}\label{ss:CstarPrel.strict}

Suppose that $A\subseteq M$ is a subalgebra. Let $l_a, r_a$ be the following seminorms on $M$ 
\[
l_a(x)=\norm{ax}, \,\, r_a(x)=\norm{xa}.
\]
The topology generated by $l_a,r_a$ is called the $A$-strict topology on $M$. $M$ is said $A$-strictly complete if every bounded $A$-strictly convergent sequence converges in $M$. It is easy to see that if $A\subseteq M$ is an essential ideal, and $M$ is unital, then an approximate identity for $A$ converges $A$-strictly to $1$.

The following is a generalization of some facts contained in \cite{pedersencorona} and of the construction of particularly well-behaved approximate identities as in \cite{Higson.Kasparov}. A similar argument for quotients of $\sigma$-unital algebras was used in \cite{farah2011countable}. We will make heavy use of this lemma for the proof of Theorem~\ref{thm:FactorCtbleSat}.

\begin{proposition}[{\cite[Corollary 6.3]{pedersencorona}}]\label{lem:approxid0EV}
Let $A$ be a $\Cstar$-algebra, $S\in A_1$ and $T\in A_{\leq1}^+$. Then 
\[\norm{[S,T]}=\epsilon\leq \frac{1}{4} \Rightarrow \norm{[S,T^{1/2}]}\leq \frac{5}{4}\sqrt\epsilon.\]
\end{proposition}
\begin{lemma}\label{lem:approxid1EV}
Let $M$ be a unital $\Cstar$-algebra, $A \subseteq M$ an essential ideal, and $\pi \colon M \to M/A$ the quotient map.  
Suppose that there is an increasing sequence $(g_n)_{n\in\en} \subset A$ which $A$-strictly converges to $1$, and that $M$ is $A$-strictly complete.

Let $(F_n)_{n\in\en}$ be an increasing sequence of finite subsets of the unit ball of $M$ and $(\epsilon_n)_{n \in \en}$ be a decreasing sequence converging to $0$, with $\epsilon_0<1/4$.  Then there is an increasing sequence $(e_n)_{n \in \en} \subset A_{\leq 1}^+$ such that, for all $n\in\en$ and $a\in F_n$, the following conditions hold, with $f_n=(e_{n+1}-e_n)^{1/2}$:
\begin{enumerate}[label=(\roman*)]
\item\label{cond0} $\abs{\norm{(1-e_{n-2})a(1-e_{n-2})}-\norm{\pi(a)}}<\epsilon_n$ for all $n\geq2$,
\item\label{cond1} $\norm{[f_n,a]}<\epsilon_n$ for all $n$,
\item\label{cond2a} $\norm{f_n(1-e_{n-2}) - f_n} < \epsilon_n$ for all $n \geq 2$,
\item\label{cond2} $\norm{f_nf_m}<\epsilon_m$ for all $m\geq n+2$,
\item\label{cond3} $\norm{[f_n,f_{n+1}]}<\epsilon_{n+1}$ for all $n$,
\item\label{cond4}  $\norm{f_naf_n}\geq \norm{\pi(a)}-\epsilon_n$ for all $n$,
\item\label{cond5}$\sum_{n \in \en} f_n^2=1$.

Further, whenever $(x_n)_{n \in \en}$ is a bounded sequence from $M$, the following conditions also hold:
\item\label{cond6}the series $\sum_{n \in \en} f_nx_nf_n$ converges to an element of $M$,
\item\label{cond7}
\[\norm{\sum_{n \in \en} f_nx_nf_n}\leq \sup_{n \in \en}\norm{x_n},\]
\item\label{cond8} whenever $\limsup_{n \to \infty}\norm{x_n}=\limsup_{n\to\infty}\norm{x_nf_n^2}$ we have
\[\limsup_{n \to \infty}\norm{x_nf_n^2}\leq\norm{\pi\left(\sum_{n \in \en} x_nf_n^2\right)}.\]
\end{enumerate}
\end{lemma}
\begin{proof}
For each $n \in \en$ let $\delta_n=10^{-100}\epsilon_n^2$, and let $(g_n)_{n \in \en}$ be an increasing sequence in $A$ whose $A$-strict limit is $1$.  We will build a sequence $(e_n)_{n \in \en}$ satisfying the following conditions:
\begin{enumerate}[label=(\arabic*)]
\item\label{conda} $\abs{\norm{(1-e_{n-2})a(1-e_{n-2})}-\norm{\pi(a)}}<\epsilon_n$ for all $n\geq2$ and $a\in F_n$,
\item\label{condb} $0\leq e_0\leq\ldots\leq e_n\leq e_{n+1}\leq\ldots\leq 1$,  and for all $n$ we have $e_n\in A$,
\item\label{condc} $\norm{e_{n}e_k-e_k}<\delta_{n+1}$ for all $n>k$,
\item \label{condd} $\norm{[e_n,a]}<\delta_n$ for all $n \in \en$ and $a\in F_{n+1}$,
\item\label{conde} $\norm{(e_{n+1}-e_n)a}\geq \norm{\pi(a)}-\delta_{n}$ for all $n \in \en$ and $a\in F_{n}$
\item\label{condf} $\norm{(e_{m+1}-e_m)^{1/2}e_n(e_{m+1}-e_m)^{1/2}-(e_{m+1}-e_m)}<\delta_{n+1}$ for all $n>m+1$,
\item\label{condg} $e_{n+1}\geq  g_{n+1}$ for all $n \in \en$.
\end{enumerate}

We claim that such a sequence will satisfy \ref{cond0}--\ref{cond5}, in light of Lemma \ref{lem:approxid0EV}.  Conditions \ref{cond0} and \ref{conda} are identical.  Condition \ref{condd} implies condition \ref{cond1}.  Condition \ref{condc} and the $\Cstar$-identity imply condition \ref{cond2a}, which in turn implies conditions \ref{cond2} and \ref{cond3}.  We have also that conditions \ref{conde} and \ref{condg} imply respectively conditions \ref{cond4} and \ref{cond5}, so the claim is proved.  After the construction we will show that \ref{cond6}--\ref{cond8} also hold.

Take $\Lambda=\{\lambda\in A^+\colon \lambda\leq 1\}$ to be the approximate identity of positive contractions (indexed by itself) and let $\Lambda'$ be a subnet of $\Lambda$ that is quasicentral w.r.t. $M$.

Since $A$ is an essential ideal of $M$, by  \cite[II.6.1.6]{Blackadar.OA} there is a faithful representation $\beta$ on an Hilbert space $H$ such that \[1_{H}=\text{SOT}-\lim_{\lambda\in\Lambda'}\{\beta(\lambda)\}, \] 
Consequently, for every finite $F\subset M$, $\epsilon>0$ and $\lambda\in\Lambda'$ there is $\mu>\lambda$ such that for all $a \in F$,
\[\nu\geq \mu\Rightarrow \norm{(\nu-\lambda)a}\geq \norm{\pi(a)}-\epsilon.\]

We will proceed by induction. Let $e_{-1}=0$ and $\lambda_0\in\Lambda'$ be such that for all $\mu>\lambda_0$ and $a\in F_1$ we have $\norm{[\mu,a]}<\delta_0$. By cofinality of $\Lambda'$ in $\Lambda$ we can find a $e_0\in\Lambda'$ such that $e_0>\lambda_0, g_0$. Find now $\lambda_1>e_0$ such that for all $\mu>\lambda_1$ and $a\in F_2$ we have \[\norm{[\mu,a]}<\delta_1, \, \norm{(\mu-e_0)a}\geq \norm{\pi(a)}-\delta_1.\] Since we have that

\begin{equation}\label{limofnets}\norm{\pi(a)}=\lim_{\lambda\in\Lambda'}\norm{(1-\lambda)a(1-\lambda)}
\end{equation} we can also ensure that for all $a\in F_{3}$ and all $\mu>\lambda_{1}$, condition \ref{cond0} is satisfied.

Picking $e_1\in\Lambda'$ such that $e_1>\lambda_1,g_1$ we have that the base step is completed.

Suppose now that $e_0,\ldots,e_{n}, f_0,\ldots,f_{n-1}$ are constructed. 

We can choose $\lambda_{n+1}$ so that for all $\mu>\lambda_{n+1}$, with $\mu\in \Lambda'$, we have $\norm{[\mu,a]}<\delta_{n+1}/4$ and $\norm{(\mu-e_n)a}\geq \norm{\pi(a)}-\delta_n$  for $a\in F_{n+2}$. Moreover, by the fact that $\Lambda'$ is an approximate identity for $A$ we can have that $\norm{f_m\mu f_m-f_m^2}<\delta_{n+2}$ for every $m<n$ and that $\norm{\mu e_k-e_k}<\delta_{n+2}$ for all $k\leq n$. By equation (\ref{limofnets}) we can also ensure that for all $a\in F_{n+2}$ and all $\mu>\lambda_{n+1}$, condition \ref{cond0} is satisfied.

Once this $\lambda_{n+1}$ is picked we may choose \[e_{n+1}\in \Lambda',\,\,\,e_{n+1}> \lambda_{n+1}, \, g_{n+1},\] to end the induction.  

It is immediate from the construction that the sequence $(e_n)_{n \in \en}$ chosen in this way satisfies conditions \ref{conda} - \ref{condg}.  To complete the proof of the lemma we need to show that conditions \ref{cond6}, \ref{cond7} and \ref{cond8} are satisfied by the sequence $\{f_n\}$.

To prove \ref{cond6}, we may assume without loss of generality that each $x_n$ is a contraction.  Recall that every contraction in $M$ is a linear combination (with complex coefficients of norm $1$) of four positive elements of norm less than $1$, and addition and multiplication by scalar are $A$-strict continuous functions.  It is therefore sufficient to consider a sequence $(x_n)$ of positive contractions.
By positivity of $x_n$, we have that $(\sum_{i\leq n} f_ix_if_i)_{n \in \en}$ is an increasing uniformly bounded sequence, since for every $n$ we have 
\[\sum_{i\leq n} f_ix_if_i\leq\sum_{i\leq n}f_i^2\qquad\text{ and }\qquad f_nx_nf_n\geq 0.\] 
Hence $(\sum_{i\leq n} f_ix_if_i)_{n \in \en}$ converges in $A$-strict topology to an element of $M$  of bounded norm, namely the supremum of the sequence, which is $\sum_{n\in\en} f_nx_nf_n$.

For \ref{cond7}, consider the algebra $\prod_{k\in\en} M$ with the sup norm and the map $\phi_n\colon \prod_{k\in\en} M\to M$ such that $\phi_n((x_i))=f_nx_nf_n$. Each $\phi_n$ is completely positive, and since $f_n^2\leq\sum_{i\in\en} f_i^2=1$, also contractive.  For the same reason the maps $\psi_n\colon\prod_{k\in\en} M\to M$ defined as $\psi_n((x_i))=\sum_{j\leq n}f_jx_jf_j$ are completely positive and contractive. Take $\Psi$ to be the supremum of the maps $\psi_n$.  Then $\Psi((x_n))=\sum_{i\in\en} f_ix_if_i$. This map is a completely positive map of norm $1$, because $\norm{\Psi}=\norm{\Psi(1)}$, and from this condition \ref{cond7} follows.

For \ref{cond8}, we can suppose $\limsup_{i\to\infty}\norm{x_i}=\limsup_{i\to\infty}\norm{x_if_i^2} = 1$. Then for all $\epsilon > 0$ there is a sufficiently large $m \in \en$ and a unit vector $\xi_m \in H$ such that 
\[\norm{x_mf_m^2(\xi_m)}\geq 1-\epsilon.\]
Since $\norm{x_i}\leq 1$ for all $i$, we have that $\norm{f_m(\xi_m)}\geq 1-\epsilon$, that is, $\abs{(f_m^2\xi_m\mid\xi_m)}\geq 1-\epsilon$. In particular we have that $\norm{\xi_m-f_m^2(\xi_m)}\leq\epsilon$. 

Since $\sum f_i^2=1$ we have that $\xi_m$ and $\xi_n$ constructed in this way are almost orthogonal for all $n,m$. In particular, choosing $\epsilon$ small enough at every step, we are able to construct a sequence of unit vectors $\{\xi_m\}$ such that $\abs{(\xi_m\mid\xi_n)}\leq 1/2^m$ for $m>n$. But this means that for any finite projection $P\in M$ only finitely many $\xi_m$ are in the range of $P$ up to $\epsilon$ for every $\epsilon>0$. In particular, if $I$ is the set of all convex combinations of finite projections, we have that that 
\[\lim_{\lambda\in I}\norm{\sum_{i\in\en} x_if_i^2-\lambda\left( \sum_{i\in\en} x_if_i^2\right)}\geq 1.\] 
Since $I$ is an approximate identity for $A$ we have that 
\[\norm{\pi\left(\sum_{i\in\en} x_if_i^2\right)}=\lim_{\lambda\in I}\norm{\sum_{i\in\en} x_if_i^2-\lambda\left(\sum_{i\in\en} x_if_i^2\right)},\] 
as desired.
\end{proof}
We now analyze approximate identities of projections:

\begin{proposition}\label{prop:Prel.Noncentralproj}
Let $A$ be a $\Cstar$-algebra and $p\in A$ be a projection. If $p$ is not central, then there is a positive $a\in A$ with $\norm{ap-pa}\geq\frac{1}{8}$.
\end{proposition}
\begin{proof}
Suppose that $p$ is not central. Let $\rho\colon A\to\mathcal B(H)$ be the GNS representation of $A$. Since $p$ is not central, and $\rho$ is the sum of irreducible representations, there is an irreducible representation $\rho_1\colon A\to\mathcal B(H)$ such that $\rho_1(p)\notin\{0,1\}$. Take $\xi,\eta$ in the Hilbert space generated by $\rho_1(A)$ such that $\langle \rho_1(p)\xi\mid\xi\rangle=0$ and $\langle \rho_1(p)\eta\mid\eta\rangle=1$. By the Kadison Transitivity Theorem (\cite[II.6.1.12]{Blackadar.OA}) there is a contraction $a\in A$ with $\norm{\rho_1(a)\xi-\eta}<1/10$. As $\rho_1$ is contractive, $\norm{ap-pa}\geq\norm{\rho_1(ap-pa)}\geq\frac{1}{2}$. Since every element of $A$ is a combination of 4 positive elements, the thesis follows.
\end{proof}

\begin{lemma}\label{lem:apoxcommuting}
Suppose that $M$ is a unital $\Cstar$-algebra, $A\subseteq M$ is an ideal with an approximate identity of projections $\{p_n\}$ and that $M$ is $A$-strictly complete. Let $\pi\colon M\to M/A$ be the quotient map. Suppose further that, with $q_n=p_{n+1}-p_n$, for every $X\subseteq\NN$ the element 
\[
q_X=\sum_{n\in X}q_n=\lim_m \sum_{n\in X\cap m}q_n
\] is such that $\pi(q_X)\in  Z(M/A)$. Then there is $n_0$ such that for all $n\geq n_0$ we have $q_n\in Z((1-p_{n_0})A(1-p_{n_0}))$.
\end{lemma}

\begin{proof}
Let $B_n=(1-p_n)A(1-p_n)$. By contradiction and Proposition~\ref{prop:Prel.Noncentralproj}, for every $n$ there is an $m>n$ and a positive contraction $a_n\in B_n$ with the property that $\norm{a_nq_m-q_ma_n}\geq\frac{1}{8}$.  Since $p_n$ is an approximate identity of projections for $A$, we can find $n_1>m$ such that $\norm{p_{n_1}a_{n}p_{n_1}-a_n}<\frac{1}{100}$. Let $b_n=p_{n_1}a_np_{n_1}$ and Note that $b_n\in (p_{n_1}-p_n)A(p_{n_1}-p_n)$ and $\norm{b_nq_m-q_mb_n}\geq \frac{1}{16}$.

Construct two sequences of natural numbers $\{n_i\}, \{k_i\}$ such that $1=n_1<k_1<n_2<\cdots$ and there are a contractions $b_i\in (p_{n_{i+1}}-p_{n_i})A(p_{n_{i+1}}-p_{n_i})$ with $\norm{b_iq_{k_i}-q_{k_i}b_i}\geq\frac{1}{16}$. Let $X=\{k_i\}$, and $b=\sum b_i=\lim_m\sum_{n\leq m}b_n$. Since $M$ is $A$-strictly complete, $b\in M\setminus A$. On the other hand, we have that $\pi(b)$ and $\pi(q_X)$ do not commute in $M/A$, a contradiction to $\pi(q_X)\in Z(M/A)$.
\end{proof}
\subsection{The multiplier and the corona}\label{s:multicorona}

Given a nonunital $\Cstar$-algebra $A$ we want to study how $A$ can be embedded in a unital $\Cstar$-algebra $B$. If $A=C_0(X)$, we can embed $A$ as an essential ideal in $C_b(X)\cong C(\beta X)$, where $\beta X$ is the \v{C}ech-Stone compactification of $X$. The space $\beta X$ is, in some sense, the maximal compactification of $X$ and it has the universal property that every continuous map from $X$ to a compact Hausdorff space $Y$ factors uniquely through $\beta X$.
In this spirit, given $A$, we construct a unital $\Cstar$-algebra $\mathcal M(A)$ containing $A$ in a universal way as an essential ideal.
\begin{defin}
Let $A$ be a $\Cstar$-algebra. The algebra $\mathcal M(A)$ is the universal $\Cstar$-algebra containing $A$ as an essential ideal and with the property that whenever $A$ sits inside a $\Cstar$-algebra $B$ as an essential ideal, then there is a unique $^*$-homomorphism $B\to \mathcal M(A)$ which is the identity on $A$.
\end{defin}

From the definition, it is not even clear that $\mathcal M(A)$ exists. On the other hand, if $\mathcal M(A)$ exists, it is unique up to isomorphism. That whenever $A$ is a $\Cstar$-algebra its multiplier algebra $\mathcal M(A)$ exists is nontrivial. The construction of the multiplier algebra can be performed in, at least, three different but equivalent ways: through double centralizers (this was the original way of constructing $\mathcal M(A)$, due to Busby in \cite{Busby}) , through representation theory, or through bimodules. For specific constructions of the multiplier algebra we refer to \cite[II.7.3]{Blackadar.OA}, \cite{Lance.HM}, or the excellent \cite{Skou:NotesMult}. 

\begin{proposition}\label{prop:CStarPrel.examplesofmult}
Let $A$ be a $\Cstar$-algebra.
\begin{itemize}
\item If $A$ is unital, then $A=\mathcal M(A)$;
\item if $A$ is nonunital and separable, $\mathcal M(A)$ is nonseparable;
\item if $A=\mathcal K(H)$, then $\mathcal M(A)=\mathcal B(H)$;
\item $\mathcal M(C_0(X))=C(\beta X)$. Also, if $A$ is unital, $\mathcal M(C_0(X,A))=C(\beta X,A)$;
\item if $A_1,\ldots,A_n,\ldots$ are unital, then $\mathcal M(\bigoplus A_n)=\prod A_n$.
\end{itemize}
\end{proposition}

In case the nonunital algebra $A$ fails to be separable the multiplier doesn't necessarily have to be larger (in terms of density character) than $A$. For example, $\mathcal M(C_0(\omega_1))=C(\omega_1+1)$, as every continuous function $f\colon \omega_1\to \ce$ is eventually constant. An interesting example of simple nonunital algebra for which the multiplier consists with the unitization was constructed by Sakai in \cite{Sakai.Der}. Other interesting, and far from abelian, examples of algebras carrying these type of properties can be found in \cite{GK.NoMult}.

Given a nonunital $\Cstar$-algebra $A$, having constructed its multiplier $\mathcal M(A)$, it is natural to consider the quotient.

\begin{defin}
 The quotient $\mathcal M(A)/A$ is said the \emph{corona algebra} of $A$.
\end{defin}

The corona algebra is the noncommutative analog of the \v{C}ech-Stone remainder of a topological space. If $A$ is $\sigma$-unital, the multiplier algebra is nonseparable, and so the corona is never separable. The most important example of a corona algebra is the Calkin algebra $\mathcal C(H)$ (For a deep analysis of $\mathcal C(H)$, see \cite{BDF.Ext}.) 

In case $A=\bigoplus A_n$, where each $A_n$ is unital, the corona of $A$ is isomorphic to $\prod A_n/\bigoplus A_n$. This algebra is called the \emph{reduced product} of the $A_n$'s. If $B=C_0(X,A)$, for some unital $\Cstar$-algebra $A$, the corona of $B$ is isomorphic to $C_b(X,A)/C_0(X,A)$. 

If $A$ is nonunital and $X$ is a locally compact space, it  is not true anymore that the corona of $C_0(X,A)$ is isomorphic to $C_b(X,\mathcal M(A))/C_0(X,A)$. In fact, the multiplier of $C_0(X,A)$ consists of the set of all norm bounded function from $X$ to $\mathcal M(A)$ which are $A$-strictly continuous. As not every such function extends to an $A$-strictly continuous function from $\beta X$ to $\mathcal M(A)$, we only have that $\mathcal M(C_0(X,A))\supseteq C(\beta X,\mathcal M(A))$. For more information on multipliers, see \cite{APT.Mult}, while to analyze part of the incredible amount of work carried over in trying to understand coronas, see \cite{LinNG.CoronaZ}, \cite{KNgP.CoronaPIS} or \cite{NG.CFP}.

By Proposition~\ref{prop:CStarPrel.examplesofmult}, if $A$ is not unital but separable, $\mathcal M(A)$ is not separable in the norm topology. On the other hand $\mathcal M(A)$ carries a second natural topology: the $A$-\emph{strict topology} as introduced in \ref{ss:CstarPrel.strict}. If $x\in\mathcal M(A)$, a basic strictly open set in $\mathcal M(A)$ is of the form $U_{a,\epsilon}=\{y\mid\norm{a(x-y)}+\norm{(x-y)a}<\epsilon\}$, for some $a\in A$ and $\epsilon>0$.

In case of the multiplier algebra, we will refer as the $A$-strict topology on $\mathcal M(A)$ as, simply, the \emph{strict topology}.

If $A$ is $\sigma$-unital, $\mathcal M(A)$ is separable in the strict topology. In this case, the strict topology is Polish when restricted to any norm bounded subset of $\mathcal M(A)$. If $A$ is unital, the strict topology coincides with the norm topology on $A$. In general, $A$ is strictly dense in $\mathcal M(A)$. In case $A=\mathcal K(H)$ the strict topology coincides with the $\sigma$-strong topology on $\mathcal B(H)$. Another easy description of the strict topology is given when $A=C_0(X)$. In this case, on bounded sets, the strict topology on $C(\beta X)$ coincides with the topology of uniform convergence on compact subsets of $X$.

\subsection{The unique ideal of a $\II_\infty$-factor}\label{sss:CstarPrel.breuer}

A von Neumann algebra is a $\Cstar$-algebra which is weakly-closed as a subalgebra of $\mathcal B(H)$. Alternatively, a $\Cstar$-algebra $M\subseteq \mathcal B(H)$ is a von Neumann algebra if $M=M''$, where $M'$ is the commutant of $M$ in $\mathcal B(H)$. In particular von Neumann algebras are always unital and carry a pletora of projections.  These objects were introduced by the seminal work of Murray and von Neumann as rings of operators (\cite{MVN.I}, \cite{MVN.II} and \cite{MVN.IV}). Although the primary interest of this thesis is on $\Cstar$-algebras, in \S\ref{ss:Breuerideal} we will extend a result of Farah and Hart (see~\cite{farah2011countable}) to quotients of semifinite von Neumann algebras. The basic notions we now introduce can be found in \cite[III.1]{Blackadar.OA}.

If $M$ is a von Neumann algebra and $p\in M$ is a projection then $p$ is said abelian if $pMp$ is, and $p$ is said \emph{finite} if it is not Murray von Neumann equivalent to any of its proper subprojections ($p$ and $q$ are Murray von Neumann equivalent if there is $v$ with $vv^*=p$ and $v^*v=q$). $p$ is said semifinite if every (nonzero) subprojection of $p$ has a (nonzero) finite subprojection.  $M$ is said finite, or semifinite, if $1$ is.

 Particular cases of von Neumann algebras are \emph{factors}, which are von Neumann algebra whose center is trivial. A factor $M$ is said of type $\II$ if $1_M$ is semifinite and $M$ has no abelian projection. If $1_M$ is finite, $M$ is said of type $\II_1$, else $M$ is said of type $\II_\infty$. Factors, and in particular $\II_1$-factors, are crucial objects for the study of von Neumann algebras and their interactions with $\Cstar$-algebras. First, every von Neumann algebra can be seen as a direct integral of factors, which therefore form the building blocks of every von Neumann algebra. To classify, or at least understand, von Neumann algebras is therefore crucial to study factors. Secondly, one should note that if $A$ is a simple separable $\Cstar$-algebra having a (faithful) trace $\tau$, the strong operator closure of the irreducible representation relative to $\tau$ is a $\II_1$-factor. For this particular reason $\II_1$-factors where recognized as key objects in the Elliott classification programme for $\Cstar$-algebras. 

Every von Neumann algebra has a unique predual, see \cite[III.2.4.1]{Blackadar.OA}. A key result is that every $\II_\infty$-factor with separable predual is of the form $M\bar {\otimes}\mathcal B(H)$ for some $\II_1$-factor $M$. (Being in the setting of von Neumann algebras, we need to take the weak closure of the algebraic tensor product to obtain the right tensor product in this category. From this the notation $\bar\otimes$). In particular $\II_\infty$-factors with a separable predual have a unique ($\Cstar$-algebraic) ideal, the one generated by finite projections (\cite[III.1.7.1]{Blackadar.OA}). In case $M=\mathcal R$, the hyperfinite $\II_1$-factor, such ideal is known as the Breuer ideal. Whenever a $\II_\infty$-factor $M$ (with separable predual) is given, and $J$ is its unique ideal, then $\mathcal M(J)=M$. We will study quotients of $\II_\infty$-factors by their unique ideal, and we will see how they resemble properties of coronas of $\sigma$-unital algebras, even though such ideal is never $\sigma$-unital (for this, see, for example \cite{Breu269}, \cite{Breu168}, or \cite{Phil90}).

\subsection{Nuclear $\Cstar$-algebras and the CPAP}\label{ss:nuclear}

In the category of $\Cstar$-algebras an equivalent definition of amenable objects is the one of nuclear $\Cstar$-algebras (see \cite{Runde.LA} for a proof that amenable $\Cstar$-algebras are nuclear, and viceversa). Nuclear $\Cstar$-algebras are fundamental objects for the classification programme of $\Cstar$-algebras. The original definition  is that a $\Cstar$-algebra is nuclear if whenever $B$ is another $\Cstar$-algebra there is a unique way in which the algebraic tensor product $A\odot B$ can be completed to a $\Cstar$-algebra. This definition, even though unnatural, is equivalent to the fact that $A$ is amenable as a Banach algebra. Another equivalent definition of nuclearity is given by the CPAP.

If $A$ and $B$ are $\Cstar$ algebras, a linear $^*$-preserving map $\phi\colon A\to B$ is said \emph{positive} if $\phi(a)\in B^+$ whenever $a\in A^+$. $\phi$ is said \emph{completely positive} if all of its amplifications 
\begin{eqnarray*}
f^{(n)}\colon M_n(A)&\to &M_n(B)\\
(x_{i,j})&\mapsto &(f(x_{i,j}))
\end{eqnarray*}
are positive.

A $\Cstar$-algebra $A$ has the completely positive approximation property (CPAP) if for all finite $G\subseteq A$ and $\e>0$ there are a  matrix algebra $M_n(\ce)$ and completely positive contractions (cpc) $\psi\colon A\to M_n(\ce)$ and $\psi\colon M_n(\ce)\to A$ such that $\norm{\psi(\phi(x))-x}<\e$ for all $x\in G$. The following was proved in \cite{ChoiEff.CPAP} (see also \cite{Kirch.CPAP} for the forward direction).

\begin{theorem}
A $\Cstar$-algebra is nuclear if and only if has the CPAP.
\end{theorem}

In Chapter~\ref{ch:FA}, and in particular in \S\ref{ss:FA.BorelLift} we will use that $A$ has the CPAP as our definition of nuclear $\Cstar$-algebra.
Every abelian $\Cstar$-algebra is nuclear (\cite{Tak.CrNuc}), as well as every finite-dimensional one. It is worth noticing that nuclearity is preserved by extensions (i.e., is $I$ and $A/I$ are nuclear, so is $A$), inductive limits and quotients (this is not an easy result! See ~\cite{ChoiEff.Nuc}). Example of nonnuclear algebras are $\mathcal B(H)$ (unless $H$ is finite-dimensional) and the Calkin algebra. An example of a nonnuclear separable $\Cstar$-algebra comes from the nonamenable group $\mathbb F_2$, by the construction of its reduced group $\Cstar$-algebra (see~\cite[II.10]{Blackadar.OA}).

\subsection{Approximate maps}\label{ss:CstarPrel.Approx}

 Dealing with approximate maps means dealing with maps which ``almost'' carry some regularity property (such a linearity, multiplicativity, positivity, and so on), in a uniform way across the unit ball of a $\Cstar$-algebra.  
\begin{defin}\label{defin:CstarPrel.Approxmaps}
Let $A$ and $B$ be $\Cstar$-algebras and $\epsilon>0$. A map $\phi\colon A\to B$ is said
\begin{enumerate}
\item\label{ss:CstarPrel.Approx1}
 $\epsilon$-\emph{linear} if $\norm{\phi(\lambda x+\mu y)-\lambda\phi(x)-\mu\phi(y)}<\epsilon$ where $x,y\in A_{\leq 1}$, $\lambda,\mu\in\ce_{\leq 1}$;
\item\label{ss:CstarPrel.Approx2} $\epsilon$-$^*$\emph{preserving} if $\norm{\phi(x^*)-\phi(x)^*}<\epsilon$ for all $x\in A_{\leq 1}$;
\item\label{ss:CstarPrel.Approx3} $\epsilon$-\emph{multiplicative} if $\norm{\phi(xy)-\phi(x)\phi(y)}<\epsilon$ whenever $x,y\in A_{\leq 1}$;
\item\label{ss:CstarPrel.Approx4} $\epsilon$-\emph{contractive} if $\sup_{x\in A_{\leq 1}}\norm{\phi(x)}\leq1+\epsilon$;
\item\label{ss:CstarPrel.Approx5} $\epsilon$-\emph{injective} if whenever $x\in A$ with $\norm{x}=1$ then $\norm{\phi(x)}\geq 1-\epsilon$.
\end{enumerate}
\end{defin}
We define an $\epsilon$-$^*$-homomorphism to be a map satisfying \eqref{ss:CstarPrel.Approx1}--\eqref{ss:CstarPrel.Approx4}.

\begin{remark}\label{rmk:CstarPrel.unital-contractive}
To aid our calculations later on, we will often assume that $\norm{\phi} \le 1$.  To obtain stability results as in Chapter~\ref{ch:Ulam} this gives no loss of generality, since if $\phi$ is an $\e$-$^*$-homomorphism as defined above, and $\norm{\phi} > 1$, then $\psi =\frac{1}{\norm{\phi}}\phi$ satisfies $\norm{\phi - \psi} \le \e$.  Similarly, if $A$ is unital and $\e$ is small enough, then we may assume without loss of generality that $\phi(1)$ is a projection.  To see this, note that $\phi(1)$ is an almost-projection and hence (by standard spectral theory tricks) is close to an actual projection $p\in B$.  Then by replacing $\phi(1)$ with $p$, we get a unital $\delta$-$^*$-homomorphism, where $\delta$ is polynomial in $\e$ not depending on either $A$ or $B$.
\end{remark}

 When $A$ is a finite-dimensional $\Cstar$-algebra, and $\e$ is sufficiently small, approximate injectivity is automatic.
\begin{proposition}
Suppose $\e < (\sqrt{10} - 3)^2$, $\ell\in\NN$, $B$ is a $\Cstar$-algebra, and $\phi \colon M_\ell\to B$ is an $\e$-$^*$-homomorphism with $1-\sqrt{\e}\leq\norm{\phi} \le 1$.  Then $\phi$ is $2\sqrt{\e}$-injective.
\end{proposition}
\begin{proof}
The condition $1-2\sqrt{\e}\leq\phi$ ensures that there an element $a$ with $\norm{a}=1$ and $1-2\sqrt{\e}\leq\norm{\phi(a)}$. This condition is needed to be close to a nonzero $^*$-homomorphism.

Note that for any $a\in M_\ell$ with norm at most $1$, we have
\[
  \left|\norm{\phi(a^*a)} - \norm{\phi(a)}^2\right| \le 2\e.
\]
(Here we are using the fact that $\norm{\phi} \le 1$.)  Let $s \colon M_\ell\to M_\ell$ be the map $s(a) = a^* a$.
\begin{claim}
  \label{claimy}
  There is an $n\in\NN$ such that for each $x\in M_\ell$ satisfying $\norm{x} =
  1$ and $\norm{\phi(x)} < 1 - 2\sqrt{\e}$, we have $\norm{\phi(s^{(n)}(x))} <
  2\sqrt{\e}.$
\end{claim}
\begin{proof}
Let $k\in\NN$.  Observe that
\begin{equation}
  \label{eq-induction} (1 - k\sqrt{\e})^2 + 2\e \le 1 - (k+1)\sqrt{\e}
\end{equation}
if and only if
\[
  k^2 - \frac{1}{\sqrt{\e}} k + \left(2 + \frac{1}{\sqrt{\e}}\right) \le 0
\]
if and only if
\[
  \frac{1}{2\sqrt{\e}}(1 - \tau) \le k \le \frac{1}{2\sqrt{\e}}(1 + \tau)
\]
where
\[
  \tau = \sqrt{1 - 4\left(2\e + \sqrt{\e}\right)}.
\]
Note that, since $\e \le 1 / 36$, we have
\[
  \tau^2 = 1 - 4(2\e + \sqrt{\e}) = 1 - 8\e - 4\sqrt{\e} \ge 1 - 8\sqrt{\e} + 16\e = (1 - 4\sqrt{\e})^2
\]
and hence $\tau \ge 1 - 4\sqrt{\e}$.  It follows that the
inequality~\eqref{eq-induction} holds for all positive integers $k$ in the range
\[
  2 \le k \le \frac{1}{\sqrt{\e}} - 2.
\]
Now suppose $x\in M_\ell$ has norm $1$ and satisfies $\norm{\phi(x)} < 1 -
2\sqrt{\e}$.  We claim that for all positive integers $k\le
\frac{1}{\sqrt{\e}} - 2$, we have
\[
  \norm{\phi(s^{(k-1)}(x))} \le 1 - (k+1)\sqrt{\e}
\]
The proof goes by induction on $k \le \frac{1}{\sqrt{\e}} - 2$.  The base case, $k = 1$, is simply our
assumption on $x$; for the induction step we have, for $k+1 \le
\frac{1}{\sqrt{\e}} - 2$,
\[
  \norm{\phi(s^{(k)}(x))} \le \norm{\phi(s^{(k-1)}(x))}^2 + 2\e \le (1 -
  (k+1)\sqrt{\e})^2 + 2\e \le 1 - (k+2)\sqrt{\e}
\]
where we have used inequality~\eqref{eq-induction}.
Finally, let $n$ be the maximal integer less than or equal to $\frac{1}{\sqrt{\e}} - 2$.
Then $n+1 > \frac{1}{\sqrt{\e}}-2$, and so
\[
  \norm{\phi(s^{(n-1)}(x))} \le 1-(n+1)\sqrt\e < 1 - \left(\frac{1}{\sqrt{\e}} -
  2\right)\sqrt{\e} = 2\sqrt{\e},
\]
as required.
\end{proof}
Claim~\eqref{claimy} shows in particular that, if $p\in M_\ell$ is a projection,
then either $\norm{\phi(p)} < 2\sqrt{\e}$, or $\norm{\phi(p)} \ge 1 -
2\sqrt{\e}$.  We will call the projections satisfying the former inequality small.
We will show that, if $\phi$ is not $2\sqrt{\e}$-injective, then every
projection in $M_\ell$ is small.

So suppose that $\phi$ is not $2\sqrt{\e}$-injective.  Since $\norm{\phi} \le
1$, there must be some $x\in M_\ell$ with norm $1$ such that $\norm{\phi(x)} < 1
- 2\sqrt{\e}$.  Replacing $x$ with $s^{(n)}(x)$ as given by Claim~\ref{claimy}, we
may assume that $x$ is a positive element of norm $1$ with $\norm{\phi(x)} <
2\sqrt{\e}$.  Since $1$ belongs to the spectrum of $x$, there is a projection
$p\in M_\ell$ of rank $1$ such that $pxp=p$.  Then, since $\e < 1 / 36$,
\[
  \norm{\phi(p)} = \norm{\phi(pxp)}\le \norm{\phi(p)}^2\norm{\phi(x)}+2\e\le
  2\sqrt{\e} + 2\e < 1 - 2\sqrt{\e}.
\]
and hence $p$ is small.  Now if $q$ is another rank $1$ projection in $M_\ell$, then
there is a unitary $u\in M_\ell$ such that $q = upu^*$, so
\[
  \norm{\phi(q)} \le \norm{\phi(u)}\norm{\phi(p)}\norm{\phi(u^*)} + 2\e <
  2\sqrt{\e} + 2\e,
\]
and so $q$ is small as well.  Finally, if $p_1$ and $p_2$ are small, orthogonal
projections, then
\[
  \norm{\phi(p_1 + p_2)} \le \norm{\phi(p_1) + \phi(p_2)} + \e \le
  \norm{\phi(p_1)} + \norm{\phi(p_2)} + \e < 4\sqrt{\e} + \e  \le 1 - 2\sqrt{\e}
\]
where in the last inequality we have used the fact that $\e < (\sqrt{10} -
3)^2$.  It follows that every projection in $M_\ell$ is small; in
particular $1$ is small.  But then for every $a\in M_\ell$ with norm $1$ we have
\[
  \norm{\phi(a)} \le \norm{\phi(1)}\norm{\phi(a)} + \e \le \norm{\phi(1)} + \e
  \le 2\sqrt{\e} + \e < 1 - 2\sqrt{\e},
\]
a contradiction.
\end{proof}

\subsection{Automorphisms of corona $\Cstar$-algebras: Liftings and trivial automorphisms}\label{ss:autosofcstar}

One of the main concerns of this thesis is to study isomorphisms and embeddings of the form $\phi\colon\mathcal M(A)/A\to\mathcal M(B)/B$, where $A$ and $B$ are nonunital separable $\Cstar$-algebras. 
If $\phi$ is any map $\phi\colon\mathcal M(A)/A\to\mathcal M(B)/B$, a map $\Phi\colon\mathcal M(A)\to\mathcal M(B)$ is a \emph{lift} of $\phi$ if the following diagram commutes
\begin{center}
\begin{tikzpicture}
 \matrix[row sep=1cm,column sep=2cm] {
&\node  (A1) {$\mathcal M(A)$};
& \node (A2) {$\mathcal M(B)$};
\\
&\node  (B1) {$\mathcal M(A)/A$};
& \node (B2) {$\mathcal M(B)/B$};
\\
}; 
\draw (A1) edge[->] node [above] {$\Phi$} (A2) ;
\draw (A1) edge[->]  node[left]{$\pi_1$} (B1) ;
\draw (A2) edge[->]   node[right] {$\pi_2$} (B2) ;
\draw (B1) edge[->] node [above] {$\phi$} (B2) ;
\end{tikzpicture}
\end{center}
where the vertical maps $\pi_1,\pi_2$ are the canonical quotient maps. The existence of a lift (usually not carrying interesting topological properties) is ensured by the Axiom of Choice. If $X\subseteq\mathcal M(A)$ and $\Phi$ is such that for all $x\in X$ we have $\pi_2(\Phi(x))=\phi(\pi_1(x))$ we say that $\Phi$ is a lift of $\phi$ on $X$. A particular case is given when $A=\bigoplus A_n$ for some unital $A_n$'s and $\SI\subseteq\mathcal P(\NN)$. In this case, if $\Phi$ is a lift of $\phi$ on 
\[
\{(x_n)\in\prod A_n\mid \{n\mid x_n\neq 0\}\in\SI\},
\] 
 we abuse of notation and say that $\Phi$ is a lift of $\phi$ on $\SI$.

 \begin{defin}\label{defin:TrivialNonAbel}
Let $A$ and $B$ be $\sigma$-unital $\Cstar$-algebras and $\phi\colon \mathcal M(A)/A\to\mathcal M(B)/B$ be a  $^*$-homomorphism. $\phi$ is said \emph{trivial} if its graph
\[
\Gamma_\phi=\{(a,b)\in\mathcal M(A)_{\leq 1}\times\mathcal M(B)_{\leq 1}\mid \phi(\pi_A(a))=\pi_B(b)\}
\]
 is Borel in the strict topology, $\pi_A, \pi_B$ being the canonical quotient maps from the multiplier onto the corona.
 \end{defin}
 
 If $A$ and $B$ are $\sigma$-unital,  $\mathcal M(A)_{\leq1}$ and $\mathcal M(B)_{\leq1}$ are Polish in the strict topology, so there can be only $\mathfrak c$-many Borel sets. In this case, there are at most $\mathfrak c$-many trivial $^*$-homomorphisms.

The definition of trivial provided above may appear, from a $\Cstar$-algebraic point of view, quite unnatural. An analyst, in fact, may ask for stronger versions of liftings, having not only a nice topological behavior (i.e., being Borel), but respecting also in some sense the algebraic operations.

The most naive generalization may be looking for a $^*$-homomorphism $\Phi\colon\mathcal M(A)\to\mathcal M(B)$ which lifts $\phi$.  Unfortunately, in some cases it is impossible to find such liftings (for example, let $X$ be space consisting of the real line with a circle attached to $-1$ and an interval attached to $-1$, and the automorphism of $C(\beta X\setminus X)$ induced by $t\to -t$). Even in less pathological cases, this is difficult to achieve  (see \S\ref{s:FA.Emb})

A stronger notion of trivial may be stated if one consider only abelian algebras, where automorphisms of coronas correspond to homeomorphisms of \v{C}ech-Stone remainders.

\begin{defin}\label{defin:TrivialAbel}
Let $X$ be  a locally compact noncompact and $\phi\in \Homeo(\beta X \setminus X)$. $\phi$ is said \emph{trivial} if there are compact $K_1,K_2\subseteq X$ and an homeomorphism $f\colon X\setminus K_1\to X\setminus K_2$ such that $\phi= f^*\restriction (\beta X\setminus X)$, where $f^*$ is the unique function extending $f$ to $\beta (X\setminus K_1)=\beta X\setminus K_1$.
\end{defin}

We denote by $\Triv(\beta X\setminus X)$ the set of trivial homeomorphisms. Like for Definition~\ref{defin:TrivialNonAbel}, if $X$ is Polish,  $|\Triv(\beta X\setminus X)|$ has size at most $\mathfrak c$.
Trivial homeomorphisms of \v{C}ech-Stone remainders of topological spaces induce trivial automorphisms of corona algebras, but the converse is not yet proven to be true (for a partial result, see~\cite[Theorem 5.3]{Farah-Shelah.RCQ}). Whenever we will work in the abelian setting we will use this as our definition of trivial.

\subsubsection{The conjectures}\label{sss:CFconj}

We now have the necessary tools to state the conjectures we will work on. They were stated in the current form in~\cite{Coskey-Farah}.

\begin{conjecture}\label{conj:CH}
$\CH$ implies that every corona of a separable nonunital $\Cstar$-algebra has nontrivial automorphisms.
\end{conjecture}
\begin{conjecture}\label{conj:PFA}
Forcing Axioms imply that every corona of a separable $\Cstar$-algebra has only trivial automorphisms.
\end{conjecture}

The following table contains most of the known results regarding solutions of the Conjectures for noncommutative algebras. In what follows $A$ is always assumed to be $\sigma$-unital and nuclear.

Let $(1)$ be the statement ``there are nontrivial automorphisms of $\mathcal M(A)/A$'' and $(2)$ its negation.
\vspace{7pt}

\begin{tabular}{|c|c|c|}\hline
 & $\CH\Rightarrow (1)$ & Forcing Axioms $\Rightarrow (2)$\\\hline\hline
 $A=\mathcal K(H)$&\cite{Phillips-Weaver}&\cite{Farah.AC}\\\hline
 $A$ stable &\cite{Coskey-Farah}&Theorem~\ref{thm:Borel}\\\hline
 $A$ simple&\cite{Coskey-Farah}& Unknown\\\hline
 $A=\bigoplus A_n$, $A_n$ unital UHF&\cite{Coskey-Farah}&\cite{McKenney.UHF}\\\hline
 $A=\bigoplus A_n$, $A_n$ unital &\cite{Coskey-Farah} & Theorem~\ref{thm:Borel}\\
 \hline
 $A=C_0(X,B)$, $X$ manifold, $B$ unital & Theorem~\ref{thm:CHMani}&Unknown\\\hline 
\end{tabular}

\vspace{7pt}

In the results in \cite{Coskey-Farah}, unlike Theorem~\ref{thm:Borel}, the assumption of nuclearity is not needed. More instances of the conjecture have been proved: in case $A=\bigoplus M_{n(k)}$, Ghasemi (\cite{Ghasemi.FDD}) proved that is possible to force that every automorphism of $\prod M_{k(n)}/\bigoplus M_{k(n)}$ is trivial. Also, the results proved in \cite{Coskey-Farah} from $\CH$ go slightly beyond the simple or the stable case.

Regarding the abelian case, the following table shows what is known regarding trivial homeomorphisms (in the sense of Definition~\ref{defin:TrivialAbel}). In what follows $X$ represents a Polish locally compact noncompact space. Let $(1)$ be $\Homeo(\beta X\setminus X)\neq\Triv(\beta X\setminus X)$ and $(2)$ its negation.

\vspace{7pt}

\begin{tabular}{|c|c|c|}\hline
 & $\CH\Rightarrow(1) $ & Forcing Axioms $\Rightarrow(2)$\\\hline\hline

 $X=\NN$&\cite{Rudin}&\cite{Shelah.PF}, \cite{Shelah-Steprans.PFAA}, \cite{Velickovic.OCAA}\\\hline

 $X$ is $0$-dimensional&Parovicenko's Theorem&\cite{Farah-McKenney.ZD}\\\hline
  $X=\bigsqcup X_i$, $X_i$ clopen&\cite{Coskey-Farah}&partially Theorem~\ref{thm:Borel}\\\hline
 $X=[0,1)$, $X=\er$ &Yu, see \cite{KP.STCC}& Unknown, but see \cite[Theorem 5.3]{Farah-Shelah.RCQ}\\\hline
 $X$ manifold &Theorem~\ref{thm:CHMani} &Unknown, but see \cite[Theorem 5.3]{Farah-Shelah.RCQ}\\\hline
 $X$ as in \S\ref{ss:arigidspace}&Unknown&Unknown, but see \cite[Theorem 5.3]{Farah-Shelah.RCQ}\\
 \hline
\end{tabular}

\vspace{7pt}

Adding to this list, Farah and Shelah proved in \cite{Farah-Shelah.RCQ} that if $X$ can be written as an increasing union of compact spaces $X_i$, with $\sup |\delta X_i|<\infty$ , $\delta X_i$ being the topological boundary of $X_i$, then the algebra $C(\beta X\setminus X)$ is countably saturated and, consequently (see \S\ref{ss:CH.consofCH}), $\CH$ implies that $\beta X\setminus X$ has nontrivial homeomorphisms. Also, Theorem~\ref{thm:CHMani} goes slightly beyond the case that $X$ is a manifold.
When proving that there are nontrivial homemorphisms of $\beta X\setminus X$, it is actually proved that the amount of homeomorphisms exceeds $\mathfrak c$, and therefore are constructed automorphisms of $C(\beta X\setminus X)$ which are nontrivial according to Definition~\ref{defin:TrivialNonAbel}.

\chapter{Saturation and $\CH$}\label{ch:CH}
\section{Model theory for $\Cstar$-algebras}\label{s:CH.MTforcstar}
Model Theory for continuous structures is a newly developed and rising topic with exciting prospects in its applications to Operator Algebras. We will use a fragment of the theory to provide some embedding results for $\Cstar$-algebras, and we will deeply analyze the model theoretical concept of saturation.

We will sketch here an introduction to model theory for $\Cstar$-algebras. We will be considering $\Cstar$-algebras as structures for the continuous logic formalism of \cite{BYBHU} (or, for the more specific case of operator algebras, \cite{FHS.II}).  Nevertheless, for many of our results it is not necessary to be familiar with that logic.  Informally, a \emph{formula} is an expression obtained from a finite set of norms of $*$-polynomials with complex coefficients by applying continuous functions and taking suprema and infima over some of the variables. Quantifications are allowed only on bounded sets, or, more specifically, on definable sets (see \cite[\S 3]{bourbaki}). A formula is \emph{quantifier-free} if it does not involve suprema or infima. The following is the precise definition:

Let $P(\bar x)$ be a $^*$-polynomial with complex coefficient in a finite tuple $\bar x=x_1,\ldots,x_n$. Then
\begin{enumerate}
\item $\norm{P(\bar x)}$ is a formula;
\item if $f\colon \er^n\to\er$ is a continuous function and $\phi_1(\bar x),\ldots,\phi(\bar x)$ are formulas, so is $f(\phi_1(\bar x), \ldots,\phi_n(\bar x))$;
\item if $\phi(\bar x,y)$ is a formula and $n\in\NN$, then $\inf_{\norm{y}\leq n} \phi(\bar x,y)$ and $\sup_{\norm{y}\leq n}\phi(\bar x,y)$ are formulas.
\end{enumerate}
If a formula is constructed only using clauses 1. and 2. is said \emph{quantifier-free}. We will make use of formulas with parameters. For this, let $A$ be a $\Cstar$-algebra and $B\subseteq A$. If $\phi$ is constructed by clauses 1.--3. by allowing the polynomials in 1. to have also coefficients in $B$, we say that $\phi$ is a $B$-formula.

Particularly interesting cases of formulas are \emph{sentences}, which are formulas without free variables (i.e., every variable is in the scope of a supremum or an infimum). Sentences can be evaluated in $\Cstar$-algebra:  if $A$ is a $\Cstar$-algebra and $\phi$ is a sentence then $\phi^A$ is a unique real number obtained by evaluating $\phi$ in the algebra $A$. The set of all sentences evaluating $0$ in $A$ is said the \emph{theory} of $A$, and denoted by $\Th(A)$. If $A$ and $B$ are $\Cstar$-algebras and $\Th(A)=\Th(B)$ then $A$ and $B$ are said to be \emph{elementary equivalent}, denoted by $A\equiv B$. The following is a combination of  \L os' Theorem and the Keisler-Shelah's Theorem in the continuous setting:
\begin{theorem}
Let $\mathcal U$ be a free ultrafilter on a cardinal $\kappa$. Then $A\equiv A^{\mathcal{U}}$. On the other hand, suppose that $A\equiv B$, then there are ultrafilters $\mathcal U$ and $\mathcal V$ such that $A^\mathcal U\cong B^\mathcal V$. 

If $\CH$ holds and $A$ and $B$ are separable then $A^\mathcal U\cong B^\mathcal V$ for every $\mathcal U,\mathcal V$ free ultrafilters on $\NN$.
\end{theorem}

\subsection{Three layers of saturation}\label{ss:sat}
The concept of saturation is key in analyzing the structure of ultraproducts and ultrapowers, as well as certain corona algebras. A condition very similar to the countable saturation of ultraproducts was considered by Kirchberg under the name ``$\epsilon$-test" in \cite{Kirch.CentralSeq} (see also \cite[Lemma 3.1]{KirchbergRordam}).                                                                                                                                                                                                                                                                                                                                  Morally, a $\Cstar$-algebra $A$ is countably saturated if every countable set of conditions that can be approximately satisfied in a given closed ball can be satisfied precisely in the same closed ball.

If $\phi(\bar x)$ is a formula with free variables $x_1,\ldots,x_n$ and $r\in\er$, we call $\phi(\bar x)=r$ a \emph{condition}. If $\bar a\in A^n$ and $\phi(\bar a)=r$ we say that $\bar a$ \emph{satisfies} the condition $\phi=r$. If $B\subseteq A$ and $\phi$ is a $B$-formula (i.e., a formula with parameters in $B$), then $\phi=r$ is said a $B$-condition.

\begin{notation}
For a compact set $K\subseteq\er$ and $\epsilon>0$, we denote the $\epsilon$-thickening of $K$ by $(K)_\epsilon=\{x\in\er \colon d(x,K)<\epsilon\}$.
\end{notation}

\begin{defin}\label{defin:Sat.Sat}
Let $A$ be a $\Cstar$-algebra, and $\Phi$ be a collection of formulas.  We say that $A$ is \emph{countably $\Phi$ saturated} if for every sequence $(\phi_n)_{n \in \en}$ of formulas from $\Phi$ with parameters from $A_{\leq 1}$, and sequence $(K_n)_{n \in \en}$ of compact sets, the following are equivalent:
\begin{enumerate}
\item[(1)] There is a sequence $(b_k)_{k \in \mathbb{N}}$ of elements of $A_{\leq 1}$ such that $\phi_n^A(\overline{b}) \in K_n$ for all $n \in \en$,
\item[(2)] For every $\epsilon > 0$ and every finite $F \subset \en$ there is $(b_k)_{k \in \mathbb{N}} \subseteq A_{\leq 1}$, depending on $\epsilon$ and $F$, such that $\phi_n^A(\overline{b}) \in (K_n)_{\epsilon}$ for all $n \in F$.
\end{enumerate}
The three most important special cases for us will be the following:
\begin{itemize}
\item{
If $\Phi$ contains all degree-$1$ $^*$-polynomials, we say that $A$ is \emph{countably degree-$1$ saturated}.
}
\item{
If $\Phi$ contains all quantifier free formulas, we say that $A$ is \emph{countably quantifier-free saturated}.}
\item{
If $\Phi$ is the set of all formulas we say that the algebra $A$ is \emph{countably saturated}.
}
\end{itemize}
\end{defin}

Clearly condition (1) in the definition always implies condition (2), but the converse does not always hold.  We recall the (standard) terminology for the various parts of the above definition.  A set of conditions satisfying (2) in the definition is called a \emph{type}; we say that the conditions are \emph{approximately finitely satisfiable} or \emph{consistent}.  When condition (1) holds, we say that the type is \emph{realized} (or \emph{satisfied}) by $(b_k)_{k \in \mathbb{N}}$. As we said, for an algebra, to be countably-$\Phi$ saturated means that every countable set of $A_{\leq1}$-conditions whose formulas are taken from $\Phi$ which is approximately satisfiable  is actually satisfiable.

An equivalent definition of quantifier-free saturation is obtained by allowing only $^*$-polynomials of degree at most $2$, see \cite[Lemma 1.2]{farah2011countable}.  By (model-theoretic) compactness the concepts defined by Definition~\ref{defin:Sat.Sat} are unchanged if each compact set $K_n$ is assumed to be a singleton.  

In the setting of logic for $\Cstar$-algebras, the analogue of a finite discrete structure is a $\Cstar$-algebra with compact unit ball, that is, a finite-dimensional algebra, and if $A$ is such, then $A\cong A^{\mathcal U}$ for every choice of $\mathcal U$ (see \cite[p. 24]{BYBHU}).

\begin{fact}[{\cite[Proposition 7.8]{BYBHU}}]
Every ultraproduct of $\Cstar$-algebras over a countably incomplete ultrafilter is countably saturated.  In particular, every finite-dimensional $\Cstar$-algebra is countably saturated.
\end{fact}

The interest in countably degree-$1$ saturated algebras started with the work of Farah and Hart in \cite{farah2011countable}, where they proved the following:

\begin{theorem}\label{thm:cd1sFH}
Let $A$ be a $\sigma$-unital $\Cstar$-algebra. Then the corona of $A$ is countably degree-1 saturated.
\end{theorem}
In the same paper, they showed that the Calkin algebra is not countably quantifier-free saturated, by an argument of Phillips reproduced in~\cite[Proposition~4.2]{farah2011countable}. If one wants only to show that the Calkin algebra is not countably saturated one can appeal to an easier argument.

\begin{proposition}[{\cite[Proposition 4.1]{farah2011countable}}]
The Calkin algebra is not countably saturated.
\end{proposition} 
\begin{proof}
For a unital $\Cstar$-algebra let $\mathcal U$ be the set of unitaries in $A$. This is a definable set, and therefore we can quantify over it by \cite[\S3]{bourbaki}.

Let $\phi_n(x)=\inf_{u\in\mathcal U}\norm{u^{n}-x}$ and $\psi(x)=\norm{xx^*-1}+\norm{x^*x-1}$. If $x$ is in the Calkin algebra then $\psi(x)+\phi_n(x)=0$ if and only $x$ is a unitary whose Fredholm index is divided by  $n$. Also, if $\psi(x)=0$ then $\phi_n(x)\geq 1$ if and only if its Fredholm index is not divided by $n$. Therefore for every finite set $F,G\subseteq \NN$ there is an $x$ for which $\psi(x)=0$, $\phi_{2^n}(x)=0$ if $n\in F$ and $\phi_{3^n}(x)\in[1,2]$ if $n\in G$, but it is impossible to find an $x$ for which $\psi(x)=\phi_{2^n}(x)=0$ and $\psi_{3^n}(x)\in[1,2]$ for all $n\in\NN$.
\end{proof}

Interesting cases of countably saturated algebras are given by reduced product, as shown in the following Theorem, which we will use in~\S\ref{ss:CH.consofCH}.
\begin{theorem}[{\cite[Theorem 2.7]{Farah-Shelah.RCQ}}]\label{thm:CH.RedProdSat}
Let $A_n$ be unital and separable $\Cstar$-algebras. Then $\prod A_n/\bigoplus A_n$ is countably saturated.
\end{theorem}

Other cases of countably degree-$1$ saturated metric structures were treated in \cite{Voiculescu.c1ds}, where it was analyzed countable degree-$1$ saturation of certain $\Cstar$-algebras which appear as corona of Banach algebras.

In \S\ref{ss:Breuerideal} and \ref{ss:Satforabelian} we will enlarge the class of algebra which are, at some level, saturated. Before doing so, we state and prove an easy lemma, allowing us to restrict the class of saturated algebras.

\begin{defin}\label{def:sccc}
A $\Cstar$-algebra $A$ has \emph{few orthogonal positive elements} if every family of pairwise orthogonal positive elements of $A$ of norm $1$ is countable.
\end{defin}

\begin{remark}This condition was introduced recently in \cite{Masumoto} under the name \emph{strong countable chain condition}, where it was expressed in terms of the cardinality of a family of pairwise orthogonal hereditary ${}^*$-subalgebras. On the other hand, in general topology this name was already introduced by Hausdorff in a different context, so we have given this property a new name to avoid overlaps.  In the non-abelian case, it is not known whether or not this condition coincides with the notion of countable chain condition for any partial order.
\end{remark}

\begin{lemma}\label{lem:CHSat.scccbad}
If an infinite dimensional $\Cstar$-algebra $A$ has few orthogonal positive elements, then $A$ is not countably degree-$1$ saturated.
\end{lemma}
\begin{proof}
Suppose to the contrary that $A$ has few orthogonal positive elements and is countably degree-$1$ saturated.  Using Zorn's lemma, find a set $Z \subseteq A_1^+$ which is maximal (under inclusion) with respect to the property that if $x, y \in Z$ and $x \neq y$, then $xy = 0$.  By hypothesis, the set $Z$ is countable; list it as $Z = \{a_n\}_{n \in \mathbb{N}}$.

For each $n \in \mathbb{N}$, define $P_n(x) = a_n xx^*$, and let $K_n = \{0\}$.  Let $P_{-1}(x) = x$, and $K_{-1} = \{1\}$.  The type $\{\norm{P_n(x)} \in K_n \colon  n \geq -1\}$ is finitely satisfiable.  Indeed, by definition of $Z$, for any $m \in \mathbb{N}$ and any $0 \leq n \leq m$ we have $\norm{P_n(a_{m+1})} = \norm{a_na_{m+1}} = 0$, and $\norm{a_{m+1}} = 1$.  By countable degree-$1$ saturation there is a positive element $b=aa^* \in A_1^+$ such that $\norm{P_n(a)} = 0$ for all $n \in \mathbb{N}$.  This contradicts the maximality of $Z$.
\end{proof}

Subalgebras of $\mc{B}(H)$ clearly have few positive orthogonal elements, whenever $H$ is separable.  As a result, we obtain the following.

\begin{corollary}\label{cor:Hnotseparable}
No infinite dimensional subalgebra of $\mc{B}(H)$, with $H$ separable, can be countably degree-$1$ saturated.
\end{corollary}

Corollary \ref{cor:Hnotseparable} shows that many familiar $\Cstar$-algebras fail to be countably degree-$1$ saturated.  In particular, it implies that no infinite dimensional separable $\Cstar$-algebra is countable degree-$1$ saturated.  It also shows that the class of countably degree-$1$ saturated algebras is not closed under taking inductive limits (consider, for example, the CAR algebra $\bigotimes_{i=1}^{\infty}M_2(\mathbb C)$, or any AF algebra) or subalgebras. In fact, it not very difficult to see that if $A$ is countably degree-1 saturated, then $A$ has density character at least $\mathfrak c$.

\subsection{The consequences of $\CH$}\label{ss:CH.consofCH}
In this section we explore the relations between $\CH$ and countably saturated algebras.

\begin{theorem}\label{thm:CHandSat}
Let $C$ be a countably saturated $\Cstar$-algebra. Then
\begin{itemize}
\item If $A$ is a $\Cstar$-algebra of density character $\aleph_1$ that embeds in an ultrapower of $C$, then $A$ embeds into $C$.
\end{itemize}
Also, if $\CH$ holds and $C$ has density character $\mathfrak c$, then
\begin{itemize} 
\item if $D$ a countably saturated $\Cstar$-algebra of density character $\mathfrak c$, and $\Th(C)=\Th(D)$, then $C\cong D$;
\item $C$ has $2^{\mathfrak c}$-many automorphisms.
\end{itemize}
\end{theorem}
\begin{proof}
The proof of the first statement can be adopted from the discrete case, see \cite[Theorem~5.1.14]{ChangKeisler} or \cite[Theorem~10.1.6]{Hodges.MT}. For the second and third statement, see \cite[\S 4.4]{FHS.II} and \cite[Theorem~3.1]{Farah-Shelah.RCQ}.
\end{proof}

A formula $\phi(\bar x)$ is said $\er^+$-valued if for every $\Cstar$-algebra $A$ and $\bar a\in A^n$, $n$ being the arity of $\bar x$, then $\phi(\bar a)\in\er^+$. $\phi(\bar x,\bar y)$ is a $\sup$-formula if it is of the form $\phi(\bar x,\bar y)=\sup_{\bar x}\psi(\bar x,\bar y)$ where $\psi$ is a quantifier-free $\er^+$-valued formula. Equivalently it is possible to define $\sup$-sentences. If $T$ is a theory of $\Cstar$-algebras then $T_\forall=\{\phi\in T\mid \phi\text{ is a }\sup\text{-sentence}\}$ is known as the universal part of $T$.
\begin{defin}
A class of $\Cstar$-algebras is said universally axiomatizable if there is a set $S$ of $\sup$-sentences such that $A\in\mathcal C$ if and only if $S\subseteq \Th(A)_\forall$.
\end{defin}
The following is Proposition~2.4.4 in \cite{bourbaki}
\begin{proposition}
A class of $\Cstar$-algebra is universally axiomatizable if and only if it is closed under ultraproducts, ultrapowers and substructures.
\end{proposition}

\begin{lemma}\label{lem:CHemb}
Let $A$ and $B$ be $\Cstar$-algebras. Then $\Th(B)_\forall\subseteq\Th(A)_\forall$ if and only if $A$ embeds into some ultrapower of $B$.
\end{lemma}
We will apply this to the class of MF algebras.
\begin{defin}
A separable $\Cstar$-algebra is MF if it embeds into $\prod M_n/\bigoplus M_n$.
\end{defin}
It is clear that the class of MF algebras is universally axiomatizable when restricted to its separable models. In fact, for a separable $A$, $A$ is MF if and only if $\Th(\prod M_n/\bigoplus M_n)_\forall\subseteq\Th(A)_\forall$. The class of MF algebras includes all AF algebras (and much more, in fact, see \cite[V.4.3.5]{Blackadar.OA})

\begin{theorem}\label{thm:CH.EmbMFCH}
Assume $\CH$ and let $A_n$ be (unital) MF separable algebras. Then $\prod A_n/\bigoplus_{\SI} A_n$ embeds (unitally) into $\prod M_{n}/\bigoplus M_{n}$ for every ideal $\SI\subseteq\mathcal P(\NN)$.
\end{theorem}
\begin{proof}
Let $A=\prod A_n/\bigoplus_{\SI} A_n$.
That
\[
\bigcap \Th(A_n)_\forall\subseteq\Th(A)
\]
was proved by Ghasemi in \cite{Ghasemi.FFV} (or, simply, note that if $\phi$ is a $\forall$-sentence then $\phi^A=\lim_{n\in\SI}\phi^{A_n}$). Since for all $n$  we have that $A_n$ is MF, then
\[
\Th(\prod M_n/\bigoplus M_n)_\forall\subseteq\Th(A_n)_\forall.
\]
Since $\prod M_n/\bigoplus M_n$ is countably saturated (see Theorem~\ref{thm:CH.RedProdSat}) and we have $\CH$, the thesis follows from Lemma~\ref{lem:CHemb} and Theorem~\ref{thm:CHandSat}
\end{proof}

It is now known if there is an algebra which is only countably degree-1 saturated but fails to have $2^\mathfrak c$-many automorphisms in presence of $\CH$. Countable saturation of the algebra $C(\beta[0,1)\setminus[0,1))$ was used in \cite{Farah-Shelah.RCQ} to prove that under $\CH$ the space $\beta[0,1)\setminus[0,1)$ has nontrivial homeomorphisms, a result that was previously proved, with different methods, by Yu (see \cite{KP.STCC}).

\section{Instances of saturation}\label{s:CH.instances}
In this section we prove that certain algebras carry some degree of saturation. We will focus first on quotients of $\II_\infty$-factors, and later on abelian $\Cstar$-algebras.

\subsection{Saturation of quotients of factors}\label{ss:Breuerideal}

The purpose of what follows is to extend Theorem 1.4 in \cite{farah2011countable} and expanding the class of quotients which are countably degree-$1$ saturated. Following the motivating example of the unique ideal of a $\II_\infty$-factor with separable predual, we provide a positive results for such quotient. It should be noted that such ideals are not $\sigma$-unital (see \S\ref{sss:CstarPrel.breuer}), and therefore Theorem~\ref{thm:cd1sFH} doesn't apply here. The proof of the following relies heavily on Lemma~\ref{lem:approxid1EV}.

\begin{theorem}\label{thm:FactorCtbleSat}
Let $M$ be a unital $\Cstar$-algebra, and let $A \subseteq M$ be an essential ideal.  Suppose that there is an increasing sequence $(g_n)_{n\in\en} \subset A$ of positive elements whose supremum is $1_M$, and suppose that any increasing uniformly bounded sequence converges in $M$.  Then $M/A$ is countably degree-$1$ saturated.
\end{theorem}
\begin{proof}
Let $\pi \colon M \to M/A$ be the quotient map.  Let $(P_n(\overline x))_{n \in \en}$ be a collection of $^*$-polynomial of degree $1$ with coefficients in $M/A$, and for each $n \in \en$ let $r_n\in\er^+$.  Without loss of generality, reordering the polynomials and eventually adding redundancy if necessary, we can suppose that the only variables occurring in $P_n$ are $x_0, \ldots, x_n$.

Suppose that the set of conditions $\{\norm{P_n(x_0,\ldots,x_n)} = r_n \colon n \in \en\}$ is approximately finitely satisfiable, in the sense of Definition \ref{defin:Sat.Sat}.  As we noted immediately after Definition \ref{defin:Sat.Sat}, it is sufficient to assume that the partial solutions are all in $(M/A)_{\leq 1}$, and we must find a total solution also in $(M/A)_{\leq 1}$.  So we have partial solutions
\[\{\pi(x_{k,i})\}_{k\leq i}\subseteq (M/A)_{\leq1}\]
such that for all $i\in\en$ and $n\leq i$ we have
\[\norm{P_n(\pi(x_{0,i}),\ldots,\pi(x_{n,i}))}\in (r_n)_{1/i}.\] 

For each $n \in \en$, let $Q_n(x_0,\ldots,x_n)$ be a polynomial whose coefficients are liftings of the coefficients of $P_n$ to $M$, and let $F_n$ be a finite set that contains
\begin{itemize}
\item all the coefficients of $Q_k$, for $k\leq n$
\item $x_{k,i}, x_{k,i}^*$ for $k\leq i\leq n$.
\item $Q_k(x_{0,i},\ldots,x_{k,i})$ for $k\leq i\leq n$.
\end{itemize}
Let $\epsilon_n = 4^{-n}$.  Find sequences $(e_n)_{n \in \en}$ and $(f_n)_{n \in \en}$ satisfying the conclusion of Lemma ~\ref{lem:approxid1EV} for these choices of $(F_n)_{n \in \en}$ and $(\epsilon_n)_{n \in \en}$.

Let $\overline x_{n,i}=(x_{0,i},\ldots,x_{n,i})$, $y_k=\sum_{i\geq k} f_ix_{k,i}f_i$, $\overline y_n=(y_0,\ldots,y_n)$ and $\overline z_n=\pi(\overline y_n)$.  Fix $n \in \en$; we will prove that $\norm{P_n(\overline z_n)}=r_n$.

First, since $x_{k,i}\in M_{\leq 1}$, as a consequence of condition \ref{cond7} of Lemma ~\ref{lem:approxid1EV}, we have that $y_i\in M_{\leq 1}$   for all $i$. 
 Moreover, since $Q_n$ is a polynomial whose coefficients are lifting of those of $P_n$ we have 
\[\norm{P_n(\overline z_n)}=\norm{\pi(Q_n(\overline y_n))}.\]
We claim that 
\[Q_n(\overline y_n)-\sum_{j \in \en} f_jQ_n(\overline x_{n,j})f_j\in A.\] 
It is enough to show that 
\[\sum_{j \in \en} f_jax_{k,j}bf_j-\sum_{j \in \en} af_jx_{k,j}f_jb\in A,\] 
where $a,b$ are coefficients of a monomial in $Q_n$,  since $Q_n$ is the sum of finitely many of these elements (and the proof for monomials of the form $ax_{k,j}^*b$ is essentially the same as the one for $ax_{k, j}b$). 

By construction we have $a,b\in F_n$, and hence by condition \ref{cond1} of Lemma ~\ref{lem:approxid1EV}, for $j$ sufficiently large, 
\[\forall x\in M_{\leq 1}\,(\norm{af_jxf_jb-f_jaxbf_j}\leq 2^{-j}(\norm{a}+\norm{b})).\] 
Therefore $\sum_{j \in \en} (f_jax_{k,j}bf_j-af_jx_{k,j}f_jb)$ is a series of elements in $A$ that is converging in norm, which implies that the claim is satisfied. In particular, 
\[\norm{P_n(\overline z_n)}=\norm{\pi\left(\sum_{j \in \en} f_jQ_n(\overline x_{n,j})f_j\right)}.\]
For each $j \geq 2$, let $a_{j}=(1-e_{j-2})Q_n(\overline x_{n,j})(1-e_{j-2})$. By condition \ref{cond0} of Lemma \ref{lem:approxid1EV}, the fact that $Q_n(\overline x_{n,j})\in F_n$, and the original choice of the $x_{n, j}$'s, we have that $\limsup \norm{a_j}=r_n$.
Similarly to the above, but this time using condition \ref{cond2a} of Lemma \ref{lem:approxid1EV}, we have
\[\norm{\pi\left(\sum_{j \in \en} f_jQ_n(\overline x_{n,j})f_j\right)}=\norm{\pi\left(\sum_{j \in \en} f_ja_jf_j\right)}\leq \norm{ \sum_{j \in \en} f_ja_jf_j}.\]

Using condition \ref{cond7} of Lemma ~\ref{lem:approxid1EV} and the fact that $Q_n(\overline x_{n,j})\in F_j$ we have that \[\norm{\sum_{j \in \en} f_ja_jf_j}\leq \limsup_{j\to\infty} \norm{a_j}=r_n.\]
Combining the calculations so far, we have shown
\[\norm{P_n(\overline z_n)} = \norm{\pi\left(\sum_{j\in\en}f_jQ_n(\overline{x}_{n, j})f_j\right)} = \norm{\pi\left(\sum_{j \in \en}f_ja_jf_j\right)} \leq r_n.\]
Since $Q_n(\overline{x}_{n, j}) \in F_j$ for all $j$, condition \ref{cond4} of Lemma \ref{lem:approxid1EV} implies
\[r_n \leq \limsup_{j\to\infty}\norm{f_jQ_n(\overline{x}_{n,j})f_j}.\]
It now remains to prove that
\[\limsup_{j\to\infty} \norm{f_ja_jf_j}\leq\norm{\pi\left(\sum_{j\in\en} f_ja_jf_j\right)}\]
so that we will have
\begin{align*}
r_n &\leq \limsup_{j\to\infty} \norm{ f_jQ_n(\overline x_{n,j})f_j} \\
 &=\limsup_{j\to\infty} \norm{f_ja_jf_j}\\
 &\leq \norm{\pi\left(\sum_{j\in\en} f_ja_jf_j\right)} \\
 &= \norm{P_n(\overline{z_n})}.
\end{align*}
We have $Q_n(\overline x_{n,j})\in F_j$, so by condition \ref{cond1} of Lemma \ref{lem:approxid1EV}, we have that 
\[\limsup_{j\to\infty} \norm{f_ja_jf_j}=\limsup_{j\to\infty} \norm{a_jf_j^2},\]
and hence
\[\sum_{j\in\en} f_ja_jf_j-\sum_{j\in\en} a_jf_j^2\in A.\]
The final required claim will then follow by condition \ref{cond8} of Lemma \ref{lem:approxid1EV}, once we verify \[\limsup_{j\to\infty}\norm{a_jf_j^2} = \limsup_{j\to\infty}\norm{a_j}.\]
We clearly have that for all $j$,
\[\norm{a_jf_j^2} \leq \norm{a_j}.\]
On the other hand, 
\begin{align*}
\limsup_{j \to \infty}\norm{a_jf_j^2} &= \limsup_{j\to\infty}\norm{f_ja_jf_j} \\
 &= \limsup_{j\to\infty}\norm{f_jQ_n(\overline{x}_{n, j})f_j} &\text{by condition \ref{cond2a}} \\
 &\geq r_n \\
 &= \limsup_{j\to\infty} \norm{a_j}.
\end{align*}\qedhere
\end{proof}
This proof followed the same strategy as \cite[Theorem 1.4]{farah2011countable}, fixing a small technical error that one can find there. Specifically, our proof  avoids their equation (10) on p. 14, which is incorrect.

An immediate corollary is the following: 
\begin{corollary}
Let $N$ be a $\II_1$ factor, $H$ a separable Hilbert space and $M=N\,\overline\otimes\,\mathcal B(H)$ be the associated $\II_\infty$ factor and $\mathcal K_M$ be its unique two-sided closed ideal.  Then $M/\mc{K}_M$ is countably degree-$1$ saturated.  In particular, this is the case when $N = \mc{R}$, the hyperfinite $\II_{1}$ factor.
\end{corollary}
\noindent

\begin{corollary}\label{cor:vnAsigmafinite}
Let $M$ be a von Neumann algebra with a  $\sigma$-finite infinite trace, and let $A$ be the ideal generated by the finite trace projections. Then $M/A$ is countably degree-$1$ saturated.
\end{corollary}

\subsection{Saturation of abelian algebras}\label{ss:Satforabelian}

Here we consider abelian $\Cstar$-algebras, and in particular the saturation properties of real rank zero abelian $\Cstar$-algebras. We show how saturation of abelian $\Cstar$-algebras is related to the classical notion of saturation for Boolean algebras.  We begin by recalling some well-known definitions and properties.

A topological space $X$ is said \emph{sub-Stonean} if any pair of disjoint open $\sigma$-compact sets has disjoint closures; if, in addition, those closures are open and compact, $X$ is said \emph{Rickart}. A space $X$ is said to be \emph{totally disconnected} if the only connected components of $X$ are singletons and \emph{$0$-dimensional} if $X$ admits a basis of clopen sets.

A topological space $X$ such that every collection of disjoint nonempty open subsets of $X$ is countable is said to carry the \emph{countable chain condition}.

Note that for a compact space being totally disconnected is the same as being $0$-dimensional, and this corresponds to the fact that $C(X)$ has real rank zero. Moreover any compact Rickart space is $0$-dimensional and sub-Stonean, while the converse is false (take for example $\beta\en\setminus\en$).  The space $X$ carries the countable chain condition if and only if $C(X)$ has few orthogonal positive elements (see Definition~\ref{def:sccc}).

\begin{remark}
Let $X$ be a compact $0$-dimensional space, $CL(X)$ its algebra of clopen sets. For a Boolean algebra $B$, let $S(B)$ its Stone space, i.e., the space of all ultrafilters in $B$.

Note that if two $0$-dimensional spaces $X$ and $Y$ are homeomorphic then $CL(X)\cong CL(Y)$ and conversely, we have that $CL(X)\cong CL(Y)$ implies $X$ and $Y$ are homeomorphic to $S(CL(X))$.

Moreover, if $f\colon X\to Y$ is a continuous map of compact $0$-dimensional spaces we have that $\phi_f\colon CL(Y)\to CL(X)$ defined as $\phi_f(C)=f^{-1}[C]$ is an homomorphism of Boolean algebras.  Conversely, for any homomorphism of Boolean algebras $\phi\colon CL(Y)\to CL(X)$ we can define a continuous map $f_\phi\colon X\to Y$. If $f$ is injective, $\phi_f$ is surjective. If $f$ is onto $\phi_f$ is $1$-to-$1$ and same relations hold for $\phi$ and $f_\phi$.
\end{remark}

A Boolean algebra is atomless if $\forall a\neq 0$ there is $b$ such that $0<b<a$. For $Y,Z\subset B$ we say that $Y<Z$ if $\forall (y,z)\in Y\times Z$ we have $y< z$. Note that, for a $0$-dimensional space, $CL(X)$ is atomless if and only if $X$ does not have isolated points. In particular \[\abs{\{a\in CL(X)\colon a\text{ is an atom}\}}=\abs{\{x\in X\colon x\text{ is isolated}\}}.\]
\begin{defin}
Let $\kappa$ be an uncountable cardinal. A Boolean algebra $B$ is said to be \emph{$\kappa$-saturated} if every finitely satisfiable type of cardinality $<\kappa$ in the first-order language of Boolean algebras is satisfiable.
\end{defin}

For atomless Boolean algebras this model-theoretic saturation can be equivalently rephrased in terms of increasing and decreasing chains:

\begin{theorem}[{\cite[Theorem 2.7]{Mija79}}]\label{thm:CHSat.satbooleanalgebra}
Let $B$ be an atomless Boolean algebra, and $\kappa$ an uncountable cardinal.  Then $B$ is $\kappa$-saturated if and only if for every directed $Y<Z$ such that $\abs{Y}+\abs{Z}<\kappa$ there is $c\in B$ such that $Y<c<Z$.
\end{theorem}

We are ready to study which kind of topological properties the compact Hausdorff space $X$ has to carry in order to have some degree of saturation of the metric structure $C(X)$ and, conversely, to establish properties that are incompatible with the weakest degree of saturation of the corresponding algebra.  From now on $X$ will denote an infinite compact Hausdorff space (note that if $X$ is finite then $C(X)\cong\ce^n$ for some $n$, and so $C(X)$ is countably saturated).

The first limiting condition for the weakest degree of saturation are given by the following Lemma:
\begin{lemma}\label{lem:CHSat.abcondition}
Let $X$ be an infinite compact Hausdorff space, and suppose that $X$ satisfies one of the following conditions:
\begin{enumerate}[label=(\roman*)]
\item\label{abcond1} $X$ has the countable chain condition;
\item\label{abcond2} $X$ is separable;
\item\label{abcond3} $X$ is metrizable;
\item\label{abcond3a} $X$ is homeomorphic to a product of two infinite compact Hausdorff spaces;
\item\label{abcond4} $X$ is not sub-Stonean;
\item\label{abcond5} $X$ is Rickart.
\end{enumerate}
Then $C(X)$ is not countably degree-$1$ saturated.
\end{lemma}
\begin{proof}
First, note that \ref{abcond3} $\Rightarrow$ \ref{abcond2} $\Rightarrow$ \ref{abcond1}.  The fact that \ref{abcond1} implies that $C(X)$ is not countably degree-$1$ saturated is an instance of Lemma \ref{lem:CHSat.scccbad}.  Failure of countable degree-$1$ saturation for spaces satisfying \ref{abcond3a} follows from \cite[Theorem 1]{farah2011countable}, while for those satisfying \ref{abcond4} it follows from \cite[Remark 7.3]{pedersencorona} and \cite[Proposition 2.6]{farah2011countable}.  It remains to consider \ref{abcond5}.

Let $X$ be Rickart.  The Rickart condition can be rephrased as saying that any bounded increasing monotone sequence of self-adjoint functions in $C(X)$ has a least upper bound in $C(X)$ (see \cite[Theorem 2.1]{pedersensubstonean}). 

Consider a sequence $(a_n)_{n\in\en}\subseteq C(X)_1^+$ of positive pairwise orthogonal elements, and let $b_n=\sum_{i\leq n} a_i$.  Then $(b_n)_{n\in\en}$ is a bounded increasing sequence of positive operators,  so it has a least upper bound $b$. Since $\norm{b_n}=1$ for all $n$, we also have $\norm{b}=1$. The type consisting of $P_{-3}(x)=x$, with $K_{-3}=\{1\}$, $P_{-2}(x)=b-x$ with $K_{-2}=[1,2]$, $P_{-1}(x)=b-x-1$ with $K_{-1}=\{1\}$ and $P_n(x)=x-b_n-1$ with $K_n=[0,1]$ is consistent with partial solution $b_{n+1}$ (for $\{P_{-3},\ldots,P_n\}$).  This type cannot have a positive solution $y$, since in that case we would have that $y-b_n\geq 0$ for all $n\in\en$, yet $b-y>0$, a contradiction to $X$ being Rickart.
\end{proof}

Note that this proof shows that the existence of a particular increasing bounded sequence that is not norm-convergent but does have a least upper bound (a condition much weaker than being Rickart) is sufficient to prove that $C(X)$ does not have countable degree-$1$ saturation. The latter argument does not use that the ambient algebra is abelian.
 
We will compare the saturation of $C(X)$ (in the sense of Definition \ref{defin:Sat.Sat}) with the saturation of $CL(X)$, in the sense of the above theorem.

The results that we are going to obtain are the following:
\begin{theorem}\label{thm:CHSat.Abelian.Conds}
Let $X$ be a compact $0$-dimensional Hausdorff space. Then 
\[C(X) \text{ is countably saturated} \Rightarrow CL(X)\text{ is countably saturated }\] 
and 
\[CL(X)\text{ is countably saturated }\Rightarrow C(X) \text{ is countably q.f. saturated.}\]
\end{theorem}
\begin{theorem}\label{thm:CHSat.Abelian.Conds2}
Let $X$ be a compact $0$-dimensional Hausdorff space, and assume further that $X$ has a finite number of isolated points.  If $C(X)$ is countably degree-$1$ saturated, then $CL(X)$ is countably saturated.  Moreover, if $X$ has no isolated points, then countable degree-$1$ saturation and countable saturation coincide for $C(X)$.
\end{theorem}
\subsubsection{Proof of Theorem \ref{thm:CHSat.Abelian.Conds}}
Countable saturation of $C(X)$ (for all formulas in the language of $\Cstar$-algebras) implies saturation of the Boolean algebra, since being a projection is a weakly-stable relation, so every formula in $CL(X)$ can be rephrased in a formula in $C(X)$; to do so, write $\sup$ for $\forall$, $\inf$ for $\exists$, $\norm{x-y}$ for $x\neq y$, and so forth, restricting quantification to projections (this is possible since the set of projections is definable, see \cite[\S 3]{bourbaki}).  This establishes the first implication in Theorem \ref{thm:CHSat.Abelian.Conds}.  The second implication will require more effort.  To start, we will to need the following Proposition, relating elements of $C(X)$ to certain collections of clopen sets:

\begin{proposition}\label{prop:CHSat.coding}
 Let $X$ be a compact $0$-dimensional space and $f\in C(X)_{\leq 1}$. Then there exists a countable collection of clopen sets $\tilde Y_f=\{Y_{n,f} \colon n \in \en\}$ which completely determines $f$, in the sense that for each $x \in X$, the value $f(x)$ is completely determined by $\{n\colon x\in Y_{n,f}\}$.
\end{proposition}
\begin{proof}
Let $\ce_{m,1}=\{\frac{j_1+\sqrt{-1}j_2}{m}\colon j_1,j_2\in\ZZ\wedge \norm{j_1+\sqrt{-1}j_2}\leq m\}$.

For every $y\in\ce_{m,1}$ consider $X_{y,f}=f^{-1}(B_{1/m}(y))$. We have that each $X_{y,f}$ is a $\sigma$-compact open subset of $X$, so is a countable union of clopen sets $X_{y,f,1},\ldots,X_{y,f,n},\ldots\in CL(X)$. Note that $\bigcup_{y\in \ce_{m,1}}\bigcup_{n\in\en} X_{y,f,n}=X$. Let $\tilde X_{m,f}=\{X_{y,f,n}\}_{(y,n)\in\ce_{m,1}\times\en}\subseteq CL(X)$.

We claim that $\tilde X_f=\bigcup_m\tilde X_{m,f}$ describes $f$ completely. Fix $x\in X$. For every $m\in\en$ we can find a (not necessarily unique) pair $(y,n)\in\ce_{m,1}$ such that $x\in X_{y,f,n}$. Note that, for any $m,n_1,n_2\in\en$ and $y\neq z$, we have that $X_{y,f,n_1}\cap X_{z,f,n_2}\neq\emptyset$ implies $\abs{y-z}\leq\sqrt{2}/m$. In particular, for every $m\in\en$ and $x\in X$ we have \[2\leq \abs{\{y\in\ce_{m,1}\colon \exists n (x\in X_{y,f,n})\}}\leq 4.\] Let $A_{x,m}=\{y\in\ce_{m,1}\colon \exists n (x\in X_{y,f,n})\}$ and choose $a_{x,m} \in A_{x, m}$ to have minimal absolute value. Then $f(x)=\lim_m a_{x,m}$ so the collection $\tilde X_f$ completely describes $f$ in the desired sense.
\end{proof}

The above proposition will be the key technical ingredient in proving the second implication in Theorem \ref{thm:CHSat.Abelian.Conds}.  We will proceed by first obtaining the desired result under $\CH$, and then showing how to eliminate the set-theoretic assumption.

\begin{lemma}\label{lem:CHSat.qfsat}
Assume $\CH$.  Let $B$ be a countably saturated Boolean algebra of cardinality $2^{\aleph_0} = \aleph_1$.  Then $C(S(B))$ is countably saturated.
\end{lemma}
\begin{proof}
Let $B' \preceq B$ be countable, and let $\mc{U}$ be a non-principal ultrafilter on $\en$.  By the uniqueness of countably saturated models of size $\aleph_1$, and the $\CH$, we have $B'^{\mc{U}} \cong B$.  By \cite[Proposition 2]{Gurevic} and \cite[Remark 2.5.1]{Bankston} we have $C(S(B)) \cong C(S(B'))^{\mc{U}}$, and hence $C(S(B))$ is countably saturated.
\end{proof}

\begin{theorem}\label{thm:CHSat.qfsat}
Assume $\CH$.  Let $X$ be a compact Hausdorff $0$-dimensional space. If $CL(X)$ is countably saturated as a Boolean algebra, then $C(X)$ is countably quantifier-free saturated.
\end{theorem}
\begin{proof}
Let $\norm{P_n}=r_n$ be a collection of conditions, where each $P_n$ is a $2$-degree $*$-polynomial in $x_0,\ldots,x_n$, such that there is a collection $F=\{f_{n,i}\}_{n\leq i}\subseteq C(X)_{\leq 1}$, with the property that for all $i$ we have $\norm{P_n(f_{0,i},\ldots,f_{n,i})}\in (r_n)_{1/i}$ for all $n\leq i$.

For any $n$, we have that $P_n$ has finitely many coefficients. Consider $G$ the set of all coefficients of every $P_n$ and $L$ the set of all possible $2$-degree ${}^*$-polynomials in $F\cup G$. Note that for any $n\leq i$ we have that $P_{n}(f_{0,i},\ldots,f_{n,i})\in L$ and that $L$ is countable. For any element $f\in L$ consider a countable collection $\tilde X_f$ of clopen sets describing $f$, as in Proposition \ref{prop:CHSat.coding}.

Since $CL(X)$ is countably saturated and $2^{\aleph_0}=\aleph_1$ we can find a countably saturated Boolean algebra $B\subseteq CL(X)$ such that $\emptyset,X\in B$, for all $f\in L$ we have $\tilde X_f\subseteq B$, and $\abs{B}=\aleph_1$. 

Let $\iota \colon B\to CL(X)$ be the inclusion map.  Then $\iota$ is an injective Boolean algebra homomorphism and hence admits a dual continuous surjection $g_\iota\colon X\to S(B)$.

\begin{claim}
For every $f\in L$ we have that $\bigcup\tilde X_f=S(B)$.
\end{claim}

\begin{proof}
Recall that \[\bigcup \tilde X_f=X.\]
By compactness of $X$, there is a finite $C_{f}\subseteq \tilde X_{f}$ such that $\bigcup C_{f}=X$. In particular every ultrafilter on $B$ (i.e., a point of $S(B)$), corresponds via $g_\iota$ to an ultrafilter on $CL(X)$ (i.e., a point of $X$), and it has to contain an element of $C_{f}$. So $\bigcup \tilde X_{f}=S(B)$.
\end{proof}

From $g_\iota$ as above, we can define the injective map $\phi\colon C(S(B))\to C(X)$ defined as $\phi(f)(x)=f(g_\iota^{-1}(x))$.  Note that $\phi$ is norm preserving: Since $\phi$ is a unital $*$-homomorphism of $\Cstar$-algebra we have that $\norm{\phi(f)}\leq\norm{f}$. For the converse, suppose that $x\in S(B)$ is such that $\abs{f(x)}=r$, and by surjectivity take $y\in X$ such that $g_\iota(y)=x$. Then \[\abs{\phi(f)(y)}=\abs{f(g_\iota(g_\iota^{-1}(x)))}=\abs{f(x)}.\]

For every $f\in L$ consider the function $f'$ defined by $\tilde X_f$ and construct the corresponding ${}^*$-polynomials $P_n'$.
\begin{claim}
\begin{enumerate}
\item\label{1} $f=\phi(f')$ for all $f\in L$.
\item \label{2}$\norm{P_n'(f_{0,i}',\ldots,f_{n,i}')}\in (r_n)_{1/i}$ for all $i$ and $n\leq i$.
\end{enumerate}
\end{claim}
\begin{proof}
Note that, since $f_{n,i}\in L$ and every coefficient of $P_n$ is in $L$, we have that $P_n(f_{0,i},\ldots,f_{n,i})\in L$. It follows that condition 1, combined with the fact that $\phi$ is norm preserving, implies condition 2.

Recall that $g=g_\iota$ is defined by Stone duality, and is a continuous surjective map $g\colon X\to Y$. In particular $g$ is a quotient map. Moreover by definition, since $X_{q,f,n}\in CL(Y)=B\subseteq CL(X)$, we have that if $x\in Y$ is such that $x\in X_{q,f,n}$ for some $(q,f,n)\in \QQ\times L \times \en$, then for all $z$ such that $g(z)=x$ we have $z\in X_{q,f,n}$. Take $f$ and $x\in X$ such that $f(x)\neq\phi(f')(x)$. Consider $m$ such that $\abs{ f(x)-\phi(f')(x)}>2/m$. Pick $y\in\ce_{m,1}$ such that there is $k$ for which $x\in X_{y,f,k}$ and find $z\in Y$ such that $g(z)=x$. Then $z\in X_{y,f,k}$, that implies $f'(z)\in B_{1/m}(y)$ and so $\phi(f')(x)=f'(z)\in B_{1/m}(y)$ contradicting $\abs{f(x)-\phi(f')(x)}\geq 2/m$. 
\end{proof}
Consider now $\{\norm {P'_n(x_0,\ldots,x_n)}=r_n\}$. This type is consistent type in $C(S(B))$ by condition 2, and $C(S(B))$ is countably saturated by Lemma \ref{lem:CHSat.qfsat}, so there is a total solution $\overline g$. Then $h_j=\phi(g_j)$ will be such that $\norm{P_n(\overline h)}=r_n$, since $\phi$ is norm preserving, proving quantifier-free saturation for $C(X)$.
\end{proof}

To remove $\CH$ from Theorem \ref{thm:CHSat.qfsat} we will show that the result is preserved by $\sigma$-closed forcing.  We first prove a more general absoluteness result about truth values of formulas.  


Our result will be phrased in terms of truth values of formulas of \emph{infinitary} logic for metric structures.  Such a logic, in addition to the formula construction rules of the finitary logic we have been considering, also allows the construction of $\sup_n \phi_n$ and $\inf_n \phi_n$ as formulas when the $\phi_n$ are formulas with a total of finitely many free variables.  Two such infinitary logics have been considered in the literature.  The first, introduced by Ben Yaacov and Iovino in \cite{BenYaacovIovino}, allows the infinitary operations only when the functions defined by the formulas $\phi_n$ all have a common modulus of uniform continuity; this ensures that the resulting infinitary formula is again uniformly continuous.  The second, introduced by the first author in \cite{Eagle2014}, does not impose any continuity restriction on the formulas $\phi_n$ when forming countable infima or suprema; as a consequence, the infinitary formulas of this logic may define discontinuous functions.  The following result is valid in both of these logics; the only complication is that we must allow metric structures to be based on incomplete metric space, since a complete metric space may become incomplete after forcing.

\begin{lemma}\label{lem:CHSat.absoluteness}
Let $M$ be a metric structure, $\phi(\overline{x})$ be a formula of infinitary logic for metric structures, and $\overline{a}$ be a tuple from $M$ of the appropriate length.  Let $\mathbb{P}$ be any notion of forcing.  Then the value $\phi^M(\overline{a})$ is the same whether computed in $V$ or in the forcing extension $V[G]$.
\end{lemma}
\begin{proof}
The proof is by induction on the complexity of formulas; the key point is that we consider the structure $M$ in $V[G]$ as the same set as it is in $V$.  The base case of the induction is the atomic formulas, which are of the form $P(\overline{x})$ for some distinguished predicate $P$.  In this case since the structure $M$ is the same in $V$ and in $V[G]$, the value of $P^M(\overline{a})$ is independent of whether it is computed in $V$ or $V[G]$.

The next case is to handle the case where $\phi$ is $f(\psi_1, \ldots, \psi_n)$, where each $\psi_i$ is a formula and $f \colon [0, 1]^n \to [0, 1]$ is continuous.  Since the formula $\phi$ is in $V$, so is the function $f$.  By induction hypothesis each $\psi_i^M(\overline{a})$ can be computed either in $V$ or $V[G]$, and so the same is true of $\phi^M(\overline{a}) = f(\psi_1^M(\overline{a}), \ldots, \psi_n^M(\overline{a}))$.  A similar argument applies to the case when $\phi$ is $\sup_n\psi_n$ or $\inf_n\psi_n$.

Finally, we consider the case where $\phi(\overline{x}) = \inf_y \psi(\overline{x}, y)$ (the case with $\sup$ instead of $\inf$ is similar).  Here we have that for every $b \in M$, $\psi^M(\overline{a}, b)$ is independent of whether computed in $V$ or $V[G]$ by induction.  In both $V$ and $V[G]$ the infimum ranges over the same set $M$, and hence $\phi^M(\overline{a})$ is also the same whether computed in $V$ or $V[G]$.
\end{proof}

We now use this absoluteness result to prove absoluteness of countable saturation under $\sigma$-closed forcing.

\begin{proposition}\label{prop:CHSat.saturationAbsolute}
Let $\mathbb{P}$ be a $\sigma$-closed notion of forcing.  Let $M$ be a metric structure, and let $\Phi$ be a set of (finitary) formulas.  Then $M$ is countably $\Phi$-saturated in $V$ if and only if $M$ is countably $\Phi$-saturated in the forcing extension $V[G]$.
\end{proposition}
\begin{proof}
First, observe that since $\mathbb{P}$ is $\sigma$-closed, forcing with $\mathbb{P}$ does not introduce any new countable set.  In particular, the set of types which must be realized for $M$ to be countably $\Phi$-saturated are the same in $V$ and in $V[G]$.

Let $\mathbf{t}(\overline{x})$ be a set of instances of formulas from $\Phi$ with parameters from a countable set $A \subseteq M$.  Add new constants to the language for each $a \in A$, so that we may view $\mathbf{t}$ as a type without parameters.  Define
\[\phi(\overline{x}) = \inf\{\psi(\overline{x}) \colon \psi \in \mathbf{t}\}.\]
Note that $\phi^M(\overline{a}) = 0$ if and only if $\overline{a}$ satisfies $\mathbf{t}$ in $M$.  This $\phi$ is a formula in the infinitary logic of \cite{Eagle2014}.  By Lemma \ref{lem:CHSat.absoluteness} for any $\overline{a}$ from $M$ we have that $\phi^M(\overline{a}) = 0$ in $V$ if and only if $\phi^M(\overline{a}) = 0$ in $V[G]$.  As the same finite tuples $\overline{a}$ from $M$ exist in $V$ and in $V[G]$, this completes the proof.
\end{proof}

Finally, we return to the proof of Theorem \ref{thm:CHSat.Abelian.Conds}.  All that remains is to show:

\begin{lemma}\label{lem:CHSat.removingCH}
$\CH$ can be removed from the hypothesis of Theorem \ref{thm:CHSat.qfsat}. 
\end{lemma}

\begin{proof}
Let $X$ be a $0$-dimensional compact space such that $CL(X)$ is countably saturated, and suppose that $\CH$ fails.  Let $\mathbb{P}$ be a $\sigma$-closed forcing which collapses $2^{\aleph_0}$ to $\aleph_1$ (see \cite[\S VII.6]{kunen:settheory}).  Let $A = C(X)$ and $B = CL(X)$.  Observe that since $\mathbb{P}$ is $\sigma$-closed we have that $A$ remains a complete metric space in $V[G]$, and by Lemma \ref{lem:CHSat.absoluteness} $A$ still satisfies the axioms for commutative unital $\Cstar$-algebras of real rank zero.  Also by Lemma \ref{lem:CHSat.absoluteness} we have that $B$ remains a Boolean algebra, and the set of projections in $A$ in both $V$ and $V[G]$ is $B$.  We note that it may not be true in $V[G]$ that $X = S(B)$, or even that $X$ is compact (see \cite{Tall}), but this causes no problems because it follows from the above that $A = C(S(B))$ in $V[G]$.  By Proposition \ref{prop:CHSat.saturationAbsolute} $B$ remains countably saturated in $V[G]$.  Since $V[G]$ satisfies $\CH$ we can apply Theorem \ref{thm:CHSat.qfsat} to conclude that $A$ is countably quantifier-free saturated in $V[G]$, and hence also in $V$ by Proposition \ref{prop:CHSat.saturationAbsolute}.
\end{proof}

With $\CH$ removed from Theorem \ref{thm:CHSat.qfsat}, we have completed the proof of Theorem \ref{thm:CHSat.Abelian.Conds}.  It would be desirable to improve this result to say that if $CL(X)$ is countably saturated then $C(X)$ is countably saturated.  We note that if the map $\phi$ in Theorem \ref{thm:CHSat.qfsat} could be taken to be an elementary map then the same proof would give the improved conclusion.
\subsubsection{Proof of Theorem \ref{thm:CHSat.Abelian.Conds2}}
We now turn to the proof of Theorem \ref{thm:CHSat.Abelian.Conds2}. We start from the easy direction:

\begin{proposition}\label{prop:CHSat.totaldisc}
If $X$ is a $0$-dimensional compact space with finitely many isolated points such that $C(X)$ is countably degree-$1$ saturated, then the Boolean algebra $CL(X)$ is countably saturated.
\end{proposition}
\begin{proof}
Assume first that $X$ has no isolated points.  In this case we get that $CL(X)$ is atomless, so it is enough to see that $CL(X)$ satisfies the equivalent condition of Theorem~\ref{thm:CHSat.satbooleanalgebra}.

Let $Y < Z$ be directed such that $\abs{Y} + \abs{Z} < \aleph_1$.  Assume for the moment that both $Y$ and $Z$ are infinite.  Passing to a cofinal increasing sequence in $Z$ and a cofinal decreasing sequence in $Y$, we can suppose that $Z=\{U_n\}_{n\in\en}$ and $Y=\{V_n\}_{n\in\en}$, where
\[U_1\subsetneq\ldots\subsetneq U_n\subsetneq U_{n+1}\subsetneq\ldots\subsetneq V_{n+1}\subsetneq V_n\subsetneq\ldots\subsetneq V_1.\]
If $\bigcup_{n\in\en} U_n=\bigcap_{n\in\en} V_n$ then $\bigcup_{n\in\en} U_n$ is a clopen set, so by the remark following the proof of Lemma \ref{lem:CHSat.abcondition}, we have a contradiction to the countable degree-$1$ saturation of $C(X)$. 

For each $n \in \en$, let $p_n=\chi_{U_n}$ and $q_n=\chi_{V_n}$, where $\chi_A$ denotes the characteristic function of the set $A$. Then 
\[p_1<\ldots < p_n<p_{n+1}<\ldots<q_{n+1}<q_n<\ldots<q_1\] 
and by countable degree-$1$ saturation there is a positive $r$ such that $p_n<r<q_n$ for every $n$. In particular $A=\{x\in X\colon r(x)=0\}$ and $C=\{x\in X\colon r(x)=1\}$ are two disjoint closed sets such that $\overline{\bigcup_{n\in\en} U_n}\subseteq C$ and $\overline{X\setminus\bigcap_{n\in\en} V_n}\subseteq A$. We want to find a clopen set $D$ such that $A\subseteq D\subseteq X\setminus C$. For each $x\in A$ pick $W_x$ a clopen neighborhood contained in $X\setminus C$. Then $A\subseteq\bigcup_{x\in A}W_x$. By compactness we can cover $A$ with finitely many of these sets, say $A\subseteq \bigcup_{i\leq n}W_{x_i}\subseteq X\setminus C$, so $D = \bigcup_{i\leq n}W_{x_i}$ is the desired clopen set.

Essentially the same argument works when either $Y$ or $Z$ is finite.  We need only change some of the inequalities from $<$ with $\leq$, noting that a finite directed set has always a maximum and a minimum.

If $X$ has a finite number of isolated points, write $X = Y \cup Z$, where $Y$ has no isolated points and $Z$ is finite.  Then $C(X) = C(Y) \oplus C(Z)$ and $CL(X) = CL(Y) \oplus CL(Z)$.  The above proof shows that $CL(Y)$ is countably saturated, and $CL(Z)$ is saturated because it is finite, so $CL(X)$ is again saturated.
\end{proof}

To finish the proof of Theorem \ref{thm:CHSat.Abelian.Conds2} it is enough to show that when $X$ has no isolated points the theory of $X$ admits elimination of quantifiers, therefore countable quantifier-free saturation is equivalent to countable saturation. This was an unpublished result of Farah and Hart, that later appeared in \cite{EFKV.QE}. By later work on quantifier elimination (\cite{EFKV.QE} and \cite{EGV.Pseudo}), it is now known that $C(\beta\NN\setminus\NN)$ (which is elementary equivalent to $C(2^\NN)$) is the only infinite dimensional $\Cstar$-algebra which admits elimination of quantifiers in the theory of unital $\Cstar$-algebras.

The proof of Theorem \ref{thm:CHSat.Abelian.Conds2} is now complete by combining Theorem \ref{thm:CHSat.Abelian.Conds}, Proposition \ref{prop:CHSat.totaldisc}, and that the theory of unital abelian $\Cstar$-algebras of real rank zero without minimal projections (i.e., the theory of $C(2^\NN)$) has quantifier elimination.

\section{$\CH$ and homeomorphisms of \v{C}ech-Stone remainders of manifolds}\label{s:CH.SCmani}
This section is dedicated to show the first instances of the validity of Conjecture~\ref{conj:CH} for projectionless abelian algebras whose spectrum has Hausdorff dimension greater than $1$. Recall that if $X$ is a locally compact noncompact Polish space, then the corona of $C_0(X)$ is isomorphic to $C(\beta X\setminus X)$, the continuous functions on the \v{C}ech-Stone remainder of $X$, and automorphisms of $C(\beta X\setminus X)$ correspond bijectively to homeomorphisms of $\beta X\setminus X$. Throughout this section $X$ will always denote a locally compact noncompact Polish topological space.

With $\Homeo(\beta X\setminus X)$ and $\Triv(\beta X\setminus X)$ the sets of all, and trivial, homeomorphisms of $\beta X\setminus X$ respectively (see Definition~\ref{defin:TrivialAbel}), by the table in \S\ref{sss:CFconj}, it was unclear whether $\CH$ implies that $\Homeo(\beta X\setminus X)\neq\Triv(\beta X\setminus X)$ for connected spaces of dimension greater $1$. In particular, it was not known whether it is consistent to have a nontrivial homeomorphism of $\beta\er^n\setminus\er^n$, for $n\geq 2$.

The following theorem, contained in \cite{V.Nontrivial}  settles this uncertainty:
\begin{theorem}\label{thm:CHMani}
Let $X$ be a locally compact noncompact manifold. Then $\CH$ implies that $\Homeo(\beta X\setminus X)$ has nontrivial elements.
\end{theorem}

The rest of the section is dedicated to prove a stronger version of Theorem~\ref{thm:CHMani}, which relies on the definition of a flexible space (see Definition~\ref{defin:flexible}). We obtain a result concerning the algebra $Q(X,A)=C_b(X,A)/C_0(X,A)$ (see~\S\ref{s:multicorona}), for a flexible space $X$ and a $\Cstar$-algebra $A$, and we will then prove Theorem~\ref{thm:CHMani} from Theorem~\ref{thm:CHFlexi}. Lastly, in \S\ref{ss:arigidspace} we give an example of a very rigid space $X$ to which our result does not yield the existence of nontrivial elements of $\Homeo(\beta X\setminus X)$.

\begin{theorem}\label{thm:CHFlexi}
Let $X$ be a flexible space and $A$ be a $\Cstar$-algebra. Suppose that $\mathfrak d=\omega_1$ and $2^{\aleph_0}<2^{\aleph_1}$. Then $Q(X,A)$ has $2^{\aleph_1}$-many automorphisms. In particular, under $\CH$, there are $2^{\mathfrak c}$-many automorphisms of $Q(X,A)$, and so nontrivial ones.
\end{theorem}

\subsection{Flexible spaces and the proof of Theorem~\ref{thm:CHFlexi}}
Let $X$ be locally compact, noncompact and Polish, and fix a metric $d$ which is inducing the topology on $X$. 

Given a closed $Y\subseteq X$ we say that $\phi\in \Homeo(Y)$ \emph{fixes the boundary} of $Y$ if, whenever $y\in \bd_X(Y)=Y\cap (\overline{X\setminus Y})$, then $\phi(y)=y$. We denote the set of all such homeomorphisms by $\Homeo_{\bd_X(Y)}(Y)$ (or $\Homeo_{\bd}(Y)$ if $X$ is clear from the context). Every $\phi\in \Homeo_{\bd}(Y)$ can be extended in a canonical way to $\tilde\phi\in\Homeo(X)$ by 
\[
\tilde\phi(x)=\begin{cases}\phi(x) & \text{ if }x\in Y\\
x&\text{otherwise.}
\end{cases}
\]
If $Y_n\subseteq X$, for $n\in \NN$, are closed and disjoint sets with the property that no compact subset of $X$ intersects infinitely many $Y_n$'s, we have that $Y=\bigcup Y_n$ is closed. If $\phi_n\in\Homeo_{\bd}(Y_n)$ then $\phi=\bigcup\phi_n\in\Homeo_{\bd}Y$ is well defined. In this situation we abuse of notation and say that $\tilde\phi$ as constructed above extends canonically $\{\phi_n\}$.

For a $\phi\in\Homeo(X)$ we will denote by $r(\phi)$ the \emph{radius} of $\phi$  as
\[
r(\phi)=\sup_{x\in X}d(x,\phi(x)).
\]
If $Y$ is compact and $\phi\in\Homeo_{\bd}(Y)$ we have that $r(\phi)<\infty$ and $r(\phi)$ is attained by some $y\in Y$. It can be easily verified that $r(\phi_0\phi_1)\leq r(\phi_0)+r(\phi_1)$ for $\phi_0,\phi_1\in\Homeo(X)$.

Note that every $\tilde\phi\in\Homeo(X)$ determines uniquely a $\psi\in\Aut(C_b(X,A))$, which induces a $\tilde\psi\in\Aut(Q(X,A))$. If $Y_n\subseteq X$ are disjoint closed sets with the property that no compact $Z\subseteq X$ intersects infinitely many of them and $\phi_n\in\Homeo_{\bd}(Y_n)$, we will abuse of notation and say that $\psi$ and $\tilde\psi$ are \emph{canonically determined} by $\{\phi_n\}$.

\begin{defin}\label{defin:flexible}
A locally compact noncompact Polish space $(X,d)$ is \emph{flexible} if there are disjoint sets $Y_n\subseteq X$ and $\phi_{n,m}\in\Homeo_{bd}(Y_n)$ with the following properties:
\begin{enumerate}[label=(\arabic*)]
\item\label{defin:cond1} every $Y_n$ is a compact subset of $X$ and there is no compact $Z\subseteq X$ that intersects infinitely many $Y_n$'s and
\item\label{defin:cond2} for all $n$, $r(\phi_{n,m})$ is a decreasing sequence tending to $0$ as $m\to\infty$, with $r(\phi_{n,m})\neq 0$ whenever $n,m\in\NN$.
\end{enumerate}
The sets $Y_n$ and the homeomorphisms $\phi_{n,m}$ are said to witness that $X$ is flexible.
\end{defin}
\begin{remark}\label{rem:maniareflexible}
 We don't know whether condition \ref{defin:cond2} is equivalent to having a sequence of disjoint $Y_n$'s satisfying \ref{defin:cond1} for which $\Homeo_{\bd}(Y_n)$ has a continuous path. This condition is clearly stronger than \ref{defin:cond2}. In fact, being $\Homeo_{bd}(Y_n)$ a group, if it contains a path, then there is a path $a(t)\subseteq\Homeo_{\bd}(Y_n)$ with $a(0)=Id$ and $a(t)\neq a(0)$ if $t\neq 0$. By continuity, if a path exists, it can be chosen so that $s<t$ implies $r(a(s))<r(a(t))$. Since any closed ball in $\er^n$ has this property, a typical example of a flexible space is a manifold. 

We should also note that if $X$ is a locally compact Polish space for which there is a closed discrete sequence $x_n$ and a sequence of open sets $U_n$ with $U_i\cap U_j=\emptyset$ if $i\neq j$, $x_n\in U_n$, and such that each $U_n$ is a manifold, then $X$ is flexible. In particular, if $X$ has a connected component which is a noncompact Polish manifold, then $X$ is flexible.

Lastly, if $X$ is flexible and $Y$ has a compact clopen $Z$, then $\bd(Y_n\times Z)=\bd(Y_n)\times Z$, therefore $Z_n=Y_n\times Z$ and $\rho_{m,n}=\psi_{n,m}\times id$ witness the flexibility of $X\times Y$. In particular, if $Y$ is compact and metrizable, $X\times Y$ is flexible.
\end{remark}

By $\NN^{\NN\uparrow}$ we denote the set of all increasing sequences of natural numbers, where $f(n)>0$ for all $n$. If $f_1,f_2\in\NN^{\NN\uparrow}$ we write $f_1\leq^* f_2$ if 
\[
\forall^\infty n  (f_1(n)\leq f_2(n)).
\]
If $Y_n\subseteq X$ are compact sets with the property that no compact $Z\subseteq X$ intersects infinitely many $Y_n$, we can associate to every $f\in\NN^{\NN\uparrow}$ a subalgebra of $C_b(X,A)$ as
\begin{eqnarray*}
D_f(X,A,Y_n)=\{g\in C_b(X,A)\mid &&\forall \epsilon>0\forall^\infty n\forall x,y\in Y_n \\
&&(d(x,y)<\frac{1}{f(n)}\Rightarrow \norm{g(x)-g(y)}<\epsilon)\}.
\end{eqnarray*}
We denote by $C_f(X,A,Y_n)$ the image of $D_f(X,A,Y_n)$ under the quotient map $\pi\colon C_b(X,A)\to Q(X,A)$. As $X$, $A$ and $Y_n$ will be fixed throughout the proof, we will simply write $C_f$ and $D_f$.

The following proposition clarifies the structure of the $D_f$'s and the $C_f$'s.

\begin{proposition}\label{prop:buildingblocks}
Let $(X,d)$ be a locally compact noncompact Polish space and $A$ be a $\Cstar$-algebra. Let $Y_n\subset X$ be infinite  compact disjoint sets such that no compact subset of $X$ intersects infinitely many $Y_n$'s. Then:
\begin{enumerate}[label=(\arabic*)]
\item\label{prop.bb1} For all $f\in\NN^{\NN\uparrow}$ we have that $D_f$ is a $\Cstar$-subalgebra of $C_b(X,A)$. If $A$ is unital, so is $D_f$;
\item\label{prop.bb2} if $f_1\leq^* f_2$ then $C_{f_1}\subseteq C_{f_2}$;
\item\label{prop.bb3} $C_b(X,A)=\bigcup_{f\in\NN^{\NN\uparrow}}D_f$;
\item\label{prop.bb4} for all $f\colon \NN\to\NN$ there is $g\in C_b(X,A)$ such that $\pi(g)\notin C_f$.
\end{enumerate}
\end{proposition}
\begin{proof}
(1) and (2) follow directly from the definition of $D_f$ and $C_f$. For (3), take $g\in C_b(X,A)$. Since each $Y_n$ is compact and metric we have that $g\restriction Y_n$ is uniformly continuous. In particular there is $\delta_n>0$ such that $d(x,y)<\delta_n$ implies $\norm{g(x)-g(y)}<2^{-n}$ for all $x,y\in Y_n$. Fix $m_n$ such that $\frac{1}{m_n}<\delta_n$ and let $f(n)=m_n$. Then $g\in D_f$.

For (4), fix $f\in\NN^{\NN\uparrow}$ and $x_n\neq y_n\in Y_n$ with $d(y_n,z_n)<\frac{1}{f(n)}$. Since no compact set intersects infinitely many $Y_n$'s, both $Y'=\{y_n\}_n$ and $Z'=\{z_n\}_n$ are closed in $X$. Pick any $a\in A$ with $\norm{a}=1$ and let $g$ be a bounded continuous function such that $g(Y')=0$ and $g(Z')=a$. It is easy to see that $g\notin C_f$.
\end{proof}
The following Lemma represents the connections between the filtration we obtained and an automorphism of $Q(X,A)$.
\begin{lemma}\label{lem:CHMani:agreeing}
Let $(X,d)$, $A$, and $Y_n$ be as in Proposition~\ref{prop:buildingblocks} and suppose that $\phi_n\in \Homeo_{\bd}(Y_n)$. Let $\tilde\phi\in\Homeo(X)$ and $\tilde\psi\in\Aut(Q(X,A))$ be canonically determined by $\{\phi_n\}$ and $f\in\NN^{\NN\uparrow}$. Then:
\begin{enumerate}[label=(\arabic*)]
\item if there are $k,n_0$ such that for all $n\geq n_0$ 
\[
r(\phi_n)\leq\frac{k}{f(n)}
\]
 we have that $\tilde\psi(g)=g$ for all $g\in C_f$;
\item if for infinitely many $n$ we have that 
\[
r(\phi_n)\geq\frac{n}{f(n)}.
\]
 then there is $g\in C_f$ such that $\tilde\psi(g)\neq g$.
\end{enumerate}
\end{lemma}
\begin{proof}
Note that, if $g\in C_b(X,A)$ and $\tilde\psi$ is as above, we have $\tilde\psi(g)=g$ if and only if $g-\psi(g)\in C_0(X,A)$ where $\psi\in\Aut(C_b(X,A))$ is canonically determined by $\{\phi_n\}$.

To prove (1), let $k,n_0$ as above. Fix $\epsilon>0$ and $n_1>n_0$ such that whenever $n\geq n_1$ we have that $\frac{k}{f(n)}<\epsilon$ and if $x,y\in Y_n$ with $d(x,y)<\frac{1}{f(n)}$ then $\norm{g(x)-g(y)}<\epsilon/k$. Such an $n_1$ can be found, since $g\in D_f$. 
Let now $x\notin \bigcup_{i\leq n_1}Y_i$. Since $g(x)-\psi(g(x))=g(x)-g(\tilde\phi(x))$, if $x\notin\bigcup Y_n$ we have $\tilde\phi(x)=x$ and so $g(x)-\psi(g)(x)=0$. 
If $x\in Y_n$ for $n\geq n_1$,we have $d(x,\phi_n(x))<r(\phi_n)\leq\frac{k}{f(n)}$ and by our choice of $n_1$,
\[
\norm{g(x)-\psi(g)(x)}=\norm{g(x)-g(\phi_n(x))}\leq\epsilon.
\]
Since $\bigcup_{i\leq n_1}Y_i$ is compact, we have that $g-\psi(g)\in C_0(X,A)$, and (1) follows.

For (2), let $k(n)$ be a sequence of natural numbers such that 
\[
r(\phi_{k(n)})\geq\frac{k(n)}{f(k(n))}.
\]
We will construct $h\in D_f$ and show that $h-\psi(h)\notin C_0(X,A)$. Fix some $a\in A$ with $\norm{a}=1$.
If $m\neq k(n)$ for all $n$, set $h(Y_m)=0$. If $m=k(n)$, let $r=r(\phi_m)$ and pick $x_0=x_0(m)$ such that $d(x_0,\phi_m(x_0))=r$. Set $x_1=x_1(m)=\phi_m(x_0)$ and, for $i=0,1$, let 
\[
Z_i=\{z\in Y_m\mid d(z,x_i)\leq r/2\}.
\]
If $z\in Z_0$ define 
\[
h(z)=(\frac{d(z,x_0)}{r})a
\]
and if $z\in Z_1$ let
\[
h(z)=(1-\frac{d(z,x_1)}{r})a,
\]
while for $z\in Y_m\setminus(Z_0\cup Z_1)$ let $h(z)=\frac{a}{2}$. Let $h'\in C_b(X,A)$ be any function such that $h'(x)=h(x)$ whenever $x\in\bigcup Y_i$. Note that we have that $h'\in D_f$, as this only depends on its values on $\bigcup Y_i$. We want to show that $h'-\psi(h')\notin C_0(X,A)$. To see this, note that if $m=k(n)$ for some $n$ we have 
\[
h'(x_0(m))=0\text{ and }\psi(h')(x_0(m))=h'(\psi_m(x_0(m)))=h'(x_1(m))=a.
\] 
Since $\{x_0(m)\}_{m\in\NN}$ is not contained in any compact subsets of $X$ we have the thesis.
\end{proof}

We are ready to introduce a notion of coherence for sequences of homeomorphisms.
\begin{defin}\label{defin:CHMani:coherent}
Let $(X,d)$, $Y_n$ and $A$ be as in Proposition~\ref{prop:buildingblocks} and $\kappa$ be uncountable. Let $\{f_\alpha\}_{\alpha<\kappa}\subseteq\NN^{\NN\uparrow}$ be a $\leq^*$-increasing sequence of functions and $\{\phi_n^\alpha\}_{\alpha<\kappa}$ be such that for all $\alpha$ and $n$, 
\[
\phi_n^\alpha\in\Homeo_{\bd}(Y_n).
\]
$\{\phi_n^\alpha\}$ is said \emph{coherent with respect to }$\{f_\alpha\}$ if 
\[
\alpha<\beta\Rightarrow \exists k \forall^\infty n (r(\phi_n^\alpha(\phi_n^{\beta})^{-1})\leq\frac{k}{f_\alpha(n)}).
\]

If $\gamma$ is countable, $\{\phi_n^\alpha\}_{\alpha\leq\gamma}$ is coherent w.r.t.  $\{f_\alpha\}_{\alpha\leq\gamma}\subseteq\NN^{\NN\uparrow}$ if for all $\alpha<\beta\leq\gamma$ we have that $ \exists k \forall^\infty n (r(\phi_n^\alpha(\phi_n^{\beta})^{-1})\leq\frac{k}{f_\alpha(n)})$.
\end{defin}
\begin{remark}
Definition~\ref{defin:CHMani:coherent} is stated in great generality. We don't ask for the sequence $\{f_\alpha\}_{\alpha<\kappa}$ to have particular properties (e.g., being cofinal) or for the space $X$ to be flexible, even though such notionwill be used in such context.

Note that if $\{\phi_n^\alpha\}_{\alpha<\omega_1}$ is such, that for all $\gamma<\omega_1$, $\{\phi_n^\alpha\}_{\alpha\leq\gamma}$ is coherent w.r.t. $\{f_\alpha\}_{\alpha\leq\gamma}$ then $\{\phi_n^\alpha\}_{\alpha<\omega_1}$ is coherent w.r.t. $\{f_\alpha\}_{\alpha<\omega_1}$.
\end{remark}

Recalling that $\mathfrak d$ denotes the smallest cardinality of a $\leq^*$-cofinal family in $\NN^{\NN\uparrow}$, we say that a $\leq^*$ increasing and cofinal sequence $\{f_\alpha\}_{\alpha<\kappa}\subseteq\NN^{\NN\uparrow}$, for some $\kappa>\mathfrak d$, is \emph{fast} if for all $\alpha$ and $n$, 
\[
 n f_\alpha(n)\leq f_{\alpha+1}(n).
\]
If $\{f_\alpha\}_{\alpha<\mathfrak d}$ is fast, the same argument as in Proposition~\ref{prop:buildingblocks} shows that
\[
Q(X,A)=\bigcup_\alpha C_{f_\alpha}.
\]

The following lemma is going to be key for our construction. Its proof follows almost immediately from the definitions above, but we sketch it for convenience.

\begin{lemma}\label{lemma:coherent->uniqueness}
Let $(X,d)$, $A$ and $Y_n$ be fixed as in Proposition~\ref{prop:buildingblocks}. Let $\{f_\alpha\}$ be a fast sequence and suppose that $\{\phi_n^\alpha\}$ is a coherent sequence w.r.t. $\{f_\alpha\}$. Let $\tilde\psi_\alpha\in\Aut(Q(X,A))$ be canonically determined by $\{\phi_n^\alpha\}_{n}$. Then there is a unique $\tilde\Psi\in\Aut(Q(X,A))$ with the property that 
\[
\tilde\Psi(g)=\tilde\psi_\alpha(g), \,\,\, g\in C_{f_\alpha}.
\]
\end{lemma}
\begin{proof}
We define $\tilde\Psi(g)=\tilde\psi_\alpha(g)$ for $g\in C_{f_\alpha}$. If $\alpha<\beta$,  we define $\tilde\psi_{\alpha\beta}=\tilde\psi_\alpha(\tilde\psi_\beta)^{-1}$. As $\tilde\psi_{\alpha\beta}$ is canonically determined by $\{\psi_{n}^\alpha(\phi_n^{\beta})^{-1}\}_n$, and by coherence there are $k,n_0\in\NN$ such that whenever $n>n_0$ we have 
\[
r(\phi_n^\alpha(\phi_n^{\beta})^{-1})<\frac{k}{f_\alpha(n)}.
\]
 By condition (1) of Lemma~\ref{lem:CHMani:agreeing} we therefore have that $\tilde\psi_{\alpha\beta}(g)=g$ whenever $g\in C_{f_\alpha}$, and in this case $\tilde\psi_{\alpha}(g)=\tilde\psi_\beta(g)$, so $\tilde\Psi$ is a well defined morphisms of $Q(X,A)$ into itself.
Let $\tilde\psi'_\alpha\in\Aut(Q(X,A))$ be canonically determined by $\{(\psi_n^\alpha)^{-1}\}_n$. Since $\{\phi_n^\alpha\}$ is coherent w.r.t $\{f_\alpha\}$, so is $\{(\phi_n^{\alpha)^{-1}}\}$. In particular, if we let $\tilde\Psi'$ defined by $\tilde\Psi'(g)=\tilde\psi'_\alpha(g)$ for $g\in C_{f_\alpha}$, we have that $\tilde\psi'$ is a well-defined morphisms from $Q(X,A)$ into itself, with the property that $\tilde\Psi'\tilde\Psi=\tilde\Psi\tilde\Psi'=Id$, hence $\tilde\Psi$ is an automorphism.
This concludes the proof.
\end{proof}

We are now ready to proceed with the proof of Theorem~\ref{thm:CHFlexi}
\begin{proof}[Proof of Theorem~\ref{thm:CHFlexi}]
Fix $d$, $Y_n$ and $\phi_{n,m}\in\Homeo_{\bd}Y_n$ witnessing that $X$ is flexible.

We have to give a technical restriction (see Remark \ref{rem:technicalities}) on the kind of elements of $\NN^{\NN\uparrow}$ we are allowed to use. This restriction depends strongly on the choice of $d$, on the witnesses $Y_n$ and on the $\phi_{n,m}$'s. We define
\[
A_n=\{k\in \NN\mid\exists m (r(\phi_{n,m})\in [1/(k+1),1/k])\}
\]
As $r(\phi_{n,m})\to 0$ for $m\to \infty$, $A_n$ is always infinite. We define 
\[
\NN^{\NN\uparrow}(X)=\{f\in\NN^{\NN\uparrow}\mid f(n)\in A_n\}\subseteq\NN^{\NN\uparrow}
\]
Since each $A_n$ is infinite, $\NN^{\NN\uparrow}(X)$ is cofinal in $\NN^{\NN\uparrow}$.

As $\mathfrak d=\omega_1$, we can fix a fast sequence $\{f_\alpha\}_{\alpha\in\omega_1}\subseteq\NN^{\NN\uparrow}(X)$. Let $C_\alpha:=C_{f_\alpha}$.
Finally fix, for each limit ordinal $\beta<\omega_1$, a sequence $\alpha_{\beta,n}$ that is strictly increasing and cofinal in $\beta$.

We will make use of Lemma \ref{lemma:coherent->uniqueness} and construct, for each $p\in 2^{\omega_1}$, a sequence $\phi_n^\alpha(p)$ that is coherent w.r.t. $\{f_\alpha\}$.  For simplicity we write $\phi_n^\alpha$ for $\phi_n^\alpha(p)$.
Let $\phi_n^0=\Id$. Once $\phi_n^\alpha$ has been constructed, let 
\[
\phi_n^{\alpha+1}=\phi_{n,m}\phi_n^\alpha,\,\,\, \text{ if }p(\alpha)=1,
\]
where $m$ is the smallest integer such that $r(\phi_{n,m})\in[\frac{1}{f_\alpha(n)+1},\frac{1}{f_\alpha(n)}]$, and $\phi_n^{\alpha+1}=\phi_n^\alpha$ otherwise.
\begin{claim}
If $\{\phi_n^\gamma\}_{\gamma\leq\alpha}$ is coherent w.r.t. $\{f_\gamma\}_{\gamma\leq\alpha}$ then $\{\phi_n^\gamma\}_{\gamma\leq\alpha+1}$ is coherent w.r.t. $\{f_\gamma\}_{\gamma\leq\alpha+1}$.
\end{claim}
\begin{proof}
We want to show that whenever $\gamma<\alpha$ there is $k$ such that 
\[
\forall^\infty n (r(\phi_n^\gamma(\phi_n^{\alpha+1})^{-1})\leq\frac{k}{f_\gamma(n)}).
\]
If $p(\alpha)=0$ this is clear, so suppose that $p(\alpha)=1$.

Note that 
\[
\phi_n^\gamma(\phi_n^{\alpha+1})^{-1}=\phi_n^\gamma(\phi_n^{\alpha})^{-1}\phi_{n,m}^{-1}
\]
where $m$ was chosen as above, and so
\[
r(\phi_n^\gamma(\phi_n^{\alpha+1})^{-1})\leq r(\phi_n^\gamma(\phi_n^{\alpha})^{-1})+r(\phi_{n,m})\leq  \frac{k}{f_\gamma(n)}+\frac{1}{f_\alpha(n)}
\]
for some $k$ (and eventually after a certain $n_0$). Since $f_\alpha(n)\geq f_\gamma(n)$ (again, eventually after a certain $n_1$), the conclusion follow.
\end{proof}
We are left with the limit step. Suppose then that $\phi_n^\alpha$ has be defined whenever $\alpha<\beta$. For shortness, let $\alpha_i=\alpha_{i,\beta}$.
\begin{claim}
For all $i\in\NN$ there is $k_i$ such that whenever $j\geq i$ there exists $n_{i,j}$ such that
\[
r(\phi_n^{\alpha_i}(\phi_n^{\alpha_j})^{-1}))\leq\frac{k_i}{f_{\alpha_i(n)}},
\]
whenever $n\geq n_{i,j}$
\end{claim}
\begin{proof}
Fix $i\in\NN$. By coherence there are $\bar k<\bar n$ such that whenever $n\geq\bar n$ we have 
\[
r(\phi_n^{\alpha_i}(\phi_n^{\alpha_{i+1}})^{-1}))<\frac{\bar k}{f_{\alpha_i}(n)}.
\]
Let $j>i$ and $n'(j)>k'(j)>\bar n$ such that if $n\geq n'(j)$ then 
\[
r(\phi_n^{\alpha_{i+1}}(\phi_n^{\alpha_j})^{-1}))<\frac{k'(j)}{f_{\alpha_{i+1}}(n)}
\]
and
\[
f_{\alpha_{i+1}}(n)\geq nf_{\alpha_i}(n).
\]
Fix $k_i=\bar k+1$ and $n_{i,j}=n'(j)$. Then for $n\geq n_{i,j}$
\begin{eqnarray*}
r(\phi_n^{\alpha_i}(\phi_n^{\alpha_j})^{-1}))&\leq& r(\phi_n^{\alpha_i}(\phi_n^{\alpha_{i+1}})^{-1}))+r(\phi_n^{\alpha_{i+1}}(\phi_n^{\alpha_j})^{-1}))\leq \frac{k'(j)}{f_{\alpha_{i+1}}(n)}+\frac{\bar k}{f_{\alpha_i}(n)}\\
&\leq& \frac{n}{f_{\alpha_{i+1}}(n)}+\frac{\bar k}{f_{\alpha_i}(n)}\leq \frac{k_i}{f_{\alpha_i}(n)}
\end{eqnarray*}
\end{proof}
Fix an sequence of $k_i$ as provided by the claim. Let $m_0=0$ and $m_{i+1}$ be the least natural above $m_i$ such that if $n\geq m_i$ and $j> i\geq l$ then 
\[
r(\phi_n^{\alpha_l}(\phi_n^{\alpha_j})^{-1})<\frac{k_l}{f_{\alpha_l}(n)}.
\]
Defining $\phi_n^\beta=\phi_n^{\alpha_i}$ whenever $n\in [m_{i-1},m_i)$, we have that coherence is preserved, that is, $\{\phi_n^\alpha\}_{\alpha\leq\beta}$ is coherent w.r.t. $\{f_\alpha\}_{\alpha\leq\beta}$. We just proved that we can define $\phi_n^\alpha$ for every countable ordinal.

By the remark following Definition~\ref{defin:CHMani:coherent} we have that the sequence $\{\phi_n^\alpha\}_{\alpha<\omega_1}$ it is coherent w.r.t. $\{f_\alpha\}_{\alpha<\omega_1}$. By Lemma~\ref{lemma:coherent->uniqueness} there is a unique $\tilde\Psi=\tilde\Psi_p\in\Aut(Q(X,A))$ determined by $\{\phi_{n}^\alpha(p)\}_{\alpha<\omega_1}$. 

To conclude the proof, we claim that if $p\neq q$ we have $\tilde\Phi_p\neq\tilde\Phi_q$.
Let $\alpha$ be minimum such that $p(\alpha)\neq q(\alpha)$, and suppose $p(\alpha)=1$. Then 
\[
\phi_n^{\alpha}(p)=\phi_n^\alpha(q)
\]
so
\[
\phi_n^{\alpha+1}(p)=\phi_{n,m}\phi_n^{\alpha}(p)=\phi_{n,m}\phi_n^{\alpha}(q)
\]
where $m$ is the smallest integer for which $r(\phi_{n,m})\in [1/(f_\alpha(n)+1),1/f_\alpha(n)]$. In particular, eventually after a certain $n_0$,
\[
r(\phi_n^{\alpha+1}(q)\phi_n^{\alpha+1}(p)^{-1})=r(\phi_{n,m})\geq \frac{1}{f_\alpha(n)+1}\geq\frac{n}{f_{\alpha+1}(n)}.
\]
By Lemma~\ref{lem:CHMani:agreeing}, if $\tilde\psi_{\alpha+1}(p)\in\Aut(Q(X,A))$ is determined by $\phi_n^{\alpha+1}(p)$,  there is $g\in C_{\alpha+1}$ such that 
\[
\tilde\psi_{\alpha+1}(p)(g)\neq\tilde\psi_{\alpha+1}(q)(g).
\]
As Lemma \ref{lemma:coherent->uniqueness} states that 
\[
\tilde\Psi_p(g)=\tilde\psi_{\alpha+1}(p)(g)
\]
whenever $g\in C_{\alpha+1}$, we have that
\[
\tilde\Psi_p(g)=\tilde\psi_{\alpha+1}(p)(g)\neq \tilde\psi_{\alpha+1}(q)(g)=\tilde\Psi_q(g)
\]
\end{proof}
\begin{remark}\label{rem:technicalities}
The requirement of using $\NN^{\NN\uparrow}(X)$ instead of $\NN^{\NN\uparrow}$ is purely technical. Following Remark \ref{rem:maniareflexible}, if it is possible to choose $Y_n$ so that $\Homeo_{\bd}(Y_n)$ has a path (e.g., if $X$ is a manifold) then we can pick $\{\phi_{n,m}\}$ in order to have $A_n=\NN$ (eventually truncating a finite set).
\end{remark}

As promised, we are ready to give a proof of Theorem~\ref{thm:CHMani}, which in particular shows that, whenever $n\geq 1$, $\beta(\er^n)\setminus\er^n$ has plenty of nontrivial homeomorphisms under $\CH$. This is evidently stronger than Theorem~\ref{thm:CHMani}.
\begin{corollary}\label{cor:CHMani}
Assume $\CH$. Let $X$ be a locally compact noncompact metrizable manifold. Then there are $2^{\mathfrak c}$-many nontrivial homeomorphisms of $X^*$. Suppose moreover that $Y$ is a locally compact space with a compact connected component. Then $\beta(X\times Y)\setminus (X\times Y)$ has $2^{\mathfrak c}$-many nontrivial homeomorphisms.
\end{corollary}
\begin{proof}
Manifolds are flexible thanks to Remark~\ref{rem:maniareflexible}, and homeomorphisms of $\beta X\setminus X$ correspond to automorphisms of $Q(X,\ce)$. Since there can be only $\mathfrak c$-many trivial homeomorphisms of $\beta X\setminus X$, the first assertion is proved. The second assertion follows similarly from Remark~\ref{rem:maniareflexible}, as if $X$ is flexible and $Y$ has a compact connected component, then $X\times Y$ is flexible.
\end{proof}
Even though Theorem~\ref{thm:CHFlexi} doesn't apply to the corona of $C_0(X,A)$ whenever $A$ is nonunital, we can still say something in a particular case. Along the same lines as in the proof of Corollary~\ref{cor:CHMani}, if $A$ is a $\Cstar$-algebra that has a nonzero central projection (recall that $p\in A$ is central if $pa=ap$ for all $a\in A$) then it is possible to prove that under $\CH$ the corona of $C_0(X,A)$ has $2^{\mathfrak c}$-many automorphisms.

\subsection{A very rigid space}\label{ss:arigidspace}

We show the existence of a one dimensional space $X$ for which it is still unknown whether $\CH$ implies the existence of a nontrivial element of $\Homeo(\beta X\setminus X)$. In fact, this space it is not flexible, it is connected (), and it does not have an increasing sequence of compact subsets $K_n$ for which $\sup_n|K_n\cap (\overline{X\setminus K_n})|<\infty$ (and therefore it doesn't satisfy the hypothesis of \cite[Theorem 2.5]{Farah-Shelah.RCQ}). The space $X$ is a modification of a construction of Kuperberg that appeared in the introduction of \cite{Phillips-Weaver}. 

The construction of $X$ take place in the plane $\er^2$, and it goes as follows: take a copy of interval $a(t)$, $t\in [0,1]$ and let $x_0=a(0)$. At the midpoint of $a$, we attach a copy of the interval. We now have three copies of the interval attached to each other. At the midpoint of the first one, we attach two copies of the interval, at the midpoint of the second one we attach three of them, and four to the third.  We order the new midpoints and attach five intervals to the first one, six to the second, and so forth. We repeat this construction infinitely many times, making sure  at every step of the construction the length of new intervals attached is short enough to satisfy the following two conditions:
  \begin{itemize}
  \item the construction takes place in a prescribed big enough compact subset of $\er^2$ containing $a$;
  \item for every new interval attached at the $n$-th stage the only point of intersection with the construction at stage $n-1$ is the midpoint to which the new interval was attached. Similarly, two new interval intersect if and only if they are attached to the same point of the construction at stage $n-1$
  \item the length of every interval attached at stage $n$ is less than $2^{-n}$.
  \end{itemize}
 Let $Y$ be the closure of this iterated construction. For $n\geq 1$, let $y_n=a(\frac{1}{2^n})$, $a_n$ be one of the intervals attached to $y_n$, and $x_n$ be the endpoint of $a_n$ not belonging to $a$. Let $X=Y\setminus\{x_n\}_{n\geq 0}$. Since $y_n\to x_0$ and the length of $a_n$ goes to $0$ when $n\to\infty$, we have  $x_n\to x_0$. In particular,  $X$ is locally compact. By construction, the set 
 \[ 
 \bigcup_{n\geq 3}\{x\mid X\setminus\{x\}\text{ has $n$-many connected components}\}
 \] is dense in $X$, and for every $n\geq 3$ there is a unique point $x\in X$ such that $X\setminus\{x\}$ has exactly $n$-many connected components, therefore $X$ has no homeomorphism other than the identity. Being $X$ connected, it is not flexible and it doesn't satisfy \cite[Hypothesis 4.1]{Coskey-Farah}.

We are left to show that we cannot apply \cite[Theorem 2.5]{Farah-Shelah.RCQ}, that is, we show that if $X=\bigcup K_n$ for some compact sets $K_n$, then $\sup_n|K_n\cap(\overline{X\cap K_n})|=\infty$. Note that $a_k\setminus\{x_k\}$ cannot be contained in a compact set of $X$ and all $a_k$'s are disjoint. In particular, if $K$ is compact and $K\cap a_k\neq\emptyset$ then there is $y\in K\cap a_k$ such that $y\in\overline{X\setminus K}$. If $\bigcup K_n=X$ for some compact sets $K_n$, for all $k$ there is $n$ such that $K_n\cap a_i\neq\emptyset$ for all $i\leq k$. Therefore $|K_n\cap (\overline{X\setminus K_n})|\geq k$.

\chapter{Ulam stability}\label{ch:Ulam}
This chapter is dedicated to the development of a very strong notion of stability which will be used in understanding the structure of automorphisms of certain corona algebras in \ref{s:FA.Emb}. Such a strong notion of stability was already used in \cite{Farah.C}, while proving that all automorphisms of the Calkin algebra are inner, under the assumption of certain Forcing Axioms. Versions of Ulam stability related to near inclusions and perturbation of operator algebras are ubiquitous in the literature. The study of this phaenomena was initiated by J. Phillips and Raeburn (\cite{PR:pertAF}) and Christensen (\cite{Christensen.NI}) and \cite{Christ.PertOpAlg} among others), and culminated in \cite{christensen2012perturbations}.

With in mind the notion of $\epsilon$-$^*$-homomorphism as in Definition~\ref{defin:CstarPrel.Approxmaps}, we are ready to give the definition of the notion of stability we will use.
\begin{defin}\label{defin:US.ulamstab}
 Let $\SC$ and $\SD$ be two classes of $\Cstar$-algebras.  We say that the pair $(\SC,\SD)$ is \emph{Ulam stable} if for every
 $\e > 0$ there is a $\delta > 0$ such that for all $A\in\SC$ and $B\in\SD$ and for every $\delta$-$^*$-homomorphism
 $\phi \colon A\to B$, there is a $^*$-homomorphism $\psi \colon A\to B$ with $\norm{\phi - \psi} < \e$.
\end{defin}

Since this is a very strong of stability, it is not surprising that the classes of $\Cstar$-algebras which are known to be Ulam stable are very few.

\begin{theorem}[Theorem 5.1, \cite{Farah.C}]\label{thm:farahulam}
There are constants $K_1,\gamma> 0$ such that whenever $\e <\gamma$, $F_1,F_2\in\SF$ and $\phi \colon F_1\to F_2$ is an $\e$-$^*$-homomorphism, there is a $^*$-homomorphism $\psi \colon F_1\to F_2$ with $\norm{\phi - \psi} < K_1 \e$.
  Hence, the pair $(\SF,\SF)$ is Ulam stable.
\end{theorem}

\begin{theorem}[{\cite{Semrl.USAbel}}]\label{thm:SemrlAbel}
Let $\mathcal A$ be the class of unital abelian $\Cstar$-algebras. Then $(\mathcal A,\mathcal A)$ is Ulam stable.
\end{theorem}

The goal of this chapter is to prove the following two Ulam stability results:

\begin{theorem}\label{thm:US.USFinDim}
Let $\mathcal F$ be the class of all finite dimensional $\Cstar$-algebras and $\mathcal C^*$ be the class of all $\Cstar$-algebras. Then $(\mathcal F,\mathcal C^*)$ is Ulam stable.
\end{theorem}

\begin{theorem}\label{thm:US.AF}
Let $\mathcal{AF}$ be the class of unital AF algebras and $\mathcal M$ be the class of all von Neumann algebras. Then $(\mathcal {AF},\mathcal M)$ is Ulam stable.
\end{theorem}

Given a class of unital nuclear $\Cstar$-algebras $\mathcal C$, let $\mathcal D_\mathcal C$ be the class of all unital inductive limits of $\Cstar$-algebras in $\mathcal C$.
Formally, $A\in\mathcal D_\mathcal C$ if and only if there are a net $\Lambda$ and algebras $A_\lambda\in\mathcal C$, for $\lambda\in\Lambda$, with
\begin{itemize}
\item $A_\lambda\subseteq A_\mu$ for every $\lambda<\mu\in\Lambda$, where the inclusion is unital;
\item $\overline{\bigcup_{\lambda\in\Lambda}}A_\lambda=A$.
\end{itemize}
If $\mathcal C$ denotes the class of full matrix algebras, $\mathcal D_\mathcal C$ is the class of all unital UHF algebras. If $\mathcal C=\mathcal F$, then $\mathcal D_\mathcal C=\mathcal {AF}$, the class of all unital AF algebras.

\begin{theorem}\label{thm:US.limits}
Let $\mathcal C$ be a class of unital nuclear $\Cstar$-algebras and $\mathcal M$ be the class of Von Neumann algebras. If $(\mathcal C,\mathcal M)$ is Ulam stable, so is $(\mathcal D_{\mathcal C},\mathcal M)$.
\end{theorem}

It is clear that Theorem~\ref{thm:US.AF} will follows as a corollary, my applying the above to $\mathcal C=\mathcal F$, the class of finite-dimensional $\Cstar$-algebras, and using Theorem~\ref{thm:US.USFinDim}.

\section{Ulam stability for finite-dimensional algebras: the proof of Theorem~\ref{thm:US.USFinDim}}

Here we offer a quantitative version and a proof of Theorem~\ref{thm:US.USFinDim}.
\begin{notation}Throughout this section  $\mathcal F$ will denote the class of finite-dimensional $\Cstar$-algebras, and $\mathcal C^*$ the class of all $\Cstar$-algebras.
\end{notation}

\begin{theorem}\label{thm:US.ulam-stability}
There are $K,\delta>0$ such that given $\e<\delta$, $F\in\mathcal F$, $A\in\mathcal C^*$ and an $\e$-$^*$-homomorphism $\phi\colon F\to A$, there exists a $^*$-homomorphism $\psi\colon F\to A$ with \[\norm{\psi-\phi}<K\e^{1/2}.\]
Consequently, the pair $(\mathcal F,\mathcal C^*)$ is Ulam stable.
\end{theorem}
The proof goes through successive approximations of an $\e$-$^*$-homomorphism $\phi$ with increasingly nice properties. Each step will consist of an already-known approximation result; our proof will thus consist of stringing each of these results together, sometimes with a little work in between.  Before beginning the proof we describe some of the tools we will use.

The first results is Proposition~\ref{proposition:approximate-multiplicativity}. This is essentially proved in~\cite[Proposition~5.14]{AGG};
one can also find similar ideas in the proof of~\cite[Proposition~5.2]{Kazhdan}.
Our version is slightly more general, in that the values of $\rho$ are taken
from the invertible elements of a separable Banach algebra, and $\rho$ is
allowed to be just Borel measurable.  In our proof, we will need the Bochner
integral, which is defined for certain functions taking values in a Banach
space.  For an introduction to the Bochner integral and its properties, we refer
the reader to~\cite[Appendix~E]{Cohn}.  For our purposes, we note that if
$(X,\Sigma,\mu)$ is a measure space, $E$ is a separable Banach space, and $\SB$
is the Borel $\sigma$-algebra on $E$ generated by the norm-open subsets of $E$,
then the Bochner integral is defined for any $(\Sigma,\SB)$-measurable function
$f \colon X\to E$ such that the function $x\mapsto \norm{f(x)}$ is in
$L^1(X,\Sigma,\mu)$, and in this case,
\[
  \int f(x)\,d\mu(x) \in E
\]
and
\[
  \norm{\int f(x)\,d\mu(x)} \le \int \norm{f(x)}\,d\mu(x).
\]
Moreover such an $f$ is the pointwise limit of $(\Sigma,\SB)$-measurable
functions $f_n$ with finite range, such that $\norm{f_n(x)} \le \norm{f(x)}$ for
all $x\in X$.

Now suppose $G$ is a compact group, $\mu$ is the Haar measure on $G$, and $E$ is
a separable Banach space.  We will call an $f \colon G\to E$ \emph{Borel-measurable}
if $f$ is measurable with respect to the Borel $\sigma$-algebras on $G$ and $E$
generated by the given topology on $G$ and the norm topology on $E$.  If $f :
G\to E$ is Borel-measurable and $g\in G$, then we have
\[
  \int f(x)\,d\mu(x) = \int f(gx)\,d\mu(x).
\]
since this holds for such functions with finite range.

Finally, if $A$ is a unital Banach algebra, we will denote by $\GL(A)$ the set of
invertible elements of $A$.

\begin{proposition}
  \label{proposition:approximate-multiplicativity}
  Suppose $A$ is a unital separable Banach algebra, $G$ is a compact, second
  countable group, and $\rho \colon G\to \GL(A)$ is a Borel-measurable map
  satisfying, for all $u,v\in G$,
  \[
    \norm{\rho(u)^{-1}} \le \kappa
  \]
  and
  \[
    \norm{\rho(uv) - \rho(u)\rho(v)} \le \e
  \]
  where $\kappa$ and $\e$ are positive constants satisfying $\e < \kappa^{-2}$.
  Then there is a Borel-measurable $\tilde{\rho} \colon G\to\GL(A)$ such that
  \begin{enumerate}
    \item\label{rho:close} for all $u\in G$, $\norm{\tilde{\rho}(u) - \rho(u)} \le \kappa\e$,
    \item\label{rho:invertible} for all $u\in G$,
    \[
      \norm{\tilde{\rho}(u)^{-1}} \le \frac{\kappa}{1 - \kappa^2\e},
    \]
    and finally,
    \item\label{rho:multiplicative} for all $u,v\in G$,
    \[
      \norm{\tilde{\rho}(uv) - \tilde{\rho}(u)\tilde{\rho}(v)} \le 2\kappa^2\e^2.
    \]
  \end{enumerate}
\end{proposition}

\begin{proof}
  Define
  \[
    \tilde{\rho}(u) = \int \rho(x)^{-1}\rho(xu)\,d\mu(x)
  \]
  where $\mu$ is the Haar measure on $G$, and the integral above is the Bochner
  integral.  First we must check that $\tilde{\rho}$ is
  Borel-measurable.  To see this, consider the set $\SD$ of all bounded,
  Borel-measurable functions $f \colon G\times G\to \GL(A)$ such that
  \[
    \tilde{f}(u) = \int f(u,x)\,d\mu(x)
  \]
  is also Borel-measurable.  Then $\SD$ is closed under finite linear
  combinations, and contains all functions of the form
  \[
    f(u,v) = \left\{\begin{array}{ll} a & (u,v)\in S\times T \\ 0 & (u,v)\not\in
    S\times T \end{array}\right.
  \]
  where $S,T\subseteq G$ are Borel and $a\in \GL(A)$.  Moreover, by
  \cite[Proposition~E.1 and Theorem~E.6]{Cohn}, $\SD$ is closed
  under pointwise limits of uniformly bounded sequences of functions.  It
  follows that $\SD$ contains each function of the form $\chi_S a$ where
  $S\subseteq G\times G$ is Borel and $a\in \GL(A)$.  (Here we are using the fact
  that the Borel $\sigma$-algebra on $G\times G$ is generated by the Borel
  rectangles, which requires the second countability of $G$.)  Furthermore this
  implies that $\SD$ contains every bounded, Borel-measurable function $f :
  G\times G\to \GL(A)$.  In particular $\SD$ contains the function $(u,v)\mapsto
  \rho(v)^{-1}\rho(vu)$.
  
  Now, to check condition~\eqref{rho:close}, we have
  \begin{align*}
    \norm{\tilde{\rho}(u) - \rho(u)} & \le \int \norm{\rho(x)^{-1}\rho(xu) - \rho(u)}\,d\mu(x) \\
      & \le \int \norm{\rho(x)^{-1}}\norm{\rho(xu) - \rho(x)\rho(u)}\,d\mu(x) \le \kappa\e.
  \end{align*}
  Note that
  \[
    \norm{1 - \tilde{\rho}(u)\rho(u)^{-1}} \le \norm{\rho(u) - \tilde{\rho}(u)}\norm{\rho(u)^{-1}} \le \kappa^2\e.
  \]
  By standard spectral theory, since $\rho(u)$ is invertible and $\norm{1 - \tilde{\rho}(u)\rho(u)^{-1}} < 1$, we have that $\tilde{\rho}(u)$ is invertible too, and moreover
  \begin{align*}
    \norm{\tilde{\rho}(u)^{-1}} & \le \norm{\rho(u)^{-1}}\left(1 + \norm{1 - \tilde{\rho}(u)\rho(u)^{-1}} + \norm{1 - \tilde{\rho}(u)\rho(u)^{-1}}^2 + \cdots\right) \\
      & \le \frac{\kappa}{1 - \kappa^2\e}
  \end{align*}
  which proves condition~\eqref{rho:invertible}.  The real work comes now in proving condition~\eqref{rho:multiplicative}.  First, we note that
  \begin{eqnarray*}
    \tilde{\rho}(u)\tilde{\rho}(v) - \tilde{\rho}(uv) &=& \iint \left(\rho(x)^{-1}\rho(xu) \rho(y)^{-1} \rho(yv) - \rho(x)^{-1}\rho(xuv)\right)\,d\mu(x)\,d\mu(y)\\& =& I_1 + I_2
  \end{eqnarray*}
  where
  \[
    I_1 = \iint \left(\rho(x)^{-1}\rho(xu) - \rho(u)\right)\left(\rho(y)^{-1}\rho(yv) - \rho(v)\right)\,d\mu(x)\,d\mu(y)
  \]
  and
  \[
    I_2 = \iint \left(\rho(x)^{-1}\rho(xu)\rho(v) + \rho(u)\rho(y)^{-1}\rho(yv) - \rho(u)\rho(v) - \rho(x)^{-1}\rho(xuv)\right)\,d\mu(x)\,d\mu(y).
  \]
  For $I_1$ we have
  \[
    \norm{I_1} \le \iint \norm{\rho(x)^{-1}}\norm{\rho(xu) - \rho(x)\rho(u)}\norm{\rho(y)^{-1}}\norm{\rho(yv) - \rho(y)\rho(v)}\,d\mu(x)\,d\mu(y) \le \kappa^2\e^2.
  \]
  As for $I_2$, we have
  \[
    I_2 = \int \rho(x)^{-1} (\rho(xu)\rho(v) - \rho(xuv))\,d\mu(x) - \int (\rho(u)\rho(x)^{-1}\rho(x)\rho(v) - \rho(u)\rho(x)^{-1}\rho(xv))\,d\mu(x).
  \]
  Using the translation-invariance of $\mu$ on the first integral above to replace $xu$ with $x$, we see that
  \begin{align*}
    I_2 & = \int \rho(xu^{-1})^{-1} (\rho(x)\rho(v) - \rho(xv))\,d\mu(x) - \int \rho(u)\rho(x)^{-1}(\rho(x)\rho(v) - \rho(xv))\,d\mu(x) \\
      & = \int (\rho(xu^{-1})^{-1} - \rho(u)\rho(x)^{-1})(\rho(x)\rho(v) - \rho(xv))\,d\mu(x)
  \end{align*}
  Finally, note that
  \[
    \norm{\rho(xu^{-1})^{-1} - \rho(u)\rho(x)^{-1}} = \norm{\rho(xu^{-1})^{-1}(\rho(x) - \rho(xu^{-1})\rho(u))\rho(x)^{-1}} \le \kappa^2\e
  \]
  and
  \[
    \norm{\rho(x)\rho(v) - \rho(xv)} \le \e
  \]
  so we have that $\norm{I_2} \le \kappa^2\e^2$.  This proves condition~\eqref{rho:multiplicative}.
\end{proof}

We are now ready to prove our main result.  In the proof we will make several successive modifications to $\phi$, and in
each case the relevant $\e$ will increase by some linear factor.  In order to keep the notation readable, we will call
the resulting $\e$'s $\e_1, \e_2, \ldots$
\begin{proof}[Proof of Theorem~\ref{thm:US.ulam-stability}]
Let $\gamma, K > 0$ witness Farah's Theorem.  Let $\delta \ll \gamma, 1/K$.  We will in particular require $\delta <
2^{-12}$.
Fix $\e < \delta$, $A\in\mathcal C^*$, $F\in\mathcal F$, and an
$\e$-$^*$-homomorphism $\phi \colon F \to A$.   By Remark~\ref{rmk:CstarPrel.unital-contractive}, we can assume that $A$ is
unital, $\phi(1) = 1$, and $\norm{\phi} \le 1$.

Let $X = \{x_0,\ldots,x_k\}$ be a finite subset of $F_{\le 2}$ which is $\e/2$-dense in $F_{\le 2}$, with $x_0=1$, and such that the elements of norm $\leq 1$ are listed first. Define a map $\phi' \colon F_{\le 2} \to A$ by letting $\phi'(x) = \phi(x_i)$, where $i$ is the minimal integer
such that $\norm{x - x_i} < \e$.  Clearly, the range of $\phi'$ is just $\{\phi(x_0),\ldots,\phi(x_k)\}$, and if $B_i =
B(x_i,\e)\cap F_{\le 2}$, then
\[
  (\phi')^{-1}(\phi(x_i)) = B_i \setminus \bigcup_{j < i} B_j
\]
so $\phi'$ is a Borel measurable map.  Moreover, $\norm{\phi'(x) - \phi(x)} < \e$ for all $x\in F_{\le 2}$. For $x\in F$ with $\norm{x}>2$, define $\phi'(x)=\frac{\phi'(\lambda x)}{\lambda}$, where $\lambda=\frac{2}{\norm{x}}$. It follows that
$\phi'$ is an $\e_1$-$^*$-homomorphism, where $\e_1 = 4\e$, and that $\phi'$ remains Borel-measurable. Note also that $\phi'(1) = 1$, as $x_0=1$ and $\norm{\phi'} \le 1$, since we listed the elements of norm $\leq 1$ first  and $X$ is required to be $\e/2$-dense. Replacing $\phi$ with
$\phi'$ and $A$ with the $\Cstar$-algebra generated by $\{\phi(x_0),\ldots,\phi(x_k)\}$, we may assume that $\phi$ is
Borel-measurable and $A$ is separable.

Since $\e_1 < 1$ and $\phi$ is unital, it follows that for every $u\in\SU(F)$, we have $\norm{\phi(u^{-1})\phi(u) -
1} < 1$ and hence that $\phi(u)$ is invertible, and $\norm{\phi(u)^{-1}} \le 2$.  Let $\rho_0$ be the restriction of $\phi$ to $\SU(F)$.   Applying
Proposition~\ref{proposition:approximate-multiplicativity} repeatedly we may find a sequence of maps $\rho_n :
\SU(F)\to \GL(A)$ satisfying, for all $u,v\in \SU(F)$, 
\[
  \norm{\rho_n(uv) - \rho_n(u)\rho_n(v)} \le \delta_n,\,\,\,\norm{\rho_{n+1}(u) - \rho_n(u)} \le \kappa_n\delta_n \text{ and } \norm{\rho_n(u)^{-1}} \le \kappa_n,
\]
where $\delta_n$ and $\kappa_n$ are defined by letting $\delta_0 = \e_1$, $\kappa_0 = 2$, and
\[
  \delta_{n+1} = 2\kappa_n^2\delta_n^2, \text{ and } \kappa_{n+1} = \frac{\kappa_n}{1 - \kappa_n^2 \delta_n}.
\]
\begin{claim}
  For each $n$, $\kappa_{n+1} - \kappa_n < 2^{-n}$ and $\delta_n \le 2^{5(1 - 2^n)}\e_1$.  Consequently, $\kappa_n < 4$ for all
  $n$, and
  \[
    \sum_{n=0}^\infty \kappa_n \delta_n < 8\e_1.
  \]
\end{claim}
\begin{proof}
  We will prove the first part of the claim by induction on $n$.  For the base
  case we note that $\delta_0 = \e_1 = 4\e < 2^{-10}$,
  \[
    \kappa_1 - \kappa_0 \le \frac{2}{1 - 2^{-8}} - 2 < 1.
  \]
  Now suppose $\kappa_0,\ldots,\kappa_n$ and $\delta_0,\ldots,\delta_n$ satisfy the induction hypothesis above.  Then we clearly
  have
  \[
    \kappa_n < 2 + 1 + \cdots + 2^{-(n-1)} < 4.
  \]
  Using this fact and the assumption $\e_1 < 2^{-10}$,
  \[
    \delta_{n+1} = 2\kappa_n^2 \delta_n^2
      < 2(4^2) 2^{10(1 - 2^n)}\e_1^2
      < 2(4^2) (2^{-10}) 2^{10(1 - 2^n)}\e_1
      = 2^{5(1 - 2^{n+1})}\e_1.
  \]
  Moreover,
  \[
    \kappa_{n+1} - \kappa_n = \frac{\kappa_n^3 \delta_n}{1 - \kappa_n^2\delta_n} < \frac{(4^3) 2^{5(1 - 2^n)}\e_1}{1 -
    (4^2) 2^{5(1 - 2^n)}\e_1} < \frac{2^{1 - 5\cdot 2^n}}{1 - 2^{-1}} = 2^{2 - 5\cdot 2^n}
  \]
  Finally, note that $2 - 5\cdot 2^n \le -n$ for all $n\ge 0$.  This proves the first two parts of the claim.  We have
  already noted that $\kappa_n < 4$, therefore
  \[
    \sum_{n=0}^\infty \kappa_n \delta_n < 4\e_1 \sum_{n=0}^\infty 2^{5(1 - 2^n)} < 4\e_1 \sum_{n=0}^\infty 2^{-n} =
    8\e_1.
  \]
  This concludes the proof.
\end{proof}

It follows from the above claim that the map $\rho$ given by $\rho(u) = \lim \rho_n(u)$ is defined on $\SU(F)$, maps
into $\GL(A)$, and is multiplicative, Borel-measurable, and satisfies $\norm{\rho - \phi} \le  8\e_1 = \e_2$.

Fix a faithful representation $\sigma$ of $A$ on a separable Hilbert space $H$,
and let $\tau \colon \SU(F)\to\GL(H)$ be the composition $\sigma\circ\rho$.
Then $\tau$ is a group homomorphism which is Borel-measurable with respect to
the strong operator topology on $\B(H)$, and for each $u\in\SU(F)$,
$\norm{\tau(u)} \le 1 + \e_2$ and $\norm{\tau(u)^*\tau(u) - 1} \le \e_2(4 +
\e_2) = \e_3$.  Since $\SU(F)$ is compact, and hence unitarizable, it follows
that there is a $T\in\GL(H)$ such that $\pi(u) = T \tau(u) T^{-1}$ is unitary
for every $u\in\SU(F)$.  Moreover, following for example the construction in the proof
of \cite[Theorem~1.4.4]{Runde.LA}, we see that we may choose such a $T$
satisfying $\norm{T - 1} \le \e_3$.  Then,
\[
  \norm{\pi(u) - \tau(u)} \le \frac{2(1 + \e_2)\e_3}{1 - \e_3} = \e_4.
\]

Recall that $\SU(F)$, with the norm topology, and $\SU(H)$, with the strong operator topology, are Polish groups; then,
by Pettis's Theorem (see e.g.,~\cite[Theorem~2.2]{Rosendal.AC}), it follows that $\pi$, a Borel-measurable group
homomorphism, is continuous with respect to these topologies.  By the Peter-Weyl Theorem, we may write $H = \bigoplus
H_k$, where each $H_k$ is finite-dimensional and $\pi\rs H_k$ is irreducible. In particular, if $p_k=\text{proj}(H_k)$,
we have that for every $k\in\NN$ and $u\in \SU(F)$, $[p_k,\pi(u)]=0$, and moreover $\pi(u)=\sum p_k\pi(u)p_k$.  Now,
recall that $\norm{\phi(u) - \rho(u)} \le \e_2$ for each $u\in\SU(F)$; hence
\[
  \norm{\sigma(\phi(u)) - \pi(u)} \le \norm{\sigma(\phi(u)) - \tau(u)} + \norm{\tau(u) - \pi(u)} \le \e_4 + \e_2.
\]
It follows that $\norm{[\sigma(\phi(u)), p_k]} \le 2(\e_4 + \e_2)$ for each $u\in\SU(F)$ and $k\in\NN$.  Since each element $a$
of a unital $\Cstar$-algebra is a linear combination of 4 unitaries whose coefficients have absolute value at most $\norm{a}$,
we deduce that
\[
  \sup_{a\in F,\norm{a}\leq 1}\norm{[\sigma(\phi(a)),p_k]}\leq 8(\e_4 + \e_2) + 8\e_1
\]
Let $\phi_k$ be defined as 
\[
  \phi_k(a)=p_k((\sigma\circ \phi)(a))p_k.
\]
It is not hard to show that $\phi_k\colon F\to\SB(H_k)$ is an $\e_5$-$^*$-homomorphism, where $\e_5 = 8(\e_4 + \e_2) +
9\e_1$.  (In fact, $\phi_k$ is nearly an $\e_1$-$^*$-homomorphism; however, to check that $\phi_k(ab) -
\phi_k(a)\phi_k(b)$ is small we need the norm on the commutator computed above.)  By \cite[Theorem 5.1]{Farah.C} and our
choice of $\gamma$ and $K$, there is a $^*$-homomorphism $\psi_k\colon F\to\mathcal B(H_k)$ such that
$\norm{\phi_k-\psi_k} \le K\e_5$.

Consider now $\psi'=\bigoplus\psi_k$ and the $\Cstar$-algebras $C = \psi'[F]$ and $B = \sigma[A]$. For every $u\in\mathcal
U(F)$, we have
\[
  \norm{\psi'(u)-\pi(u)}=\sup_k\norm{\psi_k(u)-p_k\pi(u)p_k} \le K\e_5 + \e_4 + \e_2.
\]
Since we also have $\norm{\pi(u) - \sigma(\phi(u))} \le \e_4 + \e_2$, it follows that $C \subset^{\e_6} B$, where $\e_6
= K\e_5 + 2\e_4 + 2\e_2$.
By~\cite[Theorem~5.3]{Christensen.NI}, there is a partial isometry $V\in\B(H)$ such that $\norm{V-1_H}<120\e_6^{1/2}$ and
$V C V^* \subseteq B$.  In particular, $V$ is unitary, and the $^*$-homomorphism $\eta \colon F\to B$ defined by
$\eta(a) = V \psi'(a) V^*$ satisfies $\norm{\eta(a) - \psi'(a)} < 240\e_6^{1/2}$.  Since $\sigma$ is injective, for every $x\in F$ we can define
\[
  \psi(x)=\sigma^{-1}(\eta(x)).
\]
Then $\psi$ is a $*$-homomomorphism mapping $F$ into $A$.
Moreover by construction we have that
\[
  \norm{\psi-\phi}<L\e^{1/2},
\]
where $L$ is a constant independent of $\e$, the dimension of $F$, $A$, and $\phi$.  This completes the proof.
\end{proof}

\section{Ulam stability for AF algebras: the proof of Theorem~\ref{thm:US.limits}}
We now work towards the proof of a quantitative version of Theorem~\ref{thm:US.limits} and, specifically, of Theorem~\ref{thm:US.AF}.

 \begin{corollary}\label{cor:US.AF2}
Let $\mathcal{AF}$ be the class of all unital AF algebras and $\mathcal M$ be the class of all von Neumann algebras. There is $K$ such that whenever $A\in\mathcal{AF}$,
$M\in\mathcal{M}$, $\e>0$, and $\phi \colon A\to M$ is an $\e$-$^*$-homomorphism, there is a $^*$-homomorphism $\psi\colon
A\to M$ with $\norm{\phi-\psi}<K\e^{1/4}$. Therefore $(\mathcal{AF},\mathcal M)$ is Ulam stable.
\end{corollary}

It should be pointed out that we do not require, in the statement of Theorem \ref{thm:US.limits}, the $\e$-$^*$-homomorphisms to be $\delta$-injective, for any $\delta$.

We will make use of a small proposition and of a consequence of \cite[Theorem 7.2]{Johnson.AMNM}:
\begin{proposition}\label{lem:US.closeness}
Let $M$ be a von Neumann algebra and $x\in M$, $Y\subseteq M$ such that $\norm{x - y} \le \e$ for all $y\in Y$.  If
$z$ is any WOT-accumulation point of $Y$, then $\norm{x - z} \le \e$.
\end{proposition}
\begin{theorem}\label{thm:US.johnrevised}
There is $K$ such that for any unital, nuclear $\Cstar$-algebra $A$, von Neumann algebra $M$,
$\e>0$ and for any linear $\e$-$^*$-homomorphism $\phi\colon A\to M$, there is a $^*$-homomorphism $\psi$ with
$\norm{\phi-\psi}<K\e^{\frac{1}{2}}$. 
\end{theorem}

\begin{remark}
  Theorem 7.2 in \cite{Johnson.AMNM} is more general, as it applies to a class of Banach $*$-algebras
  which does not include just nuclear $\Cstar$-algebras.  However, in this context the constant $K$ depends on the constant of
  amenability of $A$, as it depends on the best possible norm of an approximate diagonal in $A\hat\otimes A$. Since
  every $\Cstar$-algebra is $1$-amenable (see, for example,
  \cite{Runde.LA}), Theorem~\ref{thm:US.johnrevised} follows.
\end{remark}

The proof of Theorem \ref{thm:US.limits} relies heavily on the fact that the range algebra, being a von Neumann algebra, is a dual Banach algebra. The assumption of nuclearity for elements of the class $\mathcal C$ is crucial due to the application of \cite[Theorem 3.1]{Johnson.AMNM}.

Recall (see the paragraph after Theorem~\ref{thm:US.USFinDim}), that if $\mathcal C$ is a class of unital $\Cstar$-algebras the class $\mathcal D_\mathcal C$ has been defined to be the class of all unital inductive limits of algebras in $\mathcal C$.

\begin{proof}[Proof of Theorem \ref{thm:US.limits}]
Let $\epsilon>0$, $A\in\mathcal D_{\mathcal C}$, $M\in\mathcal M$ and let $\{A_\lambda\}_{\lambda\in\Lambda}$ be a directed system of algebras in $\SC$ with direct limit $A$. Fix a nonprincipal ultrafilter $\mathcal U$ on $\Lambda$ and let $\eta=\frac{\epsilon^2}{K^2}$ where $K$ is given by Theorem \ref{thm:US.johnrevised}.

As $(\mathcal C,\mathcal M)$ is Ulam stable by hypothesis, we can fix $\delta$ such that whenever $C\in\SC$,
$M\in\mathcal M$, and $\phi\colon C\to M$ is a $\delta$-$^*$-homomorphism, there is a $^*$-homomorphism $\psi$ with
$\norm{\psi-\phi}<\eta$. Let $\rho\colon A\to M$ be a
$\delta$-$^*$-homomorphism.  Now for each $\lambda\in\Lambda$, there is a $^*$-homomorphism $\psi_\lambda$ with $\psi_\lambda\colon A_\lambda\to M$ such that 
\[
\norm{\psi_\lambda-\rho\restriction A_\lambda} < \eta
\] 
We extend each map $\psi_\lambda$ to $\bigcup_{\mu\in\Lambda} A_\mu$, setting $\psi_\lambda(a)=0$ if $a\notin A_\lambda$. Note that for every $a\in \bigcup A_\lambda$ there is $\lambda_0$ such that for all $\lambda\geq \lambda_0$ we have that $\norm{\psi_\lambda(a)-\rho(a)}<\eta$. For every $a\in\bigcup A_\lambda$, define 
\[
\psi(a)=\textrm{WOT}-\lim_{\mathcal U}\psi_\lambda(a)\in M.
\]
 Such a limit exists, since $M$ is a von Neumann algebra and $\norm{\psi_\lambda(a)}\leq\norm{a}$ for every
 $\lambda\in\Lambda$. In particular the map $\psi$ is a continuous, bounded, unital, linear map with domain equal to
 $\bigcup A_\lambda$, so it can be extended to a linear (actually, a completely positive and contractive) map 
\[
\tilde\psi\colon A\to M.
\]
By Proposition~\ref{lem:US.closeness}, for every $a\in \bigcup A_\lambda$ with
$\norm{a} \le 1$, we have $\norm{\tilde{\psi}(a) - \rho(a)} \le \eta$.  It
follows that $\tilde{\psi}$ is $4\eta$-multiplicative, i.e.
$\norm{\tilde{\psi}(ab) - \tilde{\psi}(a)\tilde{\psi}(b)} \le 4\eta$ for all
$a,b\in A$ with norm at most $1$.

As $M$ is a von Neumann algebra, and particular a dual Banach algebra, we can
now apply Theorem \ref{thm:US.johnrevised} to get a $^*$-homomorphism $\psi'\colon
A\to M$ with \[\norm{\rho-\psi'}<16K\eta^{1/2}=16\epsilon.\] The conclusion
follows.
\end{proof}

\chapter{The consequences of Forcing Axioms}\label{ch:FA}
The goal of this chapter is to explore the consequences of Forcing Axioms on the structure of automorphisms of a corona algebra, in the same way that Chapter\ref{ch:CH} has done for the assumption of $\CH$.

Due to the nigh technical complexity of the results contained in this chapter we will have to abandon the convention of using $A$, $B$,... to denote a $\Cstar$-algebra, as we will use such variables to denote subsets of $\NN$. Throughout this chapter, $\Cstar$-algebras will be denoted with the letters $\cstar{A}$, $\cstar{B}$,... .

\section{A lifting Theorem I: statements}\label{s:FA.Lift1} 
 Let $\seq{k(n)}{n \in \NN}$ be a sequence of natural numbers, and $\cstar{A}$ a separable $\Cstar$-algebra admitting an increasing approximate identity of projections $\{q_n\}$. If $A\subset\NN$ we denote by $P_A$ the projection in $\prod M_{k(n)}$ whose value at coordinate $n$ is $1$ if $n\in A$ and $0$ otherwise. By $\pi$ we will denote the canonical quotient map $\pi\colon\prod M_{k(n)}\to\prod M_{k(n)}/\bigoplus M_{k(n)}$. Throughout this section and the next, we will be working with a fixed linear $^*$-preserving contractive map
  \[
    \iso \colon \prod M_{k(n)} / \bigoplus M_{k(n)}\to \mathcal M(\cstar{A}) / \cstar{A}
  \]
  such that $\norm{\iso} = 1$ and for all $A\subseteq\NN$ and all $x\in\prod M_{k(n)}$,
  \begin{gather}
    \label{eqn:restrictions}  \iso(\pi(x P_A)) = \iso(\pi(x)) \iso(\pi(P_A)) \tag{$\land$}.
  \end{gather}
    Equation (\ref{eqn:restrictions}) implies that
  \begin{itemize}
    \item  $\iso(\pi(P_A))$ is a projection for every $A\subseteq\NN$, and
    \item  If $x$ and $y$ are elements of $\prod M_{k(n)}$ with almost-disjoint supports, then $\iso(\pi(x))\iso(\pi(y)) =
    0$.
  \end{itemize}
 Given $\iso$ as above, our goal is to find a lift of $\iso$ on a large set of the following nice form:
  \begin{defin}\label{def:asymptadditive}
  Let $\cstar{A}$ be a $\Cstar$-algebra with an increasing approximate identity of projections $\{q_n\}$.     A map $\alpha \colon \prod M_{k(n)} \to \mathcal M(\cstar{A})$ is \emph{asymptotically additive} if there is a sequence of maps $\alpha_n \colon M_{k(n)}\to \cstar{A}$ such that  
  \begin{itemize}
  \item for all $x=(x_n) \in \prod M_{k(n)}$, we have 
\[
\alpha(x)   = \sum_{n=0}^\infty \alpha_n(x_n),
 \] where the sum is intended in the strict topology as the limit of the partial sums $\sum_{n\leq N}\alpha_n(x_n)$ and 
 \item for all $n$ there are $n_0,k_1<k_2$ such that the image of $\alpha_j$, for $j\leq n$, is contained in $q_{k_1}\cstar{A}q_{k_1}$ and the image of $\alpha_j$, for $j\geq n_0$, is contained in $(1-q_{k_2})\cstar{A}(1-q_{k_2})$.
 \end{itemize}

We will often identify $\alpha$ with the sequence $\seq{\alpha_n}{n\in\NN}$.
  \end{defin}
\begin{remark}
    It should be emphasized that we make no assumptions on the structure of the maps  $\alpha_n$, other than that the partial sums $\sum_{n\leq N} \alpha_n(x_n)$ must converge in the strict topology. In particular, for this to happen, if $\alpha=\sum\alpha_n$ is asymptotically additive then $\norm{\alpha}$ is well defined. 
    In particular, if $\alpha$ is itself a unital $^*$-homomorphisms, then $\alpha$ is strictly-strictly continuous, as the image of an approximate identity for $\prod M_{k(n)}$ is an approximate identity for $\mathcal M(\cstar{A})$.
    \end{remark}
    Fix a sequence of finite sets $X_{n,k}\subseteq (M_n)_{\leq 1}$ such that $0,1\in X_{n,k}$ for all $n,k$ and $X_{n,k}$ is a $2^{-k}$-dense set of $(M_{n})_{\leq 1}$.
    \begin{defin}\label{defin:FA.Skel}
    Let $\cstar{A}$ be a $\Cstar$-algebra with an approximate identity of projections $\{q_n\}$ and $\alpha\colon\prod M_{k(n)}\to\mathcal M(\cstar{A})$ be an asymptotically additive map. $\alpha$ is said to be \emph{skeletal} if for all $n$ there is $k$ such that for all $x\in M_{k(n)}$ with $\norm{x}\leq 1$ and  $y\in X_{k(n),k}$ such that $\alpha_n(x)=\alpha_n(y)$.
    \end{defin}
    
    Skeletal maps are determined by a finite set when restricted to the unit ball. We will restrict to skeletal maps because of the following fact, that we will use in \S\ref{s:FA.Lift2}.     
    \begin{proposition}
    Let $\cstar{A}$ be a $\Cstar$-algebra with an approximate identity of projections $\{q_n\}$. Then $\Skel(\cstar{A})$, the set of all skeletal maps is separable in uniform topology as a subset of $\prod_{k,n}\Map(X_{n,k},\cstar{A})$.
    \end{proposition}
    
    If $\alpha$ is a lift of $\iso$ on some dense ideal $\SI$, then we can infer some approximate structure for the $\alpha_n$'s. Recall the definition of $\epsilon$-linear, $\epsilon$-multiplicative etc. etc. from Definition~\ref{defin:CstarPrel.Approxmaps}.

  \begin{proposition}\label{prop:FA.ApproxStruct}
  Let $\iso\colon \prod M_{k(n)}/\bigoplus M_{k(n)}\to\mathcal M(\cstar{A})/\cstar{A}$ be any map.  Suppose 
  \[
  \alpha=\sum\alpha_n\colon\prod M_{k(n)}\to\mathcal M(\cstar{A})
  \]
   is an asymptotically additive lift of $\iso$ on a dense ideal $\SI$. Then for every $\epsilon>0$ 
    \begin{enumerate}[label=(\arabic*)]
  \item\label{1} if $\iso$ is linear there is $n_0$ such that for every $n\geq n_0$ $\alpha_{[n_0,n]}=\sum_{n_0\leq j\leq n}\alpha_j$ is $\epsilon$-linear;
  \item if $\iso$ is also $^*$-preserving there is $n_0$ such that for every $n\geq n_0$, $\alpha_{[n_0,n]}=\sum_{n_0\leq j\leq n}\alpha_j$ is $\epsilon$-$^*$-preserving;
  \item\label{4} if $\iso$ is also multiplicative there is $n_0$ such that for all $n\geq n_0$, $\alpha_{[n_0,n]}=\sum_{n_0\leq j\leq n}\alpha_j$ is $\epsilon$-multiplicative.
    \end{enumerate}
  Also
  \begin{enumerate}[label=(\arabic*)]\setcounter{enumi}{4}
  \item\label{5} Suppose that $\supp(x)=\{n\mid x_n\neq 0\}\in\SI$ and $\iso$ is  norm-preserving. Then $\lim_n | \norm{x_n}-\norm{\alpha_n(x_n)}|=0$.
  \end{enumerate}
  \end{proposition}

\begin{proof}
We prove \ref{1} and \ref{5} and we leave to the reader the rest of the proof. In both cases we will argue by contradiction.

If \ref{1} is false then there is $\epsilon>0$ such that for all $n$ there is $n'>n$, $\lambda,\mu\in\ce$ such that $|\lambda|,|\mu|\leq 1$ and $x,y\in \prod_{n\leq j\leq n'}(M_{k(j)})_{\leq 1}$ such that 
\[
\norm{\alpha_{[n,n']}(\lambda x+\mu y)-\lambda\alpha_{[n,n']}(x)-\mu\alpha_{[n,n']}(y)}\geq \epsilon.
\]

We construct inductively sequences $n_i<m_i<n_{i+1}<\cdots$ and $\lambda_i,\mu_i,x_i,y_i$ with the following properties
\begin{itemize}
\item $\norm{\alpha_{[n_i,m_i]}(\lambda_ix_i+\mu_iy_i)-\lambda_i\alpha_{[n_i,m_i]}(x_i)-\mu_i\alpha_{[n_i,m_i]}(y_i)}\leq \epsilon$,
\item the ranges of $\alpha_{[n_i,m_i]}$ and $\alpha_{[n_j,m_j]}$ are orthogonal if $i\neq j$.
\end{itemize}
By passing to subsequences, we can assume that there are $\lambda,\mu$ such that $|\lambda_i-\lambda|<\epsilon,|\mu_i-\mu|<\epsilon$ for all $i$. Fix $E_i=[n_i,n_{i+1})$. Since $\SI$ is nonmeager and dense there is an infinite $X$ such that $\bigcup_{n\in X}E_n\in\SI$.

For \ref{1}, let $X=\{n_i\}$ be infinite and, for $n\in X$, $x_n,y_n,\lambda_n,\mu_n\in M_{k(n)}$ be sequences witnessing the failure of \ref{1} for $\epsilon>0$ with $\norm{x_n}\norm{y_n},|\lambda_n|, |\mu_n|\leq 1$ and 

\[
\norm{\alpha_n(\lambda_nx_n+\mu_ny_n)-\lambda_n\alpha_n(x_n)-\mu_n\alpha_n(y_n)}>\epsilon.
\]

Then for all $i\in X$ we have that 
\[
\norm{\alpha_[n_i,m_i](\lambda x_i+\mu y_i)-\lambda\alpha_{[n_i,m_i]}(x_i)-\mu\alpha_{[n_i,m_i]}(y_i)}>\frac{\epsilon}{2\norm{\alpha}}
\]

  Let $\bar x=\sum_{i\in X}\lambda x_i$ and $\bar y=\sum_{i\in X} \mu y_i$. Let $c_i=\alpha_{[n_i,m_i]}(\lambda x_i+\mu y_i)-\lambda\alpha_{[n_i,m_i]}(x_i)-\mu\alpha_{[n_i,m_i]}(y_i)$. Then $\{c_i\}_{i\in X}$ is a bounded sequence of orthogonal elements of norm $\leq3\norm{\alpha} $ and therefore converges in $\mathcal M(\cstar{A})$ to an element $c=\sum_{i\in X} c_i$. Since each $c_i$ has norm greater than $\frac{\epsilon}{2\norm{\alpha}}$ we have $c\notin \cstar{A}$.
 
On the other hand it is easy to verify that, since $\bigcup_{n\in X}E_n\in\SI$ and $\alpha$ is asymptotically additive, $c=\alpha(\bar x+\bar y)-\alpha(\bar x)-\alpha(\bar y)\in\cstar{A}$, as $\alpha$ is a lift of $\iso$, which is linear. This is a contradiction.

We now prove \ref{5}. Again, by contradiction. Suppose then that $x$ is such that $\supp(x)\in\SI$ and there is $\epsilon>0$ such that for an infinite $Y\subseteq X$ we have that $|\norm{x_n}-\norm{\alpha_n(x_n)}|>\epsilon$ for all $n\in Y$. Fix $r$ an accumulation point of $\{\norm{\alpha_n(x_n)}\}_{n\in Y}$ (since $\alpha=\sum_n\alpha_n$ is well defined, this sequence is bounded) and $Y_1\subseteq Y$ be infinite and such that $n\in Y_1$ implies $|\norm{\alpha_n(x_n)-r}<\epsilon/3$. Let
\[
Y_2=\{n\mid\norm{x_n}>r+\epsilon/2\}\text{ and }Y_3=\{n\mid\norm{x_n}<r-\epsilon/2\}.
\]
Either $Y_2$ or $Y_3$ is infinite and we can assume that $Y_2$ is. (If $Y_2$ is finite the contradiction will proceed in the same exact way, and we leave the proof to the reader).

As before, we can refine $Y_2$ to have that if $i\neq j\in Y_2$ then the ranges of $\alpha_i$ and $\alpha_j$ are orthogonal. Let $y=\sum_{n\in Y_2} x_{n}$. Then $\norm{\pi(y)}>r+\epsilon$, $\pi$ being the canonical map $\prod M_{k(n)}\to\prod M_{k(n)}/\bigoplus M_{k(n)}$ and so 
 \[\norm{\iso(\pi(y))}=\norm{\pi_1(\alpha(y))}\geq r+\epsilon/2,\]
  $\pi_1\colon\mathcal M(\cstar{A})\to\mathcal M(\cstar{A})/\cstar{A}$ being the canonical quotient. On the other hand, since for $n\neq m\in Y_2$ we have $\alpha_n(x_n)\alpha_m(x_m)=0$, we have that $\norm{\alpha(y)}=\sup_{n\in Y_2}\norm{\alpha(x_n)}< r+\epsilon/3$, a contradiction to $Y_2\in\SI$ and $\alpha$ being a lift of $\iso$ on $\SI$.
\end{proof}

From now on $\iso$ will always denote a linear $^*$-preserving map with $\norm{\iso}=1$ and satisfying \eqref{eqn:restrictions}.
\begin{theorem}\label{thm:FA.ApproxStruct}
Let $\iso$ be a $^*$-homomorphism $\iso\colon\prod M_{k(n)}/\bigoplus M_{k(n)}\to\mathcal M(\cstar{A})/\cstar{A}$ and $\alpha\colon\prod M_{k(n)}\to\mathcal M(\cstar{A})$ be an asymptotically additive lift for $\iso$ on a dense ideal $\SI$. Then there is a $^*$-homomorphism $\gamma\colon\prod M_{k(n)}\to\mathcal M(\cstar{A})$ such that, for all $x\in\prod M_{k(n)}$ we have $\gamma(x)-\alpha(x)\in\cstar{A}$.
\end{theorem}
\begin{proof}
To simplify notation we assume that $M_{k(n)}=M_n$. If not, the proof goes in the same way.

Note first of all that we can assume $\norm{\alpha_n}\leq 1$, as $\limsup_n\norm{\alpha_n}\leq 1$ by condition~\ref{5} of Proposition~\ref{prop:FA.ApproxStruct}. Again by Proposition~\ref{prop:FA.ApproxStruct} there is a decreasing sequence $\epsilon_n\to 0$ such that each $\alpha_n$ is an $\epsilon_n$-$^*$-homomorphism. 
\begin{claim}
For all $\epsilon>0$ there is $n_0$ such that for all $n_0<n_1<n_2<n_3$ if $x,y$ are contractions with $x\in \prod_{n_0\leq j\leq n_1} M_{j}$ and $y\in \prod_{n_2\leq j\leq n_3}M_j$ then $\norm{\alpha(x)\alpha(y)}<\epsilon$.
\end{claim}
\begin{proof}
This proof is similar to the one of condition~\ref{1} in Proposition~\ref{prop:FA.ApproxStruct}, using the fact that $\alpha$ is an asymptotically additive map for a $^*$-homomorphism, hence for a multiplicative map, on a nonmeager dense ideal $\SI$.
\end{proof}
We modify each $\alpha_n$ by setting 
\[
\alpha_n'(x)=\alpha_n(1)\alpha_n(1)^*\alpha_n(x)\alpha_n(1)\alpha_n(1)^*.
\]
 Note that $\alpha'=\sum_n\alpha_n'$ is still an asymptotically additive map which is a lift for $\iso$ on $\SI$. Note also that if $n_0<n_1<n_2<n_3$ are such that for all contractions $x,y$ with $x\in\prod_{n_0\leq j\leq n_1}M_j$ and $y\in \prod_{n_2\leq j\leq n_3}M_j$ we have that $\norm{\alpha(x)\alpha(y)}<\epsilon$, then $\norm{\alpha'(x)z}\leq \epsilon$ whenever $z$ is a contraction in the $\Cstar$-algebra generated by $\alpha'[\prod_{n_0\leq j\leq n_1}M_j]$.
  
With $\epsilon_n=2^{-n}$, we can find an increasing sequence of natural number $\{n_i\}$ such that 
\begin{itemize}
\item $\alpha'[\prod_{n_i\leq j<n_{i+1}}M_j]$ is an $\epsilon_i$-$^*$-homomorphisms;
\item if $i<l$ then for all contractions $x,y$ with $x$ in the $\Cstar$-algebra generated by $\alpha'[\prod_{n_i\leq j<n_{i+1}}]M_j$ and $y$ in the $\Cstar$-algebra generated by $\alpha'[\prod_{n_l\leq j<n_{l+1}}M_l]$;
\item if $i+1\leq l$ then $\alpha'(x)\alpha'(y)=0$ whenever $x\in\prod_{n_i\leq j<n_{i+1}}]M_j$, $y\in \prod_{n_l\leq j<n_{l+1}}]M_j$.
\end{itemize}
Set now $\beta_j=0$ if $j<n_0$ and if $n_i\leq j<n_{i+1}$ let, for $x\in\prod_{n_i\leq j<n_{i+1}}M_j$,
\[
\beta_j(x)=\alpha'(1_{[n_i,n_{i+1})})\alpha'(1_{[n_i,n_{i+1})})^*\alpha'(x)\alpha'(1_{[n_i,n_{i+1})})\alpha'(1_{[n_i,n_{i+1})})^*.
\]
We will construct inductively $\gamma_j$ with the following properties:
\begin{itemize}
\item if $n_i\leq j<n_{i+1}$, the range of $\gamma_j$ is included in the $\Cstar$-algebra generated by $\beta(\prod_{n_i\leq j<n_{i+1}}M_j)$.
\item $\gamma_j\gamma_{j'}=0$ if $j\neq j'$, 
\item $\gamma_j$ is a $^*$-homomorphisms, and there is a fixed $K$ such that if $n_i\leq j$ then
\[
\norm{\gamma_{j}-\alpha_j}\leq10K(\epsilon_k)^{1/2}.
\]
\end{itemize}
It is clear that, if it is possible to construct such $\gamma_i$, then $\gamma(x)=\sum\gamma_i(x)$ satisfies the thesis of the theorem.

Let $K$ and $\delta>0$ be the numbers provided in Theorem~\ref{thm:US.ulam-stability}. Let $l$ such that $\epsilon_l<\delta$. If $j<n_l$, set $\gamma_j=0$. If $\gamma_j$ has been constructed for all $j<n_i$. If $n_i\leq j$, let $\beta'_j(x)=(1-\sum_{k<n_i}\gamma_k(1))\beta_j(x)(1-\sum_{k<n_i}\gamma_k(1))$. Since the range of $(1-\sum_{k<n_i}\gamma_k(1))$ is included in the $\Cstar$-algebra generated by $b_i=\beta(\prod_{k<n_i}M_k)$, and in particular in $b_i\cstar{A}b_i$, and for each contraction $x$ in $\prod_{n_i\leq j<n_{i+1}}M_j$ we have that $\norm{xb_i}<\epsilon_i$, we have that $\sum_{n_i\leq j<n_{i+1}}\beta'_j$ is a $5\epsilon_i$-$^*$-homomorphisms whose range is in the $\Cstar$-algebra generated by  $\beta(\prod_{n_i\leq j<n_{i+1}}M_j)$. By Theorem~\ref{thm:US.ulam-stability} and our choice of $\delta$ and $K$, we can find a $^*$-homomorphisms 
\[\gamma_{[n_i,n_{i+1})}\colon \prod_{n_i\leq j<n_{i+1}}M_j\to\beta(1_{[n_i,n_{i+1})})\beta(1_{[n_i,n_{i+1})})^*\cstar{A}\beta(1_{[n_i,n_{i+1})})\beta(1_{[n_i,n_{i+1})})^*
\] such that 
\[
\norm{\gamma_{[n_i,n_{i+1})}-\sum_{n_i\leq j<n_{i+1}}\beta_j}<5K(\epsilon_i)^{1/2}.
\]
By letting $\gamma_j=\gamma_{[n_i,n_{i+1})}\restriction M_j$ we have the thesis.
\end{proof}

We may state the main Theorem of this section, with in mind the definitions of ccc/Fin from \S\ref{ss:ST.idealsonomega}.

 \begin{theorem}
    \label{thm:FA.liftingtheorem}
    Assume $\OCA_\infty + \MA_{\aleph_1}$ and let $\iso$ be as above.  Then there is a ccc/Fin ideal $\SJ \subseteq \mathcal P(\NN)$, and an  asymptotically additive $\alpha \colon \prod M_{k(n)}\to \mathcal M(\cstar{A})$, such that $\alpha$ is a lift of $\iso$ on $\SJ$.
  \end{theorem}

  \begin{remark}
    Theorem~\ref{thm:FA.liftingtheorem} is an analogue of the ``$\OCA$ lifting theorem'' from~\cite{Farah.AQ}.  The ideal $\SJ$ cannot always be as large as $\mathcal P(\NN)$; however, in our applications of    Theorem~\ref{thm:FA.liftingtheorem}, we will always find that $\SJ = \mathcal P(\NN)$.
  \end{remark}

If $A\subseteq\en$ we denote by $M[A]$ the set of elements whose support is in $A$, that is \[M[A]=\{(x_n)\in\prod M_{k(n)}\mid (i\notin A\Rightarrow x_i=0)\}.\] With this notation we have that $M[\NN]=\prod M_{k(n)}$.

 The remainder of this section, along with the next, is dedicated to prove Theorem~\ref{thm:FA.liftingtheorem}.  Before getting to lifts that are asymptotically additive, however, we will need to deal with lifts that are nice in other  (weaker) senses. 
 The various notions are as follows. (Recall that $\mathcal M(\cstar{A})_{\leq 1}$ and $M[\NN]_{\leq 1}$ are Polish spaces when equipped with the strict topology, as noted in \ref{ss:CstarPrel.strict}).
  
\begin{defin}
  Let $\e \ge 0$ be given, and $X\subseteq M[\NN]$.
    \begin{itemize}
      \item  An \emph{$\e$-lift of $\iso$ on $X$} is a function $F$ mapping to $\mathcal M(\cstar{A})$, whose domain contains $X$,     such that $\norm{\pi(F(x)) - \iso(\pi(x))} \le \e$ for all $x\in X$ with $\norm{x}\leq 1$.
      \item  A \emph{$\sigma$-$\e$-lift} of $\iso$ on $X$ is a sequence of functions $F_n$ ($n\in\NN$) mapping to      $\mathcal M(\cstar{A})$, such that the domain of each $F_n$ contains $X$, and for every $x\in X$ with $\norm{x}\leq 1$, there is some $n\in\NN$ such  that $\norm{\pi(F_n(x)) - \iso(\pi(x))} \le \e$.
    \end{itemize}
  \end{defin}

  When $\e = 0$, we come back to our first definition of lift.  Our efforts will be focused on finding lifts that have  various nice properties with respect to the ambient topological structure of $M[\NN]$.

  \begin{defin}\label{defin:theideals}
  Let $\e \ge 0$ be given.
    \begin{itemize}
      \item  We define $\SI^\e$ be the set of $A\subseteq \NN$ such    that there exists an $\e$-lift of $\iso$ on $M[A]$ which is asymptotically additive.
      \item  We define $\SI^\e_C$ to be the    set of $A\subseteq\NN$ such that there exists a $\Cmeas$-measurable $\e$-lift of $\iso$ on $M[A]$.
    \end{itemize}
    When $\e = 0$, we write $\SI^0 = \SI$ and $\SI^0_C = \SI_C$.
  \end{defin}

\begin{lemma}\label{lemma:ideals}

For all $\e \ge 0$, $\SI^\e$ and $\SI^\e_C$ are ideals on $\NN$.

\end{lemma}

\begin{proof}

By definition, each $\SI^\e$ and $\SI^\e_C$ is hereditary.  To see that $\SI^\e_C$ is closed under finite unions, we will show that if $A,B\in\SI^\e_C$ and $A\cap B = \emptyset$, then $A\cup B\in\SI^\e_C$.  Choose $\Cmeas$-measurable functions $F$ and $G$ such that $F$ and $G$ are $\e$-lifts of $\iso$ on $M[A]$ and $M[B]$ respectively.  Choose $Q,R\in\mathcal M(\cstar{A})$ with $\pi(Q) = \iso(\pi(P_A))$ and $\pi(R) = \iso(\pi(P_B))$.  Put
    \[
      H(x) = Q F(x P_A) Q + R G(x P_B) R
    \]
    Then $H$ is $\Cmeas$-measurable.  Moreover, if $x\in M[A\cup B]_{\leq 1}$, we have $x = x P_A + x P_B$, and hence
    \begin{align*}
      \pi(H(x)) - \iso(\pi(x)) & = \pi(Q) (\pi(F(xP_A)) - \iso(\pi(x P_A)))\pi(Q) \\
         & \qquad + \pi(R) (\pi(G(x P_B)) - \iso(\pi(x P_B)))\pi(R)
    \end{align*}
    Since $\pi(Q)$ and $\pi(R)$ are orthogonal projections, and
    \[
      \norm{\pi(F(x P_A)) - \iso(\pi(x P_A))}, \norm{\pi(G(x P_B)) - \iso(\pi(x P_B))} \le \e
    \]
    it follows that $\norm{\pi(H(x)) - \iso(\pi(x))} \le \e$.  To see that $\SI^\e$ is closed under finite unions, note that if $F$ and $G$ above are asymptotically additive, then so is the resulting $H$.
  \end{proof}

  The strategy for the proof of Theorem~\ref{thm:FA.liftingtheorem} is to show that each of the ideals $\SI^\e$ and $\SI^\e_C$ is, in a sense, large.  The following five lemmas will do this.  Their proofs are self-contained, and together they form the backbone of the argument towards Theorem~\ref{thm:FA.liftingtheorem}.  For now, we will simply state them, deferring their proofs to section \ref{s:FA.Lift2}.

  \begin{lemma}
    \label{lemma:oca->treelike}
    Assume $\OCA_\infty$, let $\e > 0$ and $\SA\subseteq\mathcal P(\NN)$ be a treelike, a.d. family.  Then for all but  countably-many $A\in\SA$, there is a $\sigma$-$\e$-lift of $\iso$ on $M[A]$ consisting of $\Cmeas$-measurable functions.
  \end{lemma}

  \begin{lemma}
    \label{lemma:sigma->single}
    Let $\e > 0$ and $A\subseteq\NN$. Suppose that there is a $\sigma$-$\e$-lift of $\iso$ on $M[A]$, consisting of $\Cmeas$-measurable functions and that $A = \bigcup_n A_n$ is a partition of $A$ into infinite sets.  Then there is some $n$ such that $A_n\in \SI_C^{4\e}$.
  \end{lemma}

  \begin{lemma}
    \label{lemma:oca->alternative}
    Assume $\OCA_\infty + \MA_{\aleph_1}$.  Then either
    \begin{enumerate}[label=(\Roman*)]
      \item\label{alt:I}  there is an uncountable, treelike, a.d. family $\SA\subseteq\P(\NN)$ which is disjoint from  $\SI$, or
      \item\label{alt:II}  for every $\e > 0$, there is a sequence $\{F_n\}_{n\in\NN}$ of $\Cmeas$-measurable functions such that for every  $A\in\SI^\e$, there is an $n$ such that $F_n$ is an $\e$-lift of $\iso$ on $M[A]$.
    \end{enumerate}
  \end{lemma}

  \begin{lemma}
    \label{lemma:sigma-baire->borel}
    Suppose $F_n \colon M[\NN]\to \mathcal M(\cstar{A})$ is a sequence of Baire-measurable maps, $\e>0$ and $\SJ\subseteq\P(\NN)$ is a  nonmeager ideal such that for all $A\in\SJ$ there is some $n$ such that $F_n$ is an $\e$-lift of $\iso$ on $\M[A]$. Then there is a Borel-measurable map $G \colon M[\NN]\to\mathcal M(\cstar{A})$ that is a $24\epsilon$-lift of $\iso$ on $\SI$.
  \end{lemma}

  \begin{lemma}
    \label{lemma:C->asymptotic}
  Suppose $F \colon M[\NN]\to \mathcal M(\cstar{A})$ is a $\Cmeas$-measurable map and $\SJ\subseteq \P(\NN)$ is a nonmeager ideal such  that for every $A\in\SJ$, $F$ is a lift of $\iso$ on $M[A]$.  Then there is an asymptotically additive $\alpha$ that is a lift of $\iso$ on $\SI$. Also, $\alpha$ can be chosen to be a skeletal map. Hence, $\SI_C = \SI$.
  \end{lemma}
  With these lemmas in hand, we will finish this section by connecting the dots and proving Theorem~\ref{thm:FA.liftingtheorem}.
 First we provide an easy connection between Borel-measurable $\e$-lifts and $\Cmeas$-measurable lifts.

  \begin{lemma}    \label{lemma:epsilon-borel->C-measurable}

    Suppose that $\SJ$ is an ideal, and for every $\e > 0$, there is a Borel-measurable $G_\e \colon M[\NN]\to \mathcal M(\cstar{A})$ that is an $\e$-lift of $\iso$ on $\SJ$.  Then there is a $\Cmeas$-measurable $F \colon M[\NN]\to \mathcal M(\cstar{A})$ that is a lift of $\iso$ on $\SJ$.
  \end{lemma}
  
\begin{proof}

    Define
    \[
      \Gamma = \set{(x,y)}{\forall n\in\NN\quad \norm{\pi(y - G_{1/n}(x))} \le 1/n}
    \]
As each $G_\epsilon$ is Borel-measurable, $\Gamma$ is Borel.  Moreover, if $A\in\SJ$ and $x\in M[A]$, then for any choice of a lift $y$ of $\iso(\pi(x))$,  we have $(x,y)\in \Gamma$.  Let $F$ be a $\Cmeas$-measurable uniformization of $\Gamma$ according to Theorem~\ref{thm:STPrel.JVN}.  Then $F$ is a lift of $\iso$ on $M[A]$ for any $A\in\SJ$.

  \end{proof}

 Now we have the necessary tools to connect $\SI_C$, and hence $\SI$, to $\SI^\e_C$;

\begin{lemma}    \label{lemma:epsilon-C->C}

    $\bigcap_{\e > 0} \SI^\e_C = \SI$.

  \end{lemma}

  \begin{proof}

    The inclusion $\supseteq$ is clear.  For the other inclusion, let $A\in\bigcap_{\e>0}\SI^\e_C$ and fix for each $\e > 0$ a $\Cmeas$-measurable map $F_\e$ which forms an $\e$-lift of $\iso$ on $M[A]$.  Working in $\mathcal P(A)$ and applying Lemma~\ref{lemma:sigma-baire->borel} with $F_n = F_\e$ for every $n$ and $\SJ = \mathcal P(A)$, we get a Borel-measurable $G_{24\e}$ which is a $24\e$-lift of $\iso$ on $M[A]$.  Applying Lemma~\ref{lemma:epsilon-borel->C-measurable}, again with $\SJ = \mathcal P(A)$, we get a $\Cmeas$-measurable lift of $\iso$ on $M[A]$, and by Lemma~\ref{lemma:C->asymptotic}, an asymptotically additive lift of $\iso$ on $M[A]$.

  \end{proof}

  \begin{lemma}    \label{lemma:oca->ccc/fin}

    Assume $\OCA_\infty + \MA_{\aleph_1}$.  Then $\SI$ meets every uncountable, treelike, a.d. family $\SA$. Moreover, it is ccc/Fin, and therefore nonmeager.
  \end{lemma}

  \begin{proof}
We work by contradiction. 
Let $\SA$ be an uncountable, a.d. family which is disjoint from $\SI$.  By Lemma~\ref{lemma:C->asymptotic}, $\SA$ is disjoint from $\SI_C$, and by  Lemma~\ref{lemma:epsilon-C->C}, there is an $\e > 0$ such that $\SI^\e_C$ is disjoint from an uncountable subset of $\SA$.  Without loss of generality we will assume $\SI^\e_C$ and $\SA$ are disjoint.  By $\MA_{\aleph_1}$, there is  an uncountable, a.d. family $\SB$ such that for every $B\in\SB$, there are infinitely-many $A\in\SA$ with  $A\subseteq^* B$ (see, for example,~\cite[Claim 4.10]{McKenney.UHF}). By Lemma~\ref{lemma:oca->treelike}, and $\OCA_\infty$, for all   but countably-many $B\in\SB$ there is a $\sigma$-$\e/100$-lift of $\iso$ on $M[B]$ consisting of $\Cmeas$-measurable  functions.  By Lemma~\ref{lemma:sigma->single}, for each such $B\in\SB$, there is an $A\in\SA$ such that $\iso$ has a $\Cmeas$-measurable $\e$-lift of $\iso$ on $M[A]$.  This is a contradiction.

To prove that $\SI$ is cc/Fin, note that, by~\cite[Lemma~2.3]{Velickovic.OCAA} and $\MA_{\aleph_1}$, Lemma~\ref{lemma:oca->ccc/fin} can actually be extended to include all uncountable, a.d. families $\SA$. As every ccc/Fin ideal containing all finite sets is nonmeager, we have the thesis.
  \end{proof}

  \begin{proof}[Proof of Theorem~\ref{thm:FA.liftingtheorem}]

   Let $\SJ=\SI$, the set of all $A\subseteq\NN$ on which $\iso$ has a lift which is   asymptotically additive.  Lemma~\ref{lemma:C->asymptotic} and Lemma~\ref{lemma:oca->ccc/fin} imply that $\SI$ is ccc/Fin.

Since $\SI$ is ccc/Fin, the first alternative of Lemma~\ref{lemma:oca->alternative} must fail.  The second alternative implies, by Lemma~\ref{lemma:sigma-baire->borel}, that for every $\e > 0$ there is a Borel-measurable map which is an $\e$-lift of $\iso$ on $M[A]$ for every $A\in\SI$.  By Lemma~\ref{lemma:epsilon-borel->C-measurable}, it follows that there is a $\Cmeas$-measurable $F$ which lifts $\iso$ on $M[A]$, for every $A\in\SI$.  Lemma~\ref{lemma:C->asymptotic} gives an $\alpha$ as required.
\end{proof}

\section{A lifting Theorem II: proofs}\label{s:FA.Lift2}

In this section we prove the lemmas needed in \S\ref{s:FA.Lift1} for the proof of Theorem \ref{thm:FA.liftingtheorem}. The $\Cstar$-algebra $\cstar{A}$ and $\iso$ are fixed as in \S\ref{s:FA.Lift1}. Similarly $\SI$, $\SI^\e$ and $\SI_C^\e$ are as in Definition \ref{defin:theideals}.

 For each $n\in\NN$, let $X_n=X_{k(n),n}$ where the sequence $X_{n,k}$ was fixed before Definition~\ref{defin:FA.Skel}, that is, $X_{k(n),n}$ is a $2^{-n}$-dense subset of the unit ball of $M_{k(n)}$ containing both $0$ and $1$.
 
  Let $\SX = \prod X_n$, and $\SX[A] = \SX\cap M[A]$ for each $A\subseteq\NN$.  $\SX$, with the subspace (strict) topology,  is homeomorphic to the set of branches through a finitely-branching tree, with the Cantor-space topology.  We will work with just $\SX$ instead of $M[\NN]=\prod_n M_{k(n)}$. In fact, if $a\in M[\NN]$ with $\norm{a}\leq 1$ there is $x\in\SX$ with $\pi(a)=\pi(x)$.

  We view elements of $\SX[A]$ as functions, with domain $A$.  Hence if $A\subseteq B$ and $x\in\SX[A]$,  $y\in\SX[B]$, then $x\subseteq y$ means that $y$ extends $x$, or in other words that $y_n = x_n$ whenever $x_n\neq 0$.  If $x$ and $y$ have a common extension we will denote the minimal one by $x\cup y$. If $A\subseteq B$ and $x\in \SX[B]$, we will denote by $x\rs A$ the element of $\SX[A]$ which is extended  by $x$ (that is, $x P_A$). We make use of Theorem \ref{thm:STPrel.talagrand}, both applied to $\SX$ and to $\mathcal P(\NN)$.

We can now proceed with the proofs required.
 \begin{lemma}
    Assume $\OCA_\infty$, let $\e > 0$ and $\SA\subseteq\mathcal P(\NN)$ be a treelike, a.d. family.  Then for all but  countably-many $A\in\SA$, there is a $\sigma$-$\e$-lift of $\iso$ on $M[A]$ consisting of $\Cmeas$-measurable functions.
  \end{lemma}
  \begin{proof}
  Fix  $\e>0$ and $\SA$ as required and an arbitrary lift $F$ of $\iso$ such that $\norm{F(x)} \le \norm{x}$ for all $x\in M[\NN]$.

 Fix an increasing approximate identity $\{e_n\}_{n\in\NN}$ for $\cstar{A}$ (see \S\ref{ss:CstarPrel.approxid}).  To simplify the notation, if $a,b\in\mathcal M(\cstar{A})$, $m\in\NN$ and $\delta > 0$ we write $a \sim_{m,\delta} b$ for $\norm{(1 - e_m)(a - b)(1-e_m)} \leq \delta$, and $a \sim_\delta b$ for $\norm{a - b} \leq \delta$.

Fix a bijection $f \colon \NN\to 2^{<\omega}$ witnessing that $\SA$ is treelike, and  for each $A\subseteq\NN$, let $\tau(A) = \bigcup f[A]$, the branch containing the image of $A$.  Note that for any $B\in\SA$ and any infinite subset $A$ of $B$, $\tau(B) =\tau(A)\in 2^\NN$.  

Let $\SR$ be the  set of all pairs $(A,x)$ such that for some $B\in\SA$, $A$ is an infinite subset of $B$, and $x\in\SX[A]_{\leq 1}$. We define colorings  $[\SR]^2 = K_0^m\cup K_1^m$ ($m\in\NN$) by placing $\{(A,x),(B,y)\}\in K_0^m$ if and only if

\begin{enumerate}[label=($K$-\arabic*)]
\item\label{K0:branches}  $\tau(A)\neq \tau(B)$,
 \item\label{K0:equal}  $x\rs (A\cap B) = y\rs (A\cap B)$ and
      \item\label{K0:unequal} $F(x) F(P_B) \not\sim_{m,\e} F(P_A) F(y)$, or $F(P_B) F(x) \not\sim_{m,\e} F(y) F(P_A)$.
    \end{enumerate}
 
 Note that $K_0^m\supseteq K_0^{m+1}$ for every $m$. We give $\SR$ the separable metric topology obtained by identifying $(A,x)\in\SR$ with the tuple
    \[
      (A,\tau(A),x,F(x),F(P_A))\in \mathcal P(\NN)\times\mathcal P(\NN)\times \SX\times \mathcal M(\cstar{A})_{\leq 1}\times\mathcal M(\cstar{A})_{\leq 1},
    \]
where $\mathcal P(\NN)$ is endowed with the Cantor set topology from $2^\NN$ and $\SX$ and $\mathcal M(\cstar{A})$ with the strict topology.
    \begin{claim}

      For every $m\in\NN$, $K_0^m$ is open.

    \end{claim}

    \begin{proof}

      Suppose $\{(A,x),(B,y)\}\in K_0^m$.  By~\ref{K0:branches}, there is some $n$ such that $\tau(A)\rs n \neq \tau(B)\rs     n$.  Let $s = f^{-1}(2^n)$; then $A\cap B \subseteq s$.  By~\ref{K0:unequal} we may also suppose that for some $\delta > 0$ and $p\in\NN$, either
      \[
        \norm{e_{p}(1 - e_m)(F(x) F(P_B) - F(P_A) F(y))(1-e_m)} > \e + \delta
      \]
      or
      \[
        \norm{e_{p}(1 - e_m)(F(P_B) F(x) - F(y) F(P_A))(1-e_m)} > \e + \delta
      \]
 Now, let $(\bar{A},\bar{x})$ and $(\bar{B},\bar{y})$ be  elements of $\SR$ such that   
 \[
 A\cap s = \bar{A}\cap s, B\cap s = \bar{B}\cap s, x\rs s = \bar{x}\rs s\text{ and }y\rs s = \bar{y}\rs s,
 \] and
      \begin{align*}
         \norm{e_{p}(1 - e_m)F(P){\bar B})(F(x) - F(\bar{x}))(1-e_m)} &+ \norm{e_{p}(1 - e_m)F(\bar y)(F(P_A) - F(P_{\bar{A}}))(1-e_m)}+ \\
     +     \norm{e_{p}(1 - e_m)F(P_{\bar{A}})(F(y)-F(\bar{y}))(1-e_m)} &+ \norm{e_{p}(1 - e_m)F(\bar{x})(F(P_B) - F(P_{\bar{B}}))(1-e_m)} < \delta
      \end{align*}
The set of such pairs $\{(\bar{A},\bar{x}), (\bar{B},\bar{y})\}$ is an open neighborhood of  $\{(A,x), (B,y)\}$ in $[\SR]^2$, and for each such pair we have  that $\{(\bar{A},\bar{x}), (\bar{B},\bar{y})\}\in K_0^m$.

\end{proof}

Recall that, for $a,b\in 2^{\NN}$, $\Delta(a,b)=\min\{n\mid a(n)\neq b(n)\}$.

\begin{claim}

The first alternative of $\OCA_\infty$ fails for the colors $K_0^m$ ($m\in\NN$), that is, there is no uncountable $Z\subseteq 2^\NN$, and an injection $\zeta \colon Z\to \SR$ such that, for all distinct $a,b\in Z$, $\{\zeta(a),\zeta(b)\}\in K_0^{\Delta(a,b)}$.

\end{claim}

\begin{proof}

Suppose otherwise and let $Z$ and $\zeta\colon Z\to\SR$ as above  Let $\SH = \zeta[Z]$ and put
      \[
        z = \bigcup\set{x}{(A,x)\in\SH}
      \]
      By~\ref{K0:equal}, $z\in\SX$ and $z\rs A = x$ for all $(A,x)\in\SH$, therefore $z$ is well defined.  For all $(A,x)\in\SH$, $\pi(F(z)  F(P_A)) = \pi(F(P_A) F(z)) = \pi(F(x))$, hence there is $m\in\NN$ such that
      \[
       (1-e_m) F(z)F(P_A) \sim_{\e/4} (1-e_m)F(x) \sim_{\e/4}(1-e_m) F(P_A)F(z),
      \] 
     and
     \[
        F(z)F(P_A)(1-e_m) \sim_{\e/4} F(x)(1-e_m) \sim_{\e/4} F(P_A)F(z)(1-e_m).
    \]
     
        By the Pigeonhole principle, refining $Z$ to an uncountable subset (which we will still call $Z$), we may assume that      there is some fixed $m\in\NN$ such that the above holds for all $(A,x)\in\SH = \zeta[Z]$.  Since $Z$ is uncountable, we may find $a,b\in Z$ such that $\Delta(a,b)\ge m$.  Let $(A,x) = \zeta(a)$ and $(B,y) = \zeta(b)$.  Then, we have
      \[
        F(x)F(P_B) \sim_{m,\e/4} F(P_A)F(z)F(P_B) \sim_{m,\e/4} F(P_A)F(y)
      \]
      which implies $F(x)F(P_B) \sim_{m,\e/2} F(P_A)F(y)$, and similarly,
\[
        (1 - e_m) F(P_B)F(x) (1-e_m)\sim_{\e/4} (1-e_m)F(P_B)  F(z) F(P_A)(1 - e_m)  \sim_{\e/4}  (1 - e_m)F(y) F(P_A)(1-e_m) 
\]
   a contradiction to $\{(A,x),(B,y)\}\in K_0^m$.

    \end{proof}

  The second alternative of $\OCA_\infty$ must hold.  Let $(\SH_m)_{m\in\NN}$ be  some sequence of sets with 
\[
\SR = \bigcup\SH_m\text{ and }[\SH_m]^2 \subseteq K_1^m.
\]
  For each    $m\in\NN$, let $\SD_m$ be a countable subset of $\SH_m$ which is dense in $\SH_m$ with respect to the topology on   $\SR$ described above.  Fix $T\in\SA$ such that $T\neq \tau(D)$ for all $D\in\set{A}{\exists x,m ((A,x)\in\SD_m)}$. We will show that there is a $\sigma$-$\e$-lift of $\iso$ on $M[T]$ consisting of $\Cmeas$-measurable functions.

    For each $A\subseteq\NN$ and $m\in\NN$, define $\Lambda_m^A$ to be the set of pairs $(x,z)$ such that $x\in\SX[A]$,   $z\in\mathcal M(\cstar{A})_{\le 1}$, and for all $n\in\NN$ and $\delta > 0$, there is some $(B,y)\in\SD_m$ such that
 
   \begin{enumerate}
      \item\label{lem:notreelike.cond1}  $x\rs (A\cap B) = y\rs (A\cap B)$,
      \item\label{lem:notreelike.cond2}  $A\cap n = B\cap n$, and
      \item\label{lem:notreelike.cond3}  $e_n F(P_A) \sim_{\delta} e_n F(P_B)$ , $e_n F(x)F(P_A) \sim_{\delta} e_n F(x)F(P_B)$ and $e_n z \sim_{\delta} e_n F(y)$.
    \end{enumerate}
Since $\SD_m$ is countable, each $\Lambda_m^A$ is Borel.
  
  \begin{claim}\label{lem:notreelike.partition}

      There is a partition $T = T_0\cup T_1$ such that for each $i < 2$, $m\in\NN$, and $x\in\SX[T_i]$, if      $(T_i,x)\in\SH_m$, then $(x,F(x))\in \Lambda_m^{T_i}$.  Moreover, if $(x,z)\in\Lambda_m^{T_i}$ and      $(T_i,x)\in\SH_m$, then $F(x)F(P_{T_i}) \sim_{m,\e} z F(P_{T_i})$.

    \end{claim}

    \begin{proof}

We identify $n$ with the set $\{0,\ldots,n-1\}$. Fix $\{b_i\}_{i\in\NN}$ an enumeration of a dense subset of $\cstar{A}$ and $k\in\NN$. If $s\subseteq k$,      $t\in\SX[k]$ and $m,n,p,q,r < k$, then for any $(A,x)\in\SH_m$ with 
\[
A\subseteq T,\,A\cap k = s, \,x\rs k = t,\, e_n F(P_A) \sim_{1/k} b_p, \, e_nF(x)F(P_A)\sim_{1/k} b_r\text{ and }e_n F(x) \sim_{1/k} b_q,
\]
 by density of $\SD_m$ in $\SH_m$, we may find some  $(B,y)\in\SD_m$ such that $B\cap k = s$, $y\rs k = t$, $e_n F(P_B) \sim_{1/k} b_p$, $e_nF(x)F(P_B)\sim_{1/k} b_r$ and $e_n F(y) \sim_{1/k}      b_q$. This is because each of the sets
 \[
\{a\in\mathcal M(\cstar{A})_{\leq 1}\mid e_na\sim_{1/k}b_p\},\, \{a\in\mathcal M(\cstar{A})_{\leq 1}\mid e_na\sim_{1/k}b_q\}\text{ and } \{a\in\mathcal M(\cstar{A})_{\leq 1}\mid e_nF(x)a\sim_{1/k}b_r\}
 \] is open in the strict topology, and by density of $\SD_m$. 
 
  As the set of all such tuples $(s,t,m,p,q,r)$ is finite, $(B,y)$ can be chosen from some fixed  finite set $\SF_k \subseteq \bigcup\set{\SD_m}{m < k}$.  Note that, for any $(B,y)\in\SF_k$, we have that $T\cap B$ is finite. As $\SF_k$ is finite, there is $k^+ > k$ such that $T\cap B \subseteq k^+$      for all $(B,y)\in\SF_k$.
      
 We recursively construct a sequence $0 = k_0 < k_1 < \cdots$, by setting $k_{i+1} = k_i^+$ for each $i\in\NN$.      Let $T_0 = T\cap \bigcup_i[k_{2i+1},k_{2i+2})$ and $T_1 = T\sm T_0$.  Suppose $x\in\SX[T_0]$, and      $m\in\NN$ is such that $(T_0,x)\in\SH_m$.  Let $n\in\NN$ and $\delta > 0$ be given and choose $i\in\NN$ large enough      that $1 / k_{2i} < \delta / 2$, $m,n < k_{2i}$, and for some $p,q,r < k_{2i}$, we have $e_n F(P_{T_0}) \sim_{1/      k_{2i}} b_p$, $e_nF(x)F(P_{T_0})\sim_{1/k} b_r$ and $e_n F(x) \sim_{1 / k_{2i}} b_q$. 

By our choice of $\SF_{k_{2i}}$ we may find      $(B,y)\in\SF_{k_{2i}}$ such that 
\[
T_0\cap k_{2i} = B\cap k_{2i},\,\,x\rs k_{2i} = y\rs k_{2i},\,\, e_n F(P_B) \sim_{1      / k_{2i}} b_p,\, e_nF(x)F(P_B)\sim_{1/k_{2i}} b_r\text{ and }e_n F(y) \sim_{1 / k_{2i}} b_q.
\]
  All that remains to check to have $(x,F(x))\in\Lambda_m^{T_0}$      is that $x\rs (T_0\cap B) = y\rs (T_0\cap B)$.  To see this, note that by definition of $k_{2i+1} = k_{2i}^+$, we      have $T_0\cap B \subseteq k_{2i+1}$; but since $T_0\cap [k_{2i},k_{2i+1}) = \emptyset$, it follows that $T_0\cap      B\subseteq k_{2i}$, and since $x\rs k_{2i} = y\rs k_{2i}$, this implies that $x\rs (T_0\cap B) = y\rs (T_0\cap      B)$.  The same argument shows that if $(T_1,x)\in\SH_m$, then $(x,F(x))\in\Lambda_m^{T_1}$.
      
To prove the second assertion, suppose that $(x,z)\in\Lambda_m^{T_i}$, $(T_i,x)\in\SH_m$ and $\delta > 0$. Choose      $n$ large enough so that $\norm{[e_n, 1 - e_m]} < \delta.$ Since $(x,z)\in\Lambda_m^{T_i}$, we may choose $(B,y)\in \SD_m$ satisfying conditions \ref{lem:notreelike.cond1}--\ref{lem:notreelike.cond3} preceding this claim with $A =      T_i$.  Since $(T_i,x),(B,y)\in\SH_m$, $\tau(T_i) = T \neq \tau(B)$ and $x\rs (T_i\cap B) = y\rs (T_i\cap B)$.  By the $K_1^m$-homogeneity of $\SH_m$ we have that $F(P_{T_i}) F(y)\sim_{m,\e} F(x) F(P_B)$ and $F(y) F(P_{T_i})      \sim_{m,\e} F(P_B) F(x)$.

      Now,
      \begin{align*}
        e_n (1 - e_m) F(x) F(P_{T_i})(1 - e_m) & \sim_{4\delta} e_n (1 - e_m) F(x) F(P_B)(1 - e_m) \\
          & \sim_\e e_n (1 - e_m) F(y) F(P_{T_i})(1 - e_m) \\
          & \sim_\delta (1 - e_m) e_n F(y) F(P_{T_i})(1 - e_m) \\
          & \sim_\delta (1 - e_m) e_n z F(P_{T_i})(1 - e_m) \\
          & \sim_\delta e_n (1 - e_m) z F(P_{T_i})(1 - e_m)
      \end{align*}
      and hence $e_n (1 - e_m) F(x) F(P_{T_i})(1 - e_m) \sim_{7\delta + \e} e_n (1 - e_m) z F(P_{T_i})(1 - e_m)$.  Since this holds for      all sufficiently large $n$ and all $\delta > 0$, we have that $F(x) F(P_{T_i}) \sim_{m,\e} z F(P_{T_i})$ as desired.

    \end{proof}

    Let $F_m^i$ be a $\Cmeas$-measurable uniformization of $\Lambda_m^{T_i}$ given by Theorem \ref{thm:STPrel.JVN} and $G^i_m(x)=F_m^i(x)P(T_i)$. Since $F_m^i$ is $\Cmeas$-measurable and $P(T_i)$ is fixed, $G_m^i$ is $\Cmeas$-measurable.  By the above claim, and the fact that $\SR =    \bigcup\SH_m$, it follows that for every $x\in\SX[T_i]$, there is some $m$ such that $G_m^i(x)$ is    defined, and moreover 
\[
\norm{\pi(G_m^i(x)) - \iso(\pi(x))} \le \e.
\]
 This shows that the desired conclusion holds for    each $T_i$.  Repeating the argument from Lemma~\ref{lemma:ideals}, we see that it holds for $T = T_0\cup    T_1$ as well.

  \end{proof}
  \begin{lemma}

    Let $\e > 0$ and $A\subseteq\NN$. Suppose that there is a $\sigma$-$\e$-lift of $\iso$ on $M[A]$, consisting of $\Cmeas$-measurable functions and that $A = \bigcup_n A_n$ is a partition of $A$ into infinite sets.  Then there is some $n$ such that $A_n\in \SI_C^{4\e}$.
  \end{lemma}
  \begin{proof}

We work by contradiction and suppose that $A_n\notin \SI_C^{4\e}$ for all $n$.  Fix an arbitrary lift $F$ of $\iso$ on $M[A]$ and let $\{F_n\}$ be a $\sigma$-$\epsilon$-lift of $\iso$ on $M[A]$ as in the hypothesis.    Since each $F_n$ is Baire-measurable, it follows that there is a comeager subset $\SG$ of $\SX$ (recall that $\SX=\prod X_n$ was defined at the beginning of \S\ref{s:FA.Lift2}) on which every $F_n$   is continuous. Thanks to Theorem~\ref{thm:STPrel.talagrand}, we can find a partition of interval $E_i$ and $t_i$ such that $y\in\SG$ whenever $\exists^\infty i (y\restriction E_i=t_i)$. For $i<2$, put  
\[
E^i = \bigcup_{k}E_{2k+i}\,\,\,\text{and}\,\,\,t^i = \bigcup_kt_{2k+i}.
\]
   Define
    \[
      F_n'(x) = F_n(x\rs E^0 + t^1) - F_n(t^1) + F_n(x\rs E^1 + t^0) - F_n(t^0).
    \]
 Each $F_n'$ is continuous on all of $\SX$, and the functions $F_n'$ form a $\sigma$-$2\e$-lift of $\iso$ on $M[A]$.  We will write $F_n = F_n'$ in the following.
    
    For each $m\in\NN$, let $B_m = \bigcup_{n>m}A_n$.  We will construct sequences
    \begin{itemize}
      \item  $x_n \in\SX[A_n]$,
      \item  $C_n\subseteq B_n$, and
      \item  $z_n \in\SX[C_n]$,
    \end{itemize}
    such that for all $n < m$,

    \begin{enumerate}
      \item  $A_m\sm C_n\not\in\SI^{4\e}_C$,
      \item  $C_n\cap B_m \subseteq C_m$,
      \item  $z_n\rs (C_n\cap C_m) \subseteq z_m$,
      \item  $z_{n-1}\rs (C_{n-1}\cap A_n) \subseteq x_n$, and
      \item  for all $y\in\SX[B_n]$, if $y\supseteq z_n$, then
      \[
        \norm{\pi(F_n(x_0 \cup \cdots \cup x_n \cup y) - F(x_n))\pi(F(P_{A_n}))} > 2\e
      \]
    \end{enumerate}
    The construction goes by induction on $n$.  Suppose we have constructed $x_k$, $C_k$ and $z_k$ for $k<n$.  For each $x\in\SX[A_n]$ and $y\in\ball{\mathcal M(\cstar{A})}$, define $\SE_n(x,y)$ to be the set of all $z\in\SX[B_n\sm C_{n-1}]$ such that 
    \[
      \norm{\pi(F_n(x_0 \cup \cdots \cup x_{n-1} \cup x \cup z_{n-1}\cup z) - y)\pi(F(P_{A_n}))} \le 2\e
    \]
  Since $F_n$ is continuous, $\SE_n(x,y)$ is Borel, for every $x$ and $y$.

    \begin{claim}

      There is some $x\in\SX[A_n\sm C_{n-1}]$ such that $\SE_n(x,F(x))$ is not comeager.

    \end{claim}

    \begin{proof}

      Suppose otherwise.  Let 
\[
\SR=\set{(x,y)\in\SX[A_n\sm C_{n-1}]\times\ball{\mathcal M(\cstar{A})}}{\SE_n(x,y)\text{  is comeager}}.
\]
$\SR$ is analytic and, applying Theorem ~\ref{thm:STPrel.JVN}, it has a $\Cmeas$-measurable uniformization $G$.  For all $x\in\SX[A_n]$      extending $z_{n-1}\rs (C_{n-1}\cap A_n)$, then, $\SE_n(x,G(x))$ and $\SE_n(x,F(x))$ are both comeager, and hence must intersect;      so for all such $x$,
      \[
        \norm{\pi(F(x) - G(x))\pi(F(P_{A_n}))} \le 4\e
      \]
      Then the map $x\mapsto G(x)F(P_{A_n})$ is a $\Cmeas$-measurable $4\e$-lift of $\iso$ on $\SX[A_n\sm C_{n-1}]$, a contradiction.

    \end{proof}

 Choose $x\in\SX[A_n\sm C_{n-1}]$ as in the claim, and let 
\[
x_n = x\cup (z_{n-1}\rs (C_{n-1}\cap A_n)).
\]
 Since $\SE_n(x,F(x))$ is Borel, there is some finite $a\subseteq B_n\sm C_{n-1}$ and some $\sigma\in\SX[a]$    such that the set of $z\in\SE_n(x,F(x))$ extending $\sigma$ is meager.  Applying Theorem~\ref{thm:STPrel.talagrand}, we    may find a partition of $B_n\sm (a\cup C_{n-1})$ into finite sets $s_i$, and $u_i\in\SX[s_i]$, such that    for any $z\in\SX[B_n\sm C_{n-1}]$, if $z$ extends $\sigma$ and infinitely-many $u_i$, then $z\not\in\SE_n(x,F(x))$.

\begin{claim}

      There is an infinite set $L\subseteq\NN$ such that 
\[
A_m\sm (C_{n-1}\cup \bigcup\set{s_i}{i\in L})\not\in\SI^{4\e}_C
\]
for all  $m > n$.

    \end{claim}
    
\begin{proof}

      Recursively construct infinite sets $J_{n+1} \supseteq J_{n+2}\supseteq\cdots$ such that for each $m > n$,
      \[
        A_m\sm (C_{n-1}\cup \bigcup\set{s_i}{i\in J_m})\not\in\SI^{4\e}_C
      \]
  using the fact that $A_m\sm C_{n-1}\not\in\SI^{4\e}_C$ for all $m > n$.  Any infinite $L$ such  that $I\subseteq^* J_m$ for all $m > n$ satisfies the claim.

    \end{proof}
    
Let $L$ be as in the claim, and put $C_n = B_n\cap (C_{n-1} \cup \bigcup\set{s_i}{i\in L})$.  Let $z_n = (z_{n-1}\rs (C_{n-1}\cap    B_n))\cup \bigcup\set{s_i}{i\in L}$.  This completes the construction.

    Now let $x = \bigcup\set{x_n}{n\in\NN}$.  Then $x\in\SX[A]$, and hence there is some $n\in\NN$ such that
    \[
      \norm{\pi(F_n(x) - F(x))} \le 2\e
    \]
    Notice that if $y = \bigcup\set{x_m}{m > n}$, then $x = x_0\cup \cdots \cup x_n \cup y$, $y\in\SX[B_n]$, and $y$ extends $z_n$; hence
    \[
      \norm{\pi(F_n(x) - F(x_n))\pi(F(P_{A_n}))} > 2\e
    \]
    But we have
    \[
      \pi(F(x_n))\pi(F(P_{A_n})) = \pi(F(x))\pi(F(P_{A_n}))
    \]
    and this is a contradiction.

  \end{proof}

  \begin{lemma}
    Assume $\OCA_\infty + \MA_{\aleph_1}$.  Then either
    \begin{enumerate}[label=(\Roman*)]
      \item\label{alt:I}  there is an uncountable, treelike, a.d. family $\SA\subseteq\P(\NN)$ which is disjoint from  $\SI$, or
      \item\label{alt:II}  for every $\e > 0$, there is a sequence $\{F_n\}_{n\in\NN}$ of $\Cmeas$-measurable functions such that for every  $A\in\SI^\e$, there is an $n$ such that $F_n$ is an $\e$-lift of $\iso$ on $M[A]$.
    \end{enumerate}
  \end{lemma}
  \begin{proof}

    For each $A\in\SI$, we may find an asymptotically additive $\alpha^A$ such that $\alpha^A$ is a lift of $\iso$ on    $M[A]$.  
Recall that $X_n$ was chosen to be a finite $2^{-n}$ dense subset of the unital ball of $M_{k(n)}$. Without loss of generality, we may assume that for every $n\in\NN$ and $A\in\SI$, $\alpha_n^A$ is    determined by the set $X_n$, in the sense that, for some fixed linear order $<$ of $X_n$, for all $x\in X_n$ we have    that $\alpha_n^A(x)$ is equal to $\alpha_n^A(y)$, where $y$ is the $<$-minimal element of $X_n$ such that $\norm{x -    y} \le 2^{-n}$, for some well-order $<$ on $X_n$.  In this way we view each $\alpha^A$ as an element of the Polish space
    \[
      \Fn(\SX,\cstar{A}_{\leq 1}) = \prod_n (\cstar{A}_{\le 1})^{X_n}
    \]
    where the topology is given by the product topology, and each $\cstar{A}_{\leq 1}$ is considered in the norm topology. We will assume that $\alpha_n^A = 0$ whenever $n\not\in A$.

    Fix $\e > 0$, and define colorings $[\SI]^2 = K_0^m \cup K_1^m$ by placing $\{A,B\}\in K_0^m$ if and only if there are $E_0, \ldots, E_{m-1}$ pairwise disjoint, finite subsets of $(A\cap B)\sm m$ such that for all $i <  m$, there is some $x^i\in\SX[E_i]$ with
    \[
      \norm{\sum_{n\in E_i} \alpha_n^A(x^i_n) - \sum_{n\in E_i} \alpha_n^B(x^i_n)} > \e.
    \]
    Define a separable metric topology on $\SI$ by identifying $A\in\SI$ with the pair $(A,\alpha^A)\in   \mathcal P(\NN)\times\Fn(\SX,\cstar{A}_{\leq1})$.  In the corresponding topology on $[\SI]^2$, each color $K_0^m$ is open and  $K_0^m\supseteq K_0^{m+1}$ for each $m\in\NN$.

 Suppose that the first    alternative of $\OCA_\infty$ holds and fix an uncountable $Z\subseteq 2^\NN$ and a map $\zeta \colon Z\to\SI$    such that for all $x,y\in Z$, $\{\zeta(x),\zeta(y)\}\in K_0^{\Delta(x,y)}$.
 We will define a poset $\mathbb P$ with the intent to form a treelike, a.d. family which is disjoint from $\SI^\e$.  The    conditions $p\in\PP$ are of the form $p = (I_p, G_p, n_p, s_p, x_p, f_p)$, where

    \begin{enumerate}[label=($\PP$-\arabic*)]
      \item\label{PP-1}  $I_p\in [\omega_1]^{<\omega}$, $n_p\in\NN$, $G_p \colon I_p \to [Z]^{<\omega}$, $s_p \colon I_p\times n_p\to 2$,      $x_p \colon I_p\to \SX[n_p]$, and $f_p \colon n_p\to 2^{<\omega}$,
      \item\label{PP-3} if for all $\xi\in I_p$ and $m,n\in n_p$ we have that $s_p(\xi,m) = s_p(\xi,n) = 1$, then $f_p(m)$ and      $f_p(n)$ are comparable, ($x_p(\xi,n)$ is the $n$-th coordinate of $x_p(\xi)$) and
      \item\label{PP-2}  for all $\xi\in I_p$ and distinct $A,B\in \zeta''(G_p(\xi))$,
      \[
        \exists E\subseteq\set{n < n_p}{s_p(\xi,n) = 1}\quad \norm{\sum_{n\in E} \alpha^A_n(x_p(\xi,n)) - \sum_{n\in E} \alpha^B_n(x_p(\xi,n))} > \e
      \]
    \end{enumerate}
    We let $p\le q$ if and only if

    \begin{enumerate}[label=($\le$-\arabic*)]
      \item  $I_p\supseteq I_q$, $n_p\ge n_q$, $s_p\supseteq s_q$, $f_p \supseteq f_q$, and for all $\xi\in I_q$, $G_p(\xi)\supseteq      G_q(\xi)$.
      \item  for all $(\xi,n)\in I_q\times n_q$ we have $x_p(\xi,n) = x_q(\xi,n)$.
      \item  for all $m,n\in [n_q, n_p)$, if there are distinct $\xi,\eta\in I_q$ such that $s_p(\xi,m) = s_p(\eta,n) =      1$, then $f_p(m) \perp f_p(n)$.
    \end{enumerate}
 The    sets $S_\xi$ in the family will be approximated by the functions $s_p(\xi,\cdot)$, and the function witnessing that    the family is treelike will be approximated by $f_p$.  The sets $G_p(\xi)$ will form a large $K_0^1$-homogeneous    set, with common witness approximated by $x_p(\xi)$, which will be used to show that $S_\xi\not\in\SI^\e$.

    \begin{claim}\label{claim:provingisccc}
      $\PP$ is ccc.
    \end{claim}
    \begin{proof}
      Let $\SQ\subseteq\PP$ be uncountable.  By refining $\SQ$ to an uncountable subset, we may assume that the      following hold for $p\in\SQ$. (Recall that a $\Delta$-system is a family $\SF$ of finite sets for which there is root $r$ such that whenever $x,y\in\SF$ we have $x\cap y=r$, see \S\ref{sss:cccDeltasystem} and the $\Delta$-systema Lemma~\ref{lem:DeltaSystem}).

      \begin{enumerate}
        \item  There are $N\in\NN$ and $f \colon  N\to 2^{<\omega}$ such that $n_p = N$ and $f_p = f$ for all $p\in\SQ$,
        \item  The sets $I_p$ ($p\in\SQ$) form a $\Delta$-system with root $J$, and the tails $I_p\sm J$ have the same        size $\ell$, for all $p\in\SQ$.
        \item  For each $\xi\in J$, the sets $G_p(\xi)$ ($p\in\SQ$) form a $\Delta$-system with root $G(\xi)$, and the        tails $G_p(\xi)\sm G(\xi)$ all have the same size $m(\xi)$.
        \item  There are functions $t \colon J\times N\to 2$ and $y \colon J\to \SX[N]$ such that for all $(\xi,n)\in J\times N$ and all $p\in\SQ$,        $s_p(\xi,n) = t(\xi,n)$ and $x_p(\xi,n) = y(\xi,n)$.
        \item\label{condition-u} If $I_p\sm J = \{\xi_0^p < \cdots < \xi_{\ell-1}^p\}$, then the map $u \colon \ell\times N\to 2$ given by
        \[
          u(i,n) = s_p(\xi_i,n)
        \]
        is the same across all $p\in\SQ$.
        \item\label{condition-large-Delta}  If $\xi\in J$ and $G_p(\xi)\sm G(\xi) = \{z_0^p(\xi),\ldots,z_{m(\xi)-1}^p(\xi)\}$, then for all        $p,q\in\SQ$ and $i < m(\xi)$, we have $\Delta(z_i^p(\xi),z_i^q(\xi)) \ge M$, where $M = \max\{N,        \sum_{\xi\in J}m(\xi)\}$.
      \end{enumerate}

      Let $p,q\in\SQ$ be given; we claim that $p$ and $q$ are compatible.  We define an initial attempt at an      amalgamation $r = (I_r,G_r,n_r,s_r,x_r,f_r)$ as follows.  Let $I_r = I_p\cup I_q$, $n_r = N$, $f_r = f$, $s_r =  s_p\cup s_q$, and $x_r = x_p\cup x_q$, and for each $\xi\in I_r$, we let $G_r(\xi) = G_p(\xi)\cup G_q(\xi)$.  If  $r$ were in $\PP$, then we would have $r\le p,q$, as required; however, condition~\ref{PP-2} may not be satisfied  by $r$.

      It is easily verified that the following cases of condition~\ref{PP-2} are in fact already satisfied by $r$;
    
  \begin{itemize}
        \item  $\xi\not\in J$,
        \item  $\xi\in J$ and $A,B\in\zeta[G(\xi)]$, and
        \item  $\xi\in J$ and $A = \zeta(z_i^p(\xi))$, $B = \zeta(z_j^q(\xi))$, where $i\neq j$.
      \end{itemize}
      (The first two cases simply use the fact that $p,q\in\PP$; the last case uses, in addition,~\eqref{condition-u}      above.)  For the last remaining case, fix $\xi\in J$ and $i < m(\xi)$, and put $A = \zeta(z_i^p(\xi))$, $B =      \zeta(z_i^q(\xi))$.  By~\eqref{condition-large-Delta}, we have $\{A,B\}\in K_0^M$, hence there are      $M$-many pairwise-disjoint, finite subsets $E_i$ of $(A\cap B)\sm M$, such that 
      \[
        \exists x^i\in \SX[E_i] \quad \norm{\sum_{n\in E_i} \alpha_n^A(x^i_n) - \sum_{n\in E_i} \alpha_n^B(x^i_n)} > \e
      \]     
 Since $M\ge \sum_{\xi\in J}m(\xi)$, we may choose pairwise disjoint, finite sets $E(\xi,i)$ for each $\xi\in      J$ and $i < m(\xi)$, such that for each $\xi\in J$ and $i < m(\xi)$, $E+i = E(\xi,i)$ satisfies the above, with $A =      \zeta(z_i^p(\xi))$ and $B = \zeta(z_i^q(\xi))$.  Let $x^{\xi,i}\in \SX[E(\xi,i)]$ be the corresponding witness.      Let $\bar{N}\ge M$ be large enough to include every set $E(\xi,i)$, and define      $s \colon I_r\times \bar{N}\to 2$, $x \colon I_r\to \SX[\bar{N}]$, and $g \colon \bar{N}\to 2^{<\omega}$ so that

      \begin{itemize}
        \item  $s\supseteq s_r$, $x\supseteq x_r$, and $g\supseteq f_r$, 
        \item  for all $\xi\in J$ and $i < m(\xi)$, and $n \in E(\xi,i)$, $s(\xi,n) = 1$ and $x(\xi,n) = x^{\xi,i}_n$,
        \item  $s(\eta,k) = 0$ and $x(\eta,k) = 0$ for all other values of $(\eta,k)\in I_r\times \bar{N}$,
        \item  for all $\xi\in J$ and 
\[
n,n'\in \bigcup_{i< m(\xi)}E(\xi,i),
\] 
$g(n)$ and $g(n')$ are comparable and extend $\bigcup\set{g(k)}{k < N\land s_r(\xi,k) = 1}$,
        \item  for all distinct $\xi,\eta\in J$, if
\[
n\in \bigcup_{i<m(\xi)}E(\xi,i)\,\,\,,\,\,\,n'\in \bigcup_{i<m(\eta)}E(\eta,i),
\]
 then $g(n)\perp g(n')$.
      \end{itemize}
      It follows that $r' = (I_r, G_r, \bar{N}, s, x, g)\in \PP$ and $r'\le p,q$, as required.

    \end{proof}

    Let $Z = \set{z_\sigma}{\sigma < \omega_1}$ be an enumeration of $Z$.  For each $\xi$, $\sigma < \omega_1$, define    $p_{\xi,\sigma}\in \PP$ by $I_{p_{\xi,\sigma}} = \{\xi\}$, $F_{p_{\xi,\sigma}}(\xi) = \{z_\sigma\}$,    $n_{p_{\xi,\sigma}} = 0$, and $s_{p_{\xi,\sigma}} = x_{p_{\xi,\sigma}} = f_{p_{\xi,\sigma}} = \emptyset$.  Since    $\PP$ is ccc, for each $\xi$ there is some $q_\xi\in\PP$ such that for all $\sigma$, the set $D_{\xi,\sigma} =    \set{p\in\PP }{\exists \tau\ge \sigma(p \le p_{\xi,\sigma})}$ is dense below $q_\xi$.  Again applying the ccc, we    may find a $q\in\PP$ such that for all $\xi$, the set $E_\xi = \set{p\in\PP}{\exists \eta\ge \xi ( p\le q_\xi)}$ is    dense below $q$.  It follows that, by $\MA_{\aleph_1}$, we may find a filter $G\subseteq\PP$ such that $G$ meets    uncountably-many $E_\xi$, and for each such $\xi$, $G$ meets uncountably-many $D_{\xi,\sigma}$.  Moreover we may    assume that $G$ meets all of the sets
    \[
      C_n = \set{p\in\PP}{n < n_p\land \forall \xi\in I_p\; \exists i\in [n,n_p)\;\; (s_p(\xi,i) = 1)}.
    \]
    Define $I = \bigcup_{p\in G}I_p$ and, for each $\xi\in I$, define
    \begin{gather*}
      S_\xi = \set{n\in\NN}{\exists p\in G\; (n < n_p\land s(\xi,n) = 1)} \\
      \SH_\xi = \zeta''\left(\bigcup_{p\in G}G_p(\xi)\right)
    \end{gather*}
    It follows from the above that $I$ is uncountable and, for uncountably-many $\xi\in I$, $\SH_\xi$ is uncountable and $S_\xi$ is infinite.  We will assume without loss of generality that for every $\xi\in I$, $\SH_\xi$ is uncountable and $S_\xi$ is infinite.  If $f = \bigcup_{p\in G}f_p$, then $f$ witnesses that $S_\xi$ ($\xi\in I$) is a treelike, a.d. family.  For each $\xi$, define $x^\xi\in\SX[\NN]$ by $x^\xi_n = x_p(\xi,n)$ for any $p\in G$ with $n < n_p$.  Notice that for any $A,B\in\SH_\xi$, we have
    \begin{gather}
      \label{eqn:homogeneity} \exists E\in [A\cap B\cap S_\xi]^{<\omega} \quad \norm{\sum_{n\in E} \alpha_n^A(x^\xi_n) -
      \sum_{n\in E} \alpha_n^B(x^\xi_n)} > \e
      \tag{$\ast$}
    \end{gather}

    \begin{claim}

      For all $\xi \in I$, $S_\xi\not\in \SI^\e$.

    \end{claim}

    \begin{proof}

      Suppose otherwise, and fix an asymptotically additive $\beta$ which is an $\e$-lift of $\iso$ on $M[S_\xi]$.      For each $A\in\SH_\xi$, since $\beta$ and $\alpha^A$ both lift $\iso$ on $M[S_\xi\cap A]$, there is some $N\in\NN$     such that for any finite $E\subseteq A\cap S_\xi\sm N$,
      \[
        \norm{\sum_{n\in E} \alpha_n^A(x^\xi_n) - \sum_{n\in E} \beta_n(x^\xi_n)} <   \frac{\e}{2}
      \]
      By the pigeonhole principle, the same $N$ works for all $A$ in an uncountable subset $\SL$ of $\SH_\xi$. Moreover, by the separability of $\cstar{A}^N$, we may find distinct $A,B\in\SL$ such that $\norm{\alpha_n^A(x_n) -     \alpha_n^B(x_n)} < \e / 2N$ for all $n\in A\cap B\cap N$, a contradiction to~\eqref{eqn:homogeneity}.

    \end{proof}
 
   This shows that the first alternative of $\OCA_\infty$ implies (in the presence of $\MA_{\aleph_1}$) that there is    an uncountable a.d. family which is disjoint from $\SI^\e$.      Now we will show that the second alternative of $\OCA_\infty$    implies~\ref{alt:II}.
    

    Suppose $\SI = \bigcup_m \SH_m$, where $[\SH_m]^2 \subseteq K_1^m$ for each $m\in\NN$.  Fix $m\in\NN$.  We will    define a $\Cmeas$-measurable function $F$ such that, for every $A\in\SH_m$, $F$ is an $\e$-lift of    $\iso$ on $M[A]$.  Let $\SD\subseteq\SH_m$ be a countable set which is dense in $\SH_m$ in the topology where $\SI$ was identified with $(A,\alpha^A)\in \mathcal P(\NN)\times \Fn(\SX,\mathcal A_{\leq 1})$ (this topology was the one  making $K_0^m$ open.)  Let $(e_n)_{n\in\NN}$ be an  approximate identity for $\cstar{A}$.  We define $\SR$ to be the subset of $M[\NN]_{\leq 1}\times \mathcal M(\cstar{A})_{\le 1}$ consisting of those    $(x,y)$ such that there is a Cauchy sequence $B_p$ ($p\in\NN$) in $\SD$ for which $y$ is the strict limit of    $\alpha^{B_p}(x)$ as $p\to\infty$, and for every $N\in\NN$, $\supp(x)\cap N\subseteq B_p$ for all large enough $p\in\NN$.  Since $\SD$ is countable, $\SR$ is analytic.  It will suffice to prove that for all $A\in\SH_m$    and all $x\in M[A]$,
    
    
\begin{enumerate}[label=$\SR$-(\arabic*)]
      \item\label{cdn:covers}  $(x,\alpha^A(x))\in \SR$, and
      \item\label{cdn:approx}  for all $y$ with $(x,y)\in \SR$, we have $\norm{\pi(y - \alpha^A(x))} \le \e$,
    \end{enumerate}
    since then any $\Cmeas$-measurable uniformization $F$ of $\SR$ will satisfy the required properties. 

 Fix $A\in\SH_m$ and    $x\in M[A]_{\leq 1}$.  Condition~\eqref{cdn:covers} follows simply from the fact that $\SD$ is dense in $\SH_m$ and    $A\in\SH_m$.
    To show~\eqref{cdn:approx}, let $y$ be given with $(x,y)\in \SR$.  Suppose
    \[
      \norm{\pi(y - \alpha^A(x))} > \e
    \]
    Then there is some $\delta > 0$ such that for all $k\in\NN$, 
    \[
      \norm{(1 - e_k)(y - \alpha^A(x))} > \e + \delta
    \]
    Let $B_p$ ($p\in\NN$) be a Cauchy sequence from $\SD$ witnessing that $(x,y)\in\SR$ and set $N_0 = 0$.  Since $\norm{y    - \alpha^A(x)} > \e + \delta$, we may find $p_0\in\NN$ and $N_1\in\NN$ large enough that $\supp(x)\cap N_1 \subseteq    B_{p_0}$, and

    \begin{gather}
      \label{cdn:far-0} 
\norm{\sum_{n=N_0}^{N_1 - 1} \alpha_n^{B_{p_0}}(x_n) - \sum_{n=N_0}^{N_1 - 1}\alpha_n^A(x_n)} > \e + \delta
    \end{gather}
    Since the above sums are in $\cstar{A}$, we may find $k_0$ large enough that

    \begin{gather}      \label{cdn:ortho-0} 
\norm{(1 - e_{k_0})\left(\sum_{n=N_0}^{N_1 - 1} \alpha_n^{B_{p_0}}(x_n) - \sum_{n=N_0}^{N_1 -      1} \alpha_n^A(x_n)\right)} < \frac{\delta}{4}
    \end{gather}

    Now as $\norm{(1 - e_{k_0})(y - \alpha^A(x))} > \e + \delta$, we may find $p_1,N_2\in\NN$ large enough that    $B_{p_1}\cap N_1 = B_{p_0}\cap N_1$, $\supp(x)\cap N_2 \subseteq B_{p_1}$, and

    \begin{gather}     
 \label{cdn:close-0} 
\norm{\sum_{n=0}^{N_1 - 1} \alpha_n^{B_{p_0}}(x_n) - \sum_{n=0}^{N_1 - 1} \alpha_n^{B_{p_1}}(x_n)} < \frac{\delta}{2} \\
      \label{cdn:far-1} 
\norm{(1 - e_{k_0})\left(\sum_{n=0}^{N_2 - 1} \alpha_n^{B_{p_1}}(x_n) - \sum_{n=0}^{N_2 - 1} \alpha_n^A(x_n)\right)} > \e + \delta
    \end{gather}
    By conditions~\eqref{cdn:far-0} and~\eqref{cdn:close-0}, we have
    \[
      \norm{\sum_{n=N_0}^{N_1 - 1} \alpha_n^{B_{p_1}}(x_n) - \sum_{n=N_0}^{N_1 - 1} \alpha_n^A(x_n)} > \e + \frac{\delta}{2}
    \]
    whereas by conditions~\eqref{cdn:ortho-0},~\eqref{cdn:close-0} and~\eqref{cdn:far-1}, we have
    \[
      \norm{\sum_{n=N_1}^{N_2 - 1} \alpha_n^{B_{p_1}}(x_n) - \sum_{n=N_1}^{N_2 - 1} \alpha_n^A(x_n)} > \e + \frac{\delta}{4}
    \]
    Repeating this construction, we may find a sequence $N_0 < N_1 < \cdots < N_m$, and a set $B = B_{p_{m-1}}\in\SD$, such    that $\supp(x)\cap N_m \subseteq B$, and for each $i < m$,
    \[
      \norm{\sum_{n = N_i}^{N_{i+1}-1} \alpha_n^B(x_n) - \sum_{n=N_i}^{N_{i+1}-1} \alpha_n^A(x_n)} > \e
    \]
    Then, $\{A,B\}\in K_0^m$, a contradiction to $A,B\in\SH_m$, which is $K_1^m$-homogeneous.

  \end{proof}
  \begin{lemma}
    Suppose $F_n \colon M[\NN]\to \mathcal M(\cstar{A})$ is a sequence of Baire-measurable maps, $\e>0$ and $\SJ\subseteq\P(\NN)$ is a  nonmeager ideal such that for all $A\in\SJ$ there is some $n$ such that $F_n$ is an $\e$-lift of $\iso$ on $\M[A]$. Then there is a Borel-measurable map $G \colon M[\NN]\to\mathcal M(\cstar{A})$ that is a $24\epsilon$-lift of $\iso$ on $\SI$.
  \end{lemma}
\begin{proof}
Fix $F_n$ as in the hypothesis.    Since each $F_n$ is Baire-measurable, we may find a $\SG\subseteq \SX$ comeager on which every $F_n$ is   continuous.  By Theorem~\ref{thm:STPrel.talagrand}, we may find a partition of $\NN$ into finite intervals $E_i$, and $t_i\in\SX[E_i]$, such that if $\exists^\infty n (x\supseteq    t_n)$, then $x\in \SG$.  Since $\SJ$ is nonmeager, by Theorem~\ref{thm:STPrel.talagrand} there is an infinite    $L\subseteq\NN$ such that $S = \bigcup_{n\in L}E_n\in\SJ$.  For $k=0,1$, put
\[
S_k=\bigcup_i E_{2i+k}\,\,\,\,,\,\,\,\, t^k=\sum_it_{2i+k}
\]
and
    \[
      F_n'(x) = F_n(x\rs S_0 + t^1) - F(t^1) + F_n(x\rs S_1 + t^0) - F(t^0).
    \]
   Each $F_n'$ is continuous, and moreover if $A\in\SJ$, then $A\cup S\in\SJ$. Hence there is some $n$ such    that $F_n$ is an $\e$-lift of $\iso$ on $M[A\cup S]$, in which case $F_n'$ is then a $2\e$-lift of $\iso$ on $M[A]$.    We will write $F_n = F_n'$ in what follows.

    Let $\SH_n$ be the family of all $A\subseteq\NN$ such that $F_n$ is a $2\e$-lift of $\iso$ on $M[A]$; then each    $\SH_n$ is hereditary, and $\SJ =    \bigcup_n\SH_n$.  Let 
\[
L'=\{n\mid\SH_n\text{ is nonmeager}\}.
\]
  If $n\in L'$, there is some $k\subseteq\NN$ and $\sigma\in\SX[k]$ such that $\SH_n\cap \SN_\tau$ is    nonmeager, for every finitely-supported $\tau\supseteq\sigma$, where $\SN_\tau$ is the basic open subset of $\SX$    consisting of those $x\in\SX$ extending $\tau$.  Put
    \[
      F_n'(x) = F_n(x\rs [k,\infty) + \sigma)
    \]
    Then, $F_n'$ is a $2\e$-lift of $\iso$ on $M[A]$ for all $A\subseteq\NN$ such that $A\sm k\in\SH_n$.  Replacing again  $F_n$ by $F_n'$, we may assume that for every $n\in L'$, $\SH_n$ is everywhere nonmeager.  For each $m,n\in L'$, put
    \[
      \SZ_{mn} = \set{A\subseteq\NN}{\forall x\in \SX[A]\quad (\norm{\pi(F_m(x) - F_n(x))} \le 4\e)}
    \]
    Then $\SZ_{mn}$ contains $\SH_m\cap\SH_n$ and hence is everywhere nonmeager.  Moreover, $\SZ_{mn}$ is coanalytic,    and hence comeager.  Define
    \[
      \SE = \bigcap_{n,m\in L'}\SZ_{mn}\sm \bigcup_{n\notin L'}\SH_n.
    \]
    Then $\SE$ is comeager.  Applying again Theorem~\ref{thm:STPrel.talagrand}, we find a partition of $\NN$ into finite intervals $E'_i$,    and a sequence $s_i\subseteq E'_i$, such that
\[\exists^\infty i (A\cap E_i'=s_i)\Rightarrow A\in \SE.
\]
 Since $\SJ$ is nonmeager, we may find disjoint, infinite $I_0$ and $I_1$ such that 
\[
S_j =    \bigcup_{i\in I_j}s_i\in\SJ \text{ for }j=0,1.
\]
  Put $E^j = \bigcup_{i\in I_j}E'_i$ for $j=0,1$.  Now fix any $N\in L'$ and put
    \[
      G(x) = F_N(x\rs (\NN\sm E^0) + P_{S_0}) - F_N(P_{S_0}) + F_N(x\rs E^0 + P_{S_1}) - F_N(P_{S_1})
    \]
    Since $F_N$ is continuous, $G$ is Borel. Suppose now that $A\in\SJ$.  Then $(A\sm E^0)\cup S_0$ and $(A\cap    E^0)\cup S_1$ are both in $\SJ\cap\SE$.  It follows that $G$ is a $24\e$-lift of $\iso$ on $\SX[A]$, for any    $A\in\SJ$.

  \end{proof}
 \begin{lemma}
    \label{lemma:C->asymptotic}
    Suppose $F \colon M[\NN]\to \mathcal M(\cstar{A})$ is a $\Cmeas$-measurable map and $\SJ\subseteq \P(\NN)$ is a nonmeager ideal such  that for every $A\in\SJ$, $F$ is a lift of $\iso$ on $M[A]$.  Then there is an asymptotically additive $\alpha$ that is a lift of $\iso$ on $\SI$. Also, $\alpha$ can be chosen to be a skeletal map. Hence, $\SI_C = \SI$.
  \end{lemma}
\begin{proof}
  For $a\in M[\NN]$ we write $a\in M[\SI]$ if there is $A\in\SI$ such that $a\in M[A]$. For simplicity we will write $M[n]$ for $M_{k(n)}$. As before, let $X_n=X_{k(n),n}\subseteq (M[n])_{\leq 1}$ such that $X_n$ is finite and $X_n$ is $2^{-n}$-dense in the ball of $M[n]$, and we let $\SX=\prod X_n$. By definition of $X_n$, $0_{M[n]},1_{M[n]}\in X_n$.

\begin{claim}\label{Claim:continuous}

There is a strongly continuous $F'$ that is a lift of $\iso$ on $\SX\cap M[\SI]$.

\end{claim}

\begin{proof}

By the choice of $X_n$, each of these sets is finite, therefore the strict topology on $\SX$ coincides with the usual product topology, making $\SX$ a compact metric space. In this sense, since $F$ is a Baire-measurable map, there are open dense sets $U_n\supseteq U_{n+1}\cdots$ such that $F$ is strictly continuous on $Z=\bigcap U_n$. Moreover from the fact that $X_n$ is finite, and via a diagonalization argument, we can assure that there are an increasing sequence $n_i$, for $i\in\en$, and elements $x_i\in\prod_{n_i\leq j<n_{i+1}}X_j$ with the property that if $a\restriction [n_i,n_{i+1})=x_i$ then $a\in\bigcap_{n\leq i}U_n$. In particular, 
\[
\{x\mid\exists^\infty i\;( x\restriction [n_i,n_{i+1})=x_i)\}\subseteq Z.
\]
 Since $\SI$ is nonmeager, we can find an infinite $L=\{l_k\}$ such that $\bigcup_{i\in L}[n_i,n_{i+1})\in \SI$. Set now 
\[
L^0_k=[n_{l_{2k}},n_{l_{2k+1}}),\,\,\,L_k^1=[n_{l_{2k+1}},n_{l_{2k+2}}),\,\,\,L^0=\bigcup L_k^0\text{ and }L^1=\bigcup L_k^1,
\]
and  let $P_{L^r}$ be the canonical projections onto $\{(x_i)\mid i\notin L^r\Rightarrow x_i=0\}$ for $r=0,1$. Fix $x^0=\sum_k x_{l_{2k}}$ and $x^1=\sum_k x_{l_{2k+1}}$ where $x_i$ were chosen as above. Thanks to our choice of $L$, we have that for every $a\in M[\SI]$, both $aP_{L^0}+x^0$ and $aP_{L^1}+x^1$ belong to $M[\SI]$, and moreover 
\[
\{x\mid\exists^\infty i (x\restriction L^0_i=x_{l_i}\text{ or }x\restriction L^1_i=x_{l_i})\}\subseteq Z.
\]
 Putting all of these together we can therefore construct 
\[
F'(a)=F(aP_{L^0}+x^0)-F(x^0)+F(aP_{L^1}+x^1)-F(x^1).
\]
 This function is strongly continuous by definition, and it is a lift of $\iso$ on $ M[\SI]\cap\SX$.

\end{proof}

From now on we will assume that $F$ is a strongly continuous map on $\SX$ that is a lift of $\iso$ on $\SX\cap  M[\SI]$, as established in Claim \ref{Claim:continuous}. The next Claim will be needed in the proceeding.

\begin{claim}\label{Claim:Ynonmeager}

$\SX\cap  M[\SI]$ is relatively nonmeager in $\SX$.

\end{claim}

\begin{proof}

If $\SX\cap  M[\SI]$ is meager in $\SX$, we can find an increasing sequence $n_i$ and some $s_i\in\prod_{n_i\leq j<n_{i+1}}X_j$  with $\norm{s_i}=1$ such that 
\[
\{x\in\SX\mid \exists^\infty i (x\restriction[n_i,n_{i+1})=s_i)\}\cap  M[\SI]=\emptyset
\]
 Since $\SI$ is nonmeager, we have that there is an infinite $L\subseteq\en$ such that $\bigcup_{i\in L}[n_i,n_{i+1})\in \SI$. On the other hand, as $0\in X_n$ for every $n\in\en$, we have that $s=\sum_{i\in L}s_i\in M[\SI]$, and moreover we have that 
\[
s\in\{x\in\SX\mid \exists^\infty i (x\restriction[n_i,n_{i+1})=s_i)\},
\]
 a contradiction.

\end{proof}

Let $q_n$ be an approximate identity of projections for $\cstar{A}$ as required by our assumptions in \S\ref{s:FA.Lift1}. Clearly $q_n$ converges strictly to $1$, when seen as an element of $\mathcal M(\cstar{A})$. Note that $a\in \cstar{A}$ if and only if $\norm{a(1-q_n)}\to 0$.
Define 
\[
\Delta(x,y,k)=\max\{\norm{(x-y)(1-q_k)},\norm{(1-q_k)(x-y)}\}.
\]

\begin{claim}\label{Claim:stabilizer1}

For every $n$ and $\epsilon>0$ there exists $k\geq n$ and $a\in\prod_{n\leq i<k}X_i$ such that for all $x,y\in\SX$ with $x\restriction [n,\infty)=y\restriction [n,\infty)$ and $x\restriction [n,k)=a$ we have $\Delta(F(x),F(y),k)\leq\epsilon$

\end{claim}

\begin{proof}

Fix $n$ and $\epsilon$ and let $W=\prod_{i<n}X_i$. If $x\in\SX$ and $s\in W$ we can write $x(s)=s+x\restriction [n,\infty)$. If $k>n$ let 
\[
V_k=\{x\in\SX\mid\exists s,t\in W(\Delta(F(x(s)),F(x(t)),k)>\epsilon)\}.
\]
Since $F$ is continuous, $V_k$ is open in $\SX$. Let $x\in  M[\SI]$ and $s,t\in W$. As $F(x(s))-F(x(t))\in \cstar{A}$, there is $k=k(x,s,t)$ such that $\Delta(F(x(s)),F(x(t)),k)\leq\epsilon$. Since $W$ is finite, $x\notin V_k$, and in particular $\bigcap V_k\cap\SX\cap  M[\SI]=\emptyset$. By Claim \ref{Claim:Ynonmeager} there are $k\in\en$ and $U$ basic open set with $V_k\cap U=\emptyset$. Note moreover that, by definition of $V_k$,, for every $y\in\SX$ we have that $y\in V_k$ if and only if $y(s)\in V_k$ for all $s\in W$, therefore we have that there exists an $l\geq k$ and an $a\in\prod_{n\leq i<l} X_i$ with the property that 

\[
\{x\in\SX\mid x\restriction [n,l)=a\}\cap V_l=\emptyset.
\]
 Since $V_l\subseteq V_k$, $l$ and $a$ satisfy the claim.

\end{proof}

For $J\subseteq \en$ denote $Z_J=\prod_{i\in J}X_i$, and $\epsilon_i=2^{-i}$.

\begin{claim}\label{Claim:stabilizer2}

There are sequences $n_i<k_i<n_{i+1}$ and $u_i\in Z_{i}$ (where, for simplicity, $Z_i=Z_{[n_i,n_{i+1})}$) with the property that if $x,y\in\SX$ are such that $x\restriction [n_i,n_{i+1})=y\restriction [n_i,n_{i+1})=u_i$ then

\begin{enumerate}
\item\label{c1:stabilizer} $x\restriction [n_i,\infty)=y\restriction[n_i,\infty)$ implies $\Delta(F(x),F(y),k_i)<\epsilon_i$
\item\label{c2:stabilizer} $x\restriction [0,n_{i+1})=y\restriction[0,n_{i+1})$ implies $\Delta(F(x),F(y),k_i)<\epsilon_i$
\end{enumerate}

\end{claim}

\begin{proof} 
We construct such sequences by induction.
Set $k_{-1}=0=n_0$ and $u_{-1}=\emptyset$.
Suppose that we have found $n_i$, $k_{i-1}$ and $u_{i-1}$ to satisfy the requirements of the Claim. Using Claim  \ref{Claim:stabilizer1} we can find $k_i> n_i$ and $a_i\in\prod_{n_i\leq j<k_i}X_j$ so that condition \ref{c1:stabilizer} is satisfied. We now apply continuity of $F$ to find $n_{i+1}\geq k_i$ and $u_i\in\prod_{n_i\leq j<n_{i+1}}X_j$ with the property that $u_i\restriction[n_i,k_i)=a_i$ and condition \ref{c2:stabilizer} is satisfied. This concludes the proof.

\end{proof}

Let $\mc V_i=M[[n_i,n_{i+1})]$. Every $x\in M[\NN]$ can be seen as $x=\sum x_i$, with $x_i\in\mc V_i$.

Recall that $Z_i=Z_{[n_i,n_{i+1})}=\prod_{n_i\leq j<n_{i+1}}X_j$ is finite and since each $1_{M[n]}\in X_n$ we have $1_{\mc V_i}\in Z_i$. Fix a linear order of $Z_i$ and let 
\[
\sigma_i(\mc V_i)\to Z_i
\]
 mapping $x$ to the first element of $Z_i$ that is within $2\epsilon_i$ from $x$. Note that each $\sigma_i$ is Borel measurable.

 Note that since each $X_i$ is $\epsilon_i$-dense we have that $Z_i$ is $2\epsilon_{n_{i}}$-dense and, consequently, $2\epsilon_i$-dense. If $x=\sum x_i$, with $x_i\in\mc V_i$, let $x^0=\sum_{i}\sigma_{2i}(x_{2i})$ and $x^1=\sum_{i}\sigma_{2i+1}(x_{2i+1})$. Then $x^0,x^1\in\SX$ and $x-x^0-x^1\in\oplus M_{k(n)}$, as, coordinatewise, we have $\norm{x_i-x^0_i-x_i^1}<3\epsilon_i$. 

Recall that $u_i\in Z_i$ was defined to satisfy Claim \ref{Claim:stabilizer2} and let $u^0$ and $u^1$ defined as above. Define, for $k=0,1$ and $x\in\mc V_{2i+k}$,
\[
\Lambda_{2i+k}(x)=F(u^{1-k}+\sigma_{2i+k}(x))-F(u^{1-k}).
\]
Note that $\Lambda_{2i+k}\colon M[[n_{2i+k},n_{2i+k+1})]\to \cstar{A}$. Since $u^0$ and $u^1$ are fixed and $\sigma_i$ is Borel-measurable, so is each $\Lambda_i$.

Let $k_i$ as fixed in Claim \ref{Claim:stabilizer2}. We modify again $\Lambda_i$ on $\mc V_i$ setting 

\[
\Lambda_i'=(q_{k_{i+1}}-q_{k_i})\Lambda_i(q_{k_{i+1}}-q_{k_i}).
\]
Let
\[
\Lambda=\sum\Lambda_i'.
\]
Since the ranges of the $\Lambda'_i$'s are orthogonal, $\Lambda$ is a well defined asymptotically additive map from $M[\NN]$ to $\mc(\cstar{A})$. In particular, since the value of $\Lambda$ is completely determined by the values of $\sigma_i$, and by our choice of $X_n$, $\Lambda$ is a skeletal map.
\begin{claim}\label{claim:itlifts}

For every $x\in  M[\SI]$, $\Lambda(x)$ is a lift for $x$.

\end{claim}
\begin{proof}
Write $x=\sum x_i$ with $x_i\in\mc V_i$. If $x=\sum x_i$ and $y=\sum y_i$, then we have that, as $i\to\infty$, $\Lambda'_i(x_i+y_i)\to\Lambda_i'(x_i)+\Lambda'_i(y_i)$, and so $\Lambda(x+y)=\Lambda(x)+\Lambda(y)$. For this reason, is it enough to prove that  $\Lambda(x)$ is a lift for $x$ if $x_{2i}=0$ for all $i$. Recalling that, if $x^1=\sum\sigma(x_{2i+1})$, we have $x-x^1\in \oplus M_{k(n)}$, we can infer that $F(x)+F(u^0)-F(u^0-x^1)\in\cstar{A}$. In particular, we can apply conditions \ref{c1:stabilizer} and \ref{c2:stabilizer} of Claim \ref{Claim:stabilizer2} to $x^1$ and $\sigma_{2i+1}(x_{2i+1})$ to get that 
\[
\norm{(\Lambda_{2i+1}(x_{2i+1})+F(u^0)-F(u^0+x^1))(q_{k_{2i+2}}-q_{k_{2i+1}})}<2\cdot 2^{-2i}.
\]
 Since $2\cdot 2^{-2i}$ is summable, and 
\[
1-\sum_i (q_{k_{i+2}}-q_{k_{i+1}})\in \cstar{A}
\]
 we have that 
\[
F(u^0+x^1)-F(u^0)-\sum_i \Lambda_{2i+1}(x_{2i+1})(q_{k_{i+2}}-q_{k_{i+1}})\in \cstar{A}.
\]
 For the multiplication on the other side, we can apply again \ref{c1:stabilizer} and \ref{c2:stabilizer} of Claim \ref{Claim:stabilizer2}.\qedhere
\end{proof}

This concludes the proof of the Lemma.\qedhere
\end{proof}

\begin{remark}\label{remark:asymptorthogonal}
As asymptotically additive lifts are constructed as in the proof of Lemma~\ref{lemma:C->asymptotic} (see, in particular, the definition of $\Lambda_i'$ before Claim~\ref{claim:itlifts}), we have, for $j=0,1$ and $n\in \NN$, 
\[
\ran(\alpha_n^j)\subseteq (q_k-q_l)\cstar{A}(q_k-q_l),
\]
for some $k,l\in\NN$. For this reason, whenever we will consider $\alpha,\beta$ are asymptotically additive functions $\alpha\colon\prod M_{k(n)}\to\mathcal M(\cstar{A})$, $\beta\colon\prod M_{l(n)}\to\mathcal M(\cstar{A})$, we will assume that 
\[
\forall n\forall^\infty m(\alpha_n\beta_m=0).
\]
\end{remark}

\section{Consequences of the lifting theorem}\label{s:FA.consofLT}

We analyze the consequences of Theorem~\ref{thm:FA.liftingtheorem} In \S\ref{ss:FA.BorelLift} we provide partial solutions to Conjecture~\ref{conj:PFA}, and in \S\ref{s:FA.Emb} we state and prove results related to the consistency of the existence of certain embeddings between corona $\Cstar$-algebras, proving in particular Theorem~\ref{thmi:1}.

\subsection{Consequences I: Trivial automorphisms}\label{ss:FA.BorelLift}
 Recall that, if $\cstar{A}$ is a separable $C^*$-algebra, $\Lambda\in \Aut(\mathcal M(\cstar{A})/\cstar{A})$ is said to be \emph{trivial} if its graph 
\[
\Gamma_\Lambda=\{(a,b)\in\mathcal M(\cstar{A})_{\leq 1}^2\mid \Lambda(\pi(a))=\pi(b)\}
\]
 is Borel in the strict topology of $\mathcal M(\cstar{A})_{\leq 1}$, where $\pi$ denotes the usual quotient map. The following is Conjecture~\ref{conj:PFA}

\begin{conjecture}
Let $\cstar{A}$ a $\sigma$-unital $\Cstar$-algebra. Then the assumption of Forcing Axioms implies that every automorphism of $\mathcal M(\cstar{A})/\cstar{A}$ is trivial.
\end{conjecture}
The main result proved in this section is the following.

\begin{theorem}\label{thm:Borel}

Assume $\OCA_\infty$ and $\MA_{\aleph_1}$. Let $\cstar{A}$ be a separable nuclear $\Cstar$-algebra admitting an increasing approximate identity of projections. Then every automorphism of $\mathcal M(\cstar{A})/\cstar{A}$ is trivial.

\end{theorem}

From now on, $\cstar{A}$ will denote a separable, nuclear $\Cstar$-algebra with an increasing approximate identity of projections $(q_n)$.  Given $S\subseteq\NN$, define $q_S = \sum\set{q_n - q_{n-1}}{n\in S}\in\mathcal M(\cstar{A})$.  (We set $q_{-1} =  0$.)

\begin{proposition}\label{prop:onedominates}

Let $\cstar{A}$ be a separable $\Cstar$-algebra with an increasing approximate identity of projections. Let $\{E_n\}$ be a partition of $\NN$ into finite intervals. Let $\SI\subseteq\mathcal P(\NN)$ be a nonmeager dense ideal and, for $A\subseteq\NN$, $q_A^E=\sum_{n\in A}q_{E_n}$. If a projection $q\in\mathcal M(\cstar{A})/\cstar{A}$ dominates each $\{\pi(q_A^E)\mid A\in\SI\}$, then $q=1$.
\end{proposition}

Let $(Y_n)$ be an increasing sequence of finite subsets of $\cstar{A}_{\leq 1}$ whose union is dense in $\cstar{A}$, and $(\e_n)$ a sequence of positive reals converging to zero. Recall that, by nuclearity of $\cstar{A}$, for  each $n\in\NN$ we may find a finite-dimensional $\Cstar$-algebra $F_n$ and cpc maps $\phi_n \colon \cstar{A}\to F_n$,  $\psi_n \colon F_n\to \cstar{A}$, such that 
\[
\norm{\psi_n(\phi_n(x)) - x} \le \e_n\norm{x}\text{ for all }x\in Y_n.
\]
 Let $\PP$ be the set of all partitions $E = \seq{E_n}{n\in\NN}$ of $\NN$ into finite, consecutive intervals. We turn $\PP$ into a partial order, with 
 \[
 E\leq_1 F\iff \forall^\infty i\exists j E_i\cup E_{i+1}\subseteq F_j\cup F_{j+1}.
\]
 For each $E\in\PP$ set $E^0 = \seq{E_{2n}\cup E_{2n+1}}{n\in\NN}$ and $E^1 = \seq{E_{2n+1}\cup  E_{2n+2}}{n\in\NN}$ (with $E_{-1}=\emptyset$), and define, for $f\in\NN^\NN$
 \[
 \Phi_{f,E} \colon \mathcal M(\cstar{A})\to \prod_n F_{f(\max E_n)}\text{ and }\Psi_{f,E} \colon\prod_n
  F_{f(\max E_n)}\to \mathcal M(\cstar{A})
  \]
   as follows:

  \begin{align*}
    \Phi_{f,E}(t)(n) & = \phi_{f(\max E_n)}(q_{E_n}t q_{E_n}), \text{ and } \\
    \Psi_{f,E}(x) & = \sum_{n = 0}^\infty q_{E_n} \psi_{f(\max E_n)}(x_n) q_{E_n} \text{ for }x=(x_n)\in\prod F_{f(\max E_n)}.
  \end{align*}
  Since the projections $q_{E_n}$ are pairwise orthogonal, and the norms of $\psi_{f(\max E_n)}(x_n)$ are bounded, the sum in the definition of $\Psi_{f,E}(x)$ is a well-defined element of $\mathcal M(\cstar{A})$. The proof of the following is immediate from the definitions.

  \begin{proposition}    \label{prop:cpc-products-basic-properties}

    For each $E\in\PP$ and $f\in\NN^\NN$, the maps $\Phi_{f,E}$ and $\Psi_{f,E}$ are cpc.  Moreover 
    \[
    \Phi_{f,E}(\cstar{A})\subseteq\bigoplus_n F_{f(\max E_n)}\text{ and }\Psi_{f,E}(\bigoplus_n
    F_{f(\max E_n)})\subseteq\cstar{A}.
\]
  \end{proposition}
The following is a crucial lemma for our construction, since it allows us to see $\mathcal M(\cstar{A})/\cstar{A}$ as union of ``building blocks''. The proof resembles techniques used in \cite[Theorem 3.1]{Elliott.Der2} (see also \cite[Lemma 1.2]{Farah.C} or \cite{Arveson:Notes}).

  \begin{lemma}\label{lem:summing}

 Let $t\in\mathcal M(\cstar{A})$. Then there are $f\in\NN^\NN$, $E\in\PP$, and $x^i\in\prod F_{f(\max E^i_n)}$, for $i=0,1$, such that
    \[
      t - (\Psi_{f,E^0}(x^0) + \Psi_{f,E^1}(x^1))\in\cstar{A}
    \]

  \end{lemma}

  \begin{proof}

    If $k\in\NN$ and $\e > 0$, then $t q_k$ and $q_k t$ are both in $\cstar{A}$, hence we may find $\ell > k$ such that    $\norm{q_\ell t q_k - t q_k} + \norm{q_k t q_\ell - q_k t} < \e$.  Applying  this process recursively, we may find $0=k_0<k_1 < k_2 < \cdots < k_n < \cdots $ such that for every $n\in\NN$,
    \[
      \norm{q_{k_{n+1}} t q_{k_n} - t q_{k_n}} + \norm{q_{k_n} t q_{k_{n+1}} - q_{k_n} t} < \frac{1}{2^n}
    \]
    Let $E_n = [k_n, k_{n+1})$ for each $n \ge 0$.  Put

  \begin{align*}
      t^0 & = \sum_{n=0}^\infty q_{E_{2n}} t q_{E_{2n}}+ q_{E_{2n+1}} t q_{E_{2n}}+ q_{E_{2n}} t q_{E_{2n+1}} \\
      t^1 & = \sum_{n=0}^\infty q_{E_{2n+1}} t q_{E_{2n+1}}+q_{E_{2n+1}} t q_{E_{2n+2}} + q_{E_{2n+2}} t q_{E_{2n+1}}
    \end{align*}
    Then,
    \[
      q_{k_m}(t - (t^0 + t^1)) q_{k_m} = \sum_{n = 0}^m (1 - q_{k_{n+2}}) t q_{E_n} + q_{E_n} t (1 - q_{k_{n+2}})
    \]
    Since the above sum has norm smaller than $2^{-m+1}$, it follows that $t - (t^0 + t^1)\in\cstar{A}$.  Now, for   each $n\in\NN$, choose $f(n)\in\NN$ large enough that for each $i < 2$,
    \[
      \norm{q_{E_n^i} t^i q_{E_{n^i}} - \psi_{f(n)}(\phi_{f(n)}(q_{E_n^i} t^i q_{E_n^i}))} < \frac{1}{2^n}
    \]
    It follows that
    \[
      t^i - \Psi_{f,E^i}(\Phi_{f,E}(t^i))\in\cstar{A}
    \]
   Setting $x^i = \Phi_{f,E}(t^i)$, the thesis follows.

  \end{proof}

Let $D[E]=\{x\in\mathcal M(\cstar{A})\mid x=\sum q_{E_n}xq_{E_n}\}$ and, for $f\in\NN^\NN$, define $D_{f}[E]$ to be the set of all $x\in D[E]$ such that
\[
       \forall n\forall m\geq f(\max E_n)  (\norm{q_{E_n}xq_{E_n}-\psi_m(\phi_m(q_{E_n}xq_{E_n}))}<2^{-n}).
\]
We define $G[E]\subseteq D[E^1]$ by letting $x=\sum q_{E^1_n}x_nq_{E^1_n}\in G[E]$ if and only if $x_n\neq 0$ implies $x_n\notin D[E^0]$ and $G_f[E]=G[E]\cap D_f[E^1]$.
The following properties follow from the definition and Lemma \ref{lem:summing} above.

\begin{proposition}\label{prop:mainproperties}
\begin{enumerate}
\item If $f\leq^*g$ and $E\in\PP$, then $\pi(D_f[E])\subseteq \pi(D_g[E])$;
\item For all $E\leq_1 F$ and $f\leq^*g$ then \[\pi(D_f[E^0]+D_f[E^1])\subseteq\pi(D_g[F^0]+D_g[F^1]);\]
\item For every $t\in\mathcal M(\cstar{A})$ there are $f$, $E$, $x_0$ and $x_1$ such that $t-x_0-x_1\in\cstar{A}$, $x_0\in D_f[E^0]$, $x_1\in G_f[E]$. Moreover, if $t$ is positive, $x_0$ and $x_1$ may be chosen to be positive.
\end{enumerate}
\end{proposition}
Let $\Lambda$ be an automorphism of $\mathcal M(\cstar{A})/\cstar{A}$.
 By Proposition~\ref{prop:cpc-products-basic-properties}, for each $f$ and $E$, the map
\[
\Psi'_{f,E}\colon\prod F_{f(\max E_n)}/\bigoplus F_{f(\max E_n)}\to \mathcal M(\cstar{A})/\cstar{A},
\]
 defined by 
\[
\Psi'_{f,E}(\pi(x))=\pi(\Psi_{f,E}(x)),
\]
is a well-defined cpc map, and so is $\Lambda_{f,E}$ defined as 
\[
\Lambda_{f,E}=\Lambda\circ\Psi'_{f,E}\colon \prod F_{f(\max E_n)}/\bigoplus F_{f(\max E_n)}\to\mathcal M(\cstar{A})/\cstar{A}.
\]

\begin{lemma}\label{lem:liftforeverything.Borel}

Let $E\in\PP$, and $f\in\NN^\NN$ such that 
\[
\lim_n\norm{q_{E_n}- \psi_{f(\max E_n)}(\phi_{f(\max E_n)}(q_{E_n}))}\to 0.
\]
Suppose that $\alpha_{f,E}$ is an asymptotically additive map that is a lift of $\Lambda_{f,E}$ on a nonmeager ideal $\SI_{f,E}$. Then $\alpha_{f,E}$ is a lift of $\Lambda_{f,E}$ on 
\[
\{\Psi_{f,E}(x)\mid x\in D[E]\wedge\lim\norm{q_{E_n}xq_{E_n}-\psi_{f(\max E_n)}(\phi_{f(\max E_n)}(x))}\to 0\},
\]
 that is, if $x\in D[E]$ and 
\[
\lim\norm{q_{E_n}xq_{E_n}-\psi_{f(\max E_n)}(\phi_{f(\max E_n)}(x))}\to 0,
\] 
then 
\[
\pi(\alpha_{f,E}(\Phi_{f,E}(x)))-\Lambda(\pi(x))=0.
\]

\end{lemma}

\begin{proof}

We will first show that $\alpha_{f,E}(\Phi_{f,E}(1))-1\in\cstar{A}$, and then prove that this is sufficient to obtain our thesis. 
Recall that if $E\in\PP$ and $A\subseteq\NN$, we have defined $q^E_A=\sum_{n\in A}q_{E_n}$. $q^E_A$ satisfies that 
\[
\norm{q_{E_n}q^E_Aq_{E_n}-\psi_{f(\max E_n)}(\phi_{f(\max E_n)}(q^E_A))}\to 0\text { as }n\to\infty.
\]
If, in addition, we have $A\in\SI_{f,E}$, we have that  
\[
\alpha_{f,E}(\Phi_{f,E}(q^E_A))-\Lambda(\pi(q^E_A))\in\cstar{A},
\]
 since $\Phi_{f,E}(q^E_A)$ has support contained in $\SI_{f,E}$ and $\Phi_{f,E}(\Psi_{f,E}(q^E_A))-q^E_A\in\cstar{A}$.  Since $\alpha_{f,E}$ is asymptotically positive, being the lift of a positive map $\Lambda_{f,E}$, we have that $p\leq q\Rightarrow \pi(\alpha_{f,E}(p))\leq\pi(\alpha_{f,E}(q))$. Therefore, since $\Phi_{f,E}(q^E_A)\leq 1$, as an element of $\prod F_{f(\max E_n)}$, we have that $\pi(\alpha(1))$ dominates $\pi(\alpha_{f,E}(\Phi_{f,E}(q^E_A)))=\Lambda(\pi(q_A^E))$. Since $\Lambda$ is an automorphism and $\pi(\alpha_{f,E}(\Phi_{f,E}(1)))$ is a projection, we can apply Proposition~\ref{prop:onedominates} to have that 
\[
1-\alpha_{f,E}(\Phi_{f,E}(1))\in\cstar{A}.
\]
Fix now $x$ such that $x\in D[E]$ and 
\[
\lim\norm{q_{E_n}xq_{E_n}-\psi_{f(\max E_n)}(\phi_{f(\max E_n)}(x))}\to 0.
\]
 Let $y$ such that $\pi(y)=\Lambda(\pi(x))$ and define 
\[
\SI_x=\{A\subseteq\NN\mid \alpha_{f,E}(\Phi_{f,E}(q^E_A))(\alpha_{f,E}(\Phi_{f,E}(x))-\pi(y))\in\cstar{A}\}.\
\]
This is an ideal containing $\SI_{f,E}$ and so is nonmeager. Moreover, since $\alpha_{f,E}$ is strictly-strictly continuous, $x$, $E$, $f$, $\Phi_{f,E}$ and $y$ are fixed, and $\cstar{A}$ is Borel in the strict topology of $\mathcal M(\cstar{A})$, $\SI_x$ is Borel. Since every Borel dense nontrivial ideal in $\mathcal P(\NN))$ is meager, we have $\SI_x=\mathcal P(\NN)$. Since $\NN\in\SI_x$ and $1-\alpha_{f,E}(\Phi_{f,E}(1))\in\cstar{A}$, we have 
\[\pi(\alpha_{f,E}(\Phi_{f,E}(x)))-\Lambda(\pi(x)))=0
\]
as required.
\end{proof}

Assume now $\OCA_\infty$ and $\MA_{\aleph_1}$. With in mind the definition of skeletal map from \ref{defin:FA.Skel}, define 
\[
X=\{(f,E,\alpha^0,\alpha^1)\mid\alpha^i\text{ is a skeletal lift of }\Lambda_{f,E^i} \text{ on }\Phi_{f,E^i}(D_f[E^i])\}.
\]
By Theorem \ref{thm:FA.liftingtheorem}, Lemma \ref{lem:liftforeverything.Borel}, and the fact that the asymptotically additive maps in Lemma~\ref{lemma:C->asymptotic} can be chosen to be skeletal\footnote{
The need of choosing skeletal maps instead of, simply, asymptotically additive ones, is in that the set of all skeletal maps has a natural separable topology, see Proposition~\ref{prop:SkeletalPolish}.}, for every $E\in\PP$ and $f\in\NN^{\NN}$ there are $\alpha^0$ and $\alpha^1$ such that $(f,E,\alpha^0,\alpha^1)\in X$.

\begin{lemma}\label{lem:afteragivenpoint}

Let $(f,E,\alpha^0,\alpha^1), (g,F,\beta^0,\beta^1)\in X$ and $\epsilon>0$. Then there is $M$ such that for all $n>M$ and $x=q_{[M,n]}xq_{[M,n]}$ with $\norm{x}\leq 1$, if $x\in D_f[E^i]\cap D_g[F^j]$ we have \[\norm{\alpha^i(\Phi_{f,E^i}(x))-\beta^j(\Phi_{g,F^j}(x))}\leq\epsilon.\]

\end{lemma}

\begin{proof}

We work by contradiction. Since for every $f$ and $E$ there are $\alpha^0,\alpha^1$ such that $(f,E,\alpha^0,\alpha^1)\in X$, modifying $E$ and $F$ if necessary, we can assume there exist $\epsilon>0$ and $(f,E,\alpha^0,\alpha^1), (g,F,\beta^0,\beta^1)$ such that there is an increasing sequence $m_1<m_2<\cdots$ and $x_i=q_{[m_i,m_{i+1})}x_iq_{[m_i,m_{i+1})}$ with $\norm{x_i}=1$, $x_i\in D_f[E^0]\cap D_g[F^0]$
and such that for every $i$ we have that 

\[
\norm{\alpha^0(\Phi_{f,E^0}(x_i))-\beta^0(\Phi_{g,F^0}(x_i))}>\epsilon.
\]
Let
$x=\sum x_i$.  Then $x\in D_f[E^0]\cap D_g[F^0]$. Since $(f,E,\alpha^0,\alpha^1),(g,F,\beta^0,\beta^1)\in X$, we have that for all $z\in D_f[E^0]\cap D_g[F^0]$, 
\[
\pi(\alpha^0(\Phi_{f,E^0}(z)))=\Lambda(\pi(z))=\pi(\beta^0(\Phi_{g,F^0}(z))).
\]
On the other hand, by the definition of asymptotically additive map, for every $n$, we have that there is $m>n$ such that
\[
\img(\alpha^0_n)\img(\alpha^0_m)=\img(\beta^0_n)\img(\beta^0_m)=\img(\alpha_n^0)\img(\beta^0_m)=\img(\beta^0_n)\img(\alpha^0_m)=0
\]
This is because, being $\alpha^0$ and $\beta^0$ asymptotically additive, the range of $\alpha^0_n$ is contained in a corner of the form $(q_{i}-q_j)\cstar{A}(q_i-q_j)$ (and the same holds for $\beta^0_n$).

We can therefore find an increasing sequence $n_k$ such that for every $l>k$ we have that 

\begin{eqnarray*}
&&\norm{\alpha^0(\Phi_{f,E^0}(x_{n_k}))\alpha^0(\Phi_{f,E^0}(x_{n_l}))},\norm{\alpha^0(\Phi_{f,E^0}(x_{n_k}))\beta^0(\Phi_{g,F^0}(x_{n_l}))},\\
&&\norm{\beta^0(\Phi_{g,F^0}(x_{n_k}))\alpha^0(\Phi_{f,E^0}(x_{n_l}))}, \norm{\beta^0(\Phi_{g,F^0}(x_{n_k}))\beta^0(\Phi_{g,F^0}(x_{n_l}))}=0
\end{eqnarray*}
Setting $Y=\bigcup [m_{n_k},m_{n_k+1})$ and $z=q_Yxq_Y$ we have that $z\in D_f[E^0]\cap D_g[F^0]$, $\norm{z}=1$ and 

\begin{eqnarray*}
\norm{\pi(\alpha^0(\Phi_{f,E^0}(z))-\beta^0(\Phi_{g,F^0}(z)))}&\geq&\\
\limsup\norm{\alpha^0(\Phi_{f,E^0}(x_{n_k}))-\beta^0(\Phi_{g,F^0}(x_{n_k}))}&\geq&\epsilon,
\end{eqnarray*}
a contradiction to the fact that $(f,E,\alpha^0,\alpha^1),(g,F,\beta^0,\beta^1)\in X$. 

\end{proof}

 For a fixed $\epsilon>0$, define a coloring $[X]^2=K_0^\epsilon\cup K_1^\epsilon$ with \[\{(f,E,\alpha^0,\alpha^1),(g,F,\beta^0,\beta^0)\}\in K_0^\epsilon\] if and only if there is $n\in\NN$ and $x=q_nxq_n$ with $\norm{x}=1$ and such that one of the following conditions applies:

\begin{enumerate}
\item $x\in D_f[E^0]\cap D_g[F^0]$ and $\norm{\alpha^0(\Phi_{f,E^0}(x))-\beta^0(\Phi_{g,F^0}(x))}>\epsilon$
\item $x\in G_f[E]\cap D_g[F^0]$ and $\norm{\alpha^1(\Phi_{f,E^1}(x))-\beta^0(\Phi_{g,F^0}(x))}>\epsilon$
\item $x\in D_f[E^0]\cap G_g[F]$ and $\norm{\alpha^0(\Phi_{f,E^0}(x))-\beta^1(\Phi_{g,F^1}(x))}>\epsilon$
\item $x\in G_f[E]\cap G_g[F]$ and $\norm{\alpha^1(\Phi_{f,E^1}(x))-\beta^1(\Phi_{g,F^1}(x))}>\epsilon$.
\end{enumerate}

\begin{proposition}\label{prop:SkeletalPolish}

For every $\epsilon>0$, $K_0^\epsilon$ is open in some Polish topology.

\end{proposition}

\begin{proof}

Fix $\epsilon>0$. Let $Z=(\NN^\NN)^2\times \Skel(\cstar{A})^2$ with the product topology $\Skel(\cstar{A})$ is the set of all skeletal maps with the uniform topology as in Definition~\ref{defin:FA.Skel}, and $\NN^\NN$ is endowed with the Cantor topology.

This topology is Polish. Moreover, $X\subseteq Z$ and conditions 1.-4. above are open in this topology hence so is $K_0^\epsilon$.

\end{proof}

\begin{lemma}
Assume $\mathfrak b> \omega_1$. Then for every $\epsilon>0$, there is no uncountable $K_0^\epsilon$-homogeneous set.

\end{lemma}

\begin{proof}

By contradiction, let $\epsilon>0$ and $Y$ be a $K_0^\epsilon$-homogeneous set of size $\aleph_1$. We will refine $Y$ to an uncountable subset of itself several times, but we will keep the name $Y$ to not confuse the reader.

As $\mathfrak b>\omega_1$, we can find $\hat f$ and $\hat E$ with the property that for all $(f,E)$ such that there are $\alpha^0$, $\alpha^1$ with $(f,E,\alpha^0,\alpha^1)\in Y$, then $f<^*\hat f$ and $E<_1 \hat E$. By the definitions of $<^*$ and $<_1$, we have therefore that if $(f,E,\alpha^0,\alpha^1)\in Y$ there are $n_f$ and $m_E$ with the property that for all $n\geq n_f$ and $m\geq m_E$ we have $f(n)<\hat f(n)$ and that there is $k$ such that $E_{m}\cup E_{m+1}\subseteq \hat E_k\cup\hat E_{k+1}$. By the pigeonhole principle, we can refine $Y$ so that $n_f=\overline n$ and $m_E=\overline m$, whenever $(f,E,\alpha^0,\alpha^1)\in Y$.

Fix now $(E,f,\alpha^0,\alpha^1)\in Y$ and $\hat\alpha^0,\hat\alpha^1$ such that $(\hat f,\hat E,\hat \alpha^0,\hat\alpha^1)\in X$. Thanks to Lemma \ref{lem:afteragivenpoint}, we can find $M\geq\overline m$ such for all $n>M$ and $x=q_{[M,n]}xq_{[M,n]}$ with $\norm{x}\leq 1$ then if $x\in D_f[E^i]$ we have 

\begin{equation}\label{OCAequation1}
\norm{\sum_{k}\alpha^i_k(\Phi_{f,E^i}(q_{E_k^i}xq_{E_k^i}))-\sum_l\hat\alpha^j_k(\Phi_{g,F^j}(q_{\hat E_l^j}xq_{\hat E_l^j}))}\leq\epsilon/2.
\end{equation}
By another counting argument we can suppose that the minimum $M$ such that this hold is equal to a given $\overline M$ for every element of $Y$. Again using pigeonhole, we can assure that for all $i\leq\overline n$ and $j\leq \overline M+1$ we have that if $(f,E,\alpha^0,\alpha^1),(g,F,\beta^0,\beta^1)\in Y$ then $f(i)=g(i)$ and $E_j=F_j$. Note that $\overline K=\max E_{\overline M}>\overline M$.

Note that, for every $i$ such that $2i\leq \overline M$ and $(f,E,\alpha^0,\alpha^1),(g,F,\beta^0,\beta^1)\in Y$, the domains of $\alpha^0_i$ and of $\beta^0_i$ are the same, as well as the domains of $\alpha^1_i$ and $\beta^1_i$, as there are only countably many finite-dimensional $\Cstar$-algebras. Therefore, for $x=q_{\overline K}x q_{\overline K}$, $x\in D_f[E^0]$ implies that $\Phi_{f,E^0}(x)=\Phi_{g,F^0}(x)$ and if $x\in D_g[F^1]$ then $\Phi_{f,E^1}(x)=\Phi_{g,F^1}(x)$. Since the space of all skeletal maps from $\sum_{i\mid 2i\leq \overline M}F_{f(i)}\to \cstar{A}$ is separable in uniform topology, we can refine $Y$ to an uncountable subset of it such that whenever $(f,E,\alpha^0,\alpha^1),(g,F,\beta^0,\beta^1)\in Y$ and $i$ is such $2i\leq \overline M$, then

\begin{enumerate}[label=(O\arabic*)]
\item\label{OCAcond1} $\norm{\alpha^0_i-\beta^0_i}<\epsilon/(2\overline M)$;
\item\label{OCAcond2} $\norm{\alpha^1_i-\beta^1_i}<\epsilon/(2\overline M)$.
\end{enumerate}
This is the final refinement we need. Pick $(f,E,\alpha^0,\alpha^1),(g,F,\beta^0,\beta^1)\in Y$ and $x$ witnessing that. Then $x=q_{[n,n']}xq_{[n,n']}$ for some $n,n'\in\NN$. If $n'<\overline K$, then, since $E_i=F_i$ for all $i$ such that $q_{E_i}xq_{E_i}\neq 0$, we have that either $x\in D_f[E^0]\cap D_g[F^0]$, or $x\in G[E]\cap G[F]$. If the first case applies, we have a contradiction thanks to condition \ref{OCAcond1}, while the second case is contradicted by condition \ref{OCAcond2}. If $n>\overline M$, then (\ref{OCAequation1}) leads to a contradiction. Finally, if $n\leq\overline M\leq\overline K<n'$, we can split $x=y+z$ where $y=q_{k}xq_k$ and $z=q_{(k,n']}xq_{(k,n']}$ for some $k<\overline K$, since $x\in D+f[E^0]$. We obtain a contradiction noting that $x\in D_f[E^0]$ implies $\alpha^0(\Phi_{f,E^0}(x))=\alpha^0(\Phi_{f,E^0}(y))+\alpha^0(\Phi_{f,E^0}(z))$ (the case of $x\in G_f[E]$ is treated similarly).

\end{proof}

Fix $\epsilon_k=2^{-k}$ and write $X=\bigcup_n \SX_{n,k}$ where each $\SX_{n,k}$ is $K_1^{\epsilon_k}$-homogeneous, thanks to $\OCA_\infty$. Since $<^*\times<_1$ is a $\sigma$-directed order, for every $k\in\NN$, we can find $D_k$ and $\SY_k$ such that 

\begin{itemize}
\item $D_k$ is a countable dense subset of $\SY_k$;
\item $\SY_k$ is $K^{\epsilon_k}_1$-homogeneous;
\item $\SY_k$ is $\leq^*\times\leq_1$-cofinal.
\end{itemize}

\begin{lemma}\label{lem:almostthere}

Suppose that $x_0,x_1$ are such that there are $n_l^0,n_l^1$ and $\langle(f_l,E_l,\alpha^0_l,\alpha^1)\rangle\subseteq D_k$, for $l\in\NN$ such that 

\begin{enumerate}[label=(\arabic*)]
\item\label{lem:cond:Borelgraph,c1} for every $l$ there is $i$ such that \[\max (E_l)_{2i}=n_l^0\text{ and }\max (E_l)_{2i+1}=n_l^1;\]
\item\label{lem:cond:Borelgraph,c2} if $l<l'$ then for every $i$ such that $\max (E_l)_i\leq\max\{n_l^0,n_l^1\}$ we have $(E_{l'})_i=(E_l)_i$ (the $E_l$'s extend themselves) and $f_l(i)=f_{l'}(i)$
\item\label{lem:cond:Borelgraph,c3} $q_{n_l^0}x_0q_{n_l^0}\in D_{f_l}[E_l^0]$, $q_{n_l^1}x_1q_{n_l^1}\in G_{f_l}[E_l]$
\end{enumerate}
Then 
\[
\norm{\pi(\lim_l\sum_{j<l}\alpha_l^0(\Phi_{f_l,E_l^0}(q_{(n_j^0,n_{j+1}^0]}x_0q_{(n_j^0,n_{j+1}^0]})))-\Lambda(\pi(x_0))}<10\epsilon_k
\] 
and
\[
\norm{\pi(\lim_l\sum_{j<l}\alpha_l^1(\Phi_{f_l,E_l^1}(q_{(n_j^1,n_{j+1}^1]}x_1q_{(n_j^1,n_{j+1}^1]})))-\Lambda(\pi(x_1))}<10\epsilon_k.
\]
\end{lemma}

\begin{proof}

We prove the statement for $x_0$, since the proof in the case of $x_1$ is equivalent. Given $\{n_l^0\}_{l\in\NN}$ and $\langle(f_l,E_l,\alpha^0_l,\alpha^1_l)\rangle$ as in the hypothesis, we can construct the partition $\hat E$ defining $\hat E_n=(E_l)_n$ if $\max (E_l)_n\leq n_l^0$. Note that by condition \ref{lem:cond:Borelgraph,c2}, $\hat E$ is well-defined. Define 
\[
x_{0,m}=q_{\hat E_{2m}\cup \hat E_{2m+1}}x_0q_{\hat E_{2m}\cup \hat E_{2m+1}}.
\]
 It is clear from the definition of $\hat E$ that $x_0=\sum_m x_{0,m}$. We can then pick $f$ big enough such that $x_0\in D_f[\hat E^0]$, since $x_0\in D[E^0]$. 

Since $\SY_k$ is $\leq^*\times\leq_1$-cofinal, there is $(g,F,\alpha^0,\alpha^1)\in \SY_k$ such that $f\leq^* g$ and $\hat E\leq_1 F$. By definition of $\leq_1$, we have that for every $n$ big enough there is a minimal $m=m(n)$ such that $\hat E_{2n}\cup\hat E_{2n+1}\subseteq F_m\cup F_{m+1}$, therefore we can write uniquely, according on whether $x_{0,n}\in D_f[E^0]$ or $x_{0,n}\in  G_f[E]$, $x_0=z_0+z_1$ with $z_0\in D_g[F^0]$, $z_1\in G_g[F]$. Note that $\pi(\alpha^0(\Phi_{g,F^0}(z_0)))=\Lambda(\pi(z_0))$, and similarly $\alpha^1(\Phi_{g,F^1}(z_1))=\Lambda(\pi(z_1))$. On the other hand, since for every $l$ we have 
\[
\{(f_l,E_l,\alpha^0_l,\alpha^1_l),(g,F,\alpha^0,\alpha^1)\}\in K_1^{\epsilon_k}
\]
by homogeneity of $\SY_k$, if $m\leq n_l^0$ we have that 
\[
\norm{\alpha^0_l(\Phi_{f_l,E_l^0}(x_{0,m}))-\alpha^0(\Phi_{g,F^0}(x_{0,m}))}\leq\epsilon_k\text{ if }x_{0,m}\in D_g[F^0]
\]
 and 
\[
\norm{\alpha^0_l(\phi_{f_l,E^0_l}(x_{0,m}))-\alpha^1(\Phi_{g,F^1}(x_{0,m})}\leq\epsilon_k\text{ if }x_{0,m}\in G_g[F].
\]
Since, modulo $\cstar{A}$, 
\[
z_0=\sum_{m\mid x_{0,m}\in D_g[F^0]}x_{0,m}\text{ and } z_1=\sum_{m\mid x_{0,m}\in G_g[F]}x_{0,m},
\]
passing to strict limits of partial sums we have the thesis.

\end{proof}

The following lemma provides the last step through the proof of Theorem \ref{thm:Borel}. Recall that $\Gamma_\Lambda$ is the graph of $\Lambda$.

\begin{lemma}

Assume $\OCA_{\infty}$ and $\MA_{\aleph_1}$. Let $\cstar{A}$, $q_n$ and $\Lambda$ as before and $a,b\in \mathcal M(\cstar{A})_{\leq 1}^+$. The following conditions are equivalent:

\begin{enumerate}[label=(\roman*)]
\item\label{thm:Borelgraph,c1} $(a,b)\in\Gamma_\Lambda$;
\item\label{thm:Borelgraph,c2} For every $k\in\NN$, there are $x_0,x_1, y_0, y_1\in\mathcal M(\cstar{A})_{\leq 1}^+$ such that $\pi(a)=\pi(x_0+x_1)$, $\pi(b)=\pi(y_0+y_1)$ with the property that there are two sequences $n_l^0,n_l^1$ and $\langle(f_l,E_l,\alpha^0_l,\alpha^1_l)\rangle$ a sequence of elements of $D_k$ satisfying conditions~\ref{lem:cond:Borelgraph,c1}--\ref{lem:cond:Borelgraph,c3} of Lemma~\ref{lem:almostthere}, 

\begin{enumerate}[label=(\arabic*)]\setcounter{enumii}{3}
\item\label{lem:cond:Borelgraph,c4}
\[
\norm{\lim_i\sum_{j<i}\alpha_i^0(\Phi_{f_i,E_i^0}(q_{(n_j^0,n_{j+1}^0]}x_0q_{(n_j^0,n_{j+1}^0]}))-y_0}<5\epsilon_k
\]
\item\label{lem:cond:Borelgraph,c5} and
\[ 
\norm{\lim_i\sum_{j<i}\alpha_i^1(\Phi_{f_i,E_i^1}(q_{(n_j^1,n_{j+1}^1]}x_1q_{(n_j^1,n_{j+1}^1]}))-y_1}<5\epsilon_k
\] 
where both limits are strict limits.
\end{enumerate}

\item\label{thm:Borelgraph,c3} For all $x^0$, $x^1$, $y^0$ and $y^1$ positive elements of norm $\leq 1$, if $\pi(x^0+x^1)=\pi(a)$ and for every $k\in\NN$ it is true that there are sequences $n_l^0,n_l^1$ and $(f_l,E_l,\alpha^0,\alpha^1)$ satisfying conditions \ref{lem:cond:Borelgraph,c1}--\ref{lem:cond:Borelgraph,c3} of Lemma \ref{lem:almostthere}, \ref{lem:cond:Borelgraph,c4} and \ref{lem:cond:Borelgraph,c5}, then $\pi(y^0+y^1)=\pi(b)$.
\end{enumerate}

\end{lemma}

\begin{proof}

Suppose that \ref{thm:Borelgraph,c1} holds and fix $k\in\NN$. By cofinality of $\SY_k$ we may find $(f,E,\alpha^0,\alpha^1)\in \SY_k$ and $x_0,x_1$ positive with the property that $\pi(a)=\pi(x_0+x_1)$, $x_0\in D_f[E^0]$ and $x_1\in G_f[E]$, thanks to Lemma \ref{lem:summing} and Proposition \ref{prop:mainproperties}. Let $y_0=\alpha^0(\Phi_{f,E^0}(x_0)$ and $y_1=\alpha^1(\Phi_{f,E^1}(x_1))$. Since $\alpha^0$ and $\alpha^1$ are chosen so that $(f,E,\alpha^0,\alpha^1)\in X$, we have that $\pi(y_0+y_1)=\pi(b)$. Let $n_{-1}^0=n_{-1}^1=0$ and suppose that $n_l^0,n_l^1$ and $(f_l,E_l,\alpha^0_l,\alpha^1_l)\in D_k$ are constructed. By density of $D_k$ we can find $n_{l+1}^0>n_l^0$, $n_{l+1}^0>n_l^1$ and $(f_{l+1},E_l,\alpha^0_l,\alpha^1_l)\in D_k$ with the property that 
\begin{itemize}
\item$(E_{l+1})_i=E_i$ for all $i$ such that $\max E_i\leq\max{n_{l+1}^0,n_{l+1}^0}$,
\item there is $i$ such that $\max E_{2i}=n_{l+1}^0$ and $j$ such that $\max E_{2j+1}=n_{l+1}^1$
\item if $\max E_i\leq \max{n_{l+1}^0,n_{l+1}^0}$ then $f_{l+1}(i)=f(i)$.
\end{itemize}
In particular such a construction ensures that conditions \ref{lem:cond:Borelgraph,c1}-\ref{lem:cond:Borelgraph,c3} of Lemma~\ref{lem:almostthere} are satisfied. Moreover, since for each $l$ we have that $(f,E,\alpha^0,\alpha^1),(f_l,E_l,\alpha^0_l,\alpha^1_l)\in K_1^{\epsilon_k}$, we have that for all $j\in\NN$
\[
\norm{\alpha_i^0(\Phi_{f_i,E_i^0}(q_{(n_j^0,n_{j+1}^0]}x_0q_{(n_j^0,n_{j+1}^0]}))-\alpha^0(\Phi_{f,E^0}(q_{(n_j^0,n_{j+1}^0]}x_0q_{(n_j^0,n_{j+1}^0]}))}<\epsilon_k,
\]
so
\[\norm{\lim_i\sum_{j\leq i}\alpha^0(\Phi_{f,E^0}(q_{(n_j^0,n_{j+1}^0]}x_0q_{(n_j^0,n_{j+1}^0]}))-\alpha^0(\Phi_{f,E^0}x_0)}<2\e_k.
\]
Since 
\[
y_0=\alpha^0(\Phi_{f,E^0}x_0)=\lim_i\sum_{j\leq i}\alpha^0(\Phi_{f,E^0}(q_{(n_j^0,n_{j+1}^0]}x_0q_{(n_j^0,n_{j+1}^0]})),
\]
applying the triangular inequality we get \ref{lem:cond:Borelgraph,c4}. A similar calculation leads to \ref{lem:cond:Borelgraph,c5}, and so we get \ref{thm:Borelgraph,c2}.

Assume now \ref{thm:Borelgraph,c2}. We should note that conditions~\ref{lem:cond:Borelgraph,c1}--\ref{lem:cond:Borelgraph,c5} in particular are implying that $\norm{\Lambda(\pi(x_0))-\pi(y_0)}\leq \epsilon_k$ and $\norm{\Lambda(\pi(x_1))-\pi(y_1)}\leq\epsilon_k$, therefore \ref{thm:Borelgraph,c1} follows. For this reason, we also have that \ref{thm:Borelgraph,c1} implies \ref{thm:Borelgraph,c3}. Similarly pick $a,b\in\mathcal M(\cstar{A})_{\leq1}$ both positive. If there are $x_0,x_1,y_0,y_1$ satisfying that for every $k$ there are $n_l^0,n_l^1$ and $(f_l,E_l,\alpha^0_l,\alpha^1_l)\in D_k$ satisfying conditions \ref{lem:cond:Borelgraph,c1}--\ref{lem:cond:Borelgraph,c5}, and such that $\pi(x_0+x_1)=\pi(a)$, then we have that $\pi(y_0+y_1)=\Lambda(\pi(x_0+x_1))$. If \ref{thm:Borelgraph,c3} holds, the left hand side is equal to $\pi(b)$, hence $(a,b)\in\Gamma_\Lambda$, proving \ref{thm:Borelgraph,c1}.
\end{proof}

\begin{proof}[Proof of Theorem \ref{thm:Borel}]
Condition \ref{thm:Borelgraph,c2} gives that $\Gamma_\Lambda^{1,+}=\Gamma_\Lambda\restriction \mathcal M(\cstar{A})^+_{\leq 1}\times \mathcal M(\cstar{A})^+_{\leq 1}$ is analytic, while \ref{thm:Borelgraph,c3} ensures that the graph is coanalytic. Consequently $\Gamma_{\Lambda}^{1,+}$ is Borel. As $(a,b)\in\Gamma_\Lambda$ if and only if $(a+a^*,b+b^*),(a-a^*,b-b^*)\in\Gamma_{\Lambda}$ and that, if $a$ and $b$ are self-adjoints then $(a,b)\in\Gamma_\Lambda$ if and only if $(|a|+a,|b|+b),(|a|-a,|b|-b)\in\Gamma_\Lambda^{1,+}$ and since addition, *, and absolute value are strictly continuous operations we have that $\Gamma_\Lambda$ is Borel.
\end{proof}

\subsection{Consequences II: Nice liftings and non-embedding theorems}\label{s:FA.Emb}
In this section we explore more consequences of the lifting result Theorem~\ref{thm:FA.liftingtheorem}. Recall (\S\ref{ss:ST.idealsonomega}) that an ideal $\SI\subseteq\mathcal P(\NN)$ is said dense if $\Fin\subseteq\SI$ and for every infinite $X\subseteq\NN$ there is an infinite $Y\subseteq X$ such that $Y\in\SI$.

\begin{theorem}\label{thm:noinjection}

Assume $\OCA_\infty$ and $\MA_{\aleph_1}$. Let $\SI\subseteq\mathcal P(\NN)$ be a meager dense ideal and $\cstar{A}$ a separable $\Cstar$-algebra admitting an increasing approximate identity of projections. Then, for any choice of nonzero unital $\Cstar$-algebras, there is no unital embedding $\phi\colon\prod A_n/\bigoplus_{\SI}A_n\to\mathcal M(\cstar{A})/\cstar{A}$. 
\end{theorem}

\begin{proof}
For $X\subseteq\NN$ we denote by $p_X\in\ell_\infty$ the canonical projection onto $X$, and by $\tilde p_X$ its image in $\ell_\infty/c_0$.

We argue by contradiction. Since each $A_n$ is unital, we can find an embedding $\ell_\infty/c_{\SI}\to\prod A_n/\bigoplus_{\SI}A_n$. We will prove that such an embedding cannot exist. Since $\SI$ contains all finite sets, $c_0\subseteq c_{\SI}$ and we can consider $\pi\colon\ell_\infty/c_0\to\ell_\infty/c_{\SI}$ the canonical quotient map. Let 
\[
\psi=\phi\circ\pi\colon\ell_\infty/c_0\to\mathcal M(\cstar{A})/\cstar{A}.
\]
By Theorem~\ref{thm:FA.liftingtheorem}, there exists an asymptotically additive $\alpha$ and a ccc/Fin ideal $\SJ$ on which $\alpha$ is a lift of $\psi$. By Theorem~\ref{thm:FA.ApproxStruct} we can assume that $\alpha$ is a $^*$-homomorphism. Since $\SJ$ is ccc/Fin, and so nonmeager, we can find an infinite $X\in\SJ\setminus \SI$. 

Then $\alpha(p_X)$ is a projection, $\alpha(p_X)\notin\cstar{A}$. Since $\alpha$ is an asymptotically additive $^*$-homomorphism, $\alpha=\sum_{n\in X}\alpha_n(1)$, and so there is $n_0$ such that for all $n\geq n_0$ we have $\norm{\alpha_{n}(1)}=1$. Note that all $\alpha_n(1)$ are orthogonal to each other. In particular we have that for all infinite $Y\in\SJ$, $\pi_{\cstar{A}}(\sum_{n\in Y}\alpha_{n}(1))$ is a nonzero projection, and so $\norm{\pi_{\cstar{A}}\alpha(p_Y)}=1$. $\pi_{\cstar{A}}\colon\mathcal M(\cstar{A})\to\mathcal M(\cstar{A})/\cstar{A}$ being the canonical quotient map. Let $Y\subseteq X$ be infinite, $Y\in\SI\cap\SJ$, by density of $\SI$. Since $\alpha$ is a lift on $\SJ$ we have that
\[
0=\norm{\psi(\tilde p_Y)}=\norm{\pi_{\cstar{A}}(\alpha(p_Y))}=1,
\]
 a contradiction.
\end{proof}

Note that the role of Forcing Axioms in the hypothesis is crucial. In fact, for every given ideal $\SI$ containing $\Fin$ we have that $\Th(\ell_\infty/c_0)=\Th(\ell_\infty/c_{\SI})$. In particular, by countable saturation of $\ell_\infty/c_0$ (see \cite{Farah-Shelah.RCQ} or \cite{EV.Sat}) under $\CH$ we have that $\ell_\infty/c_\SI$ embeds into $\ell_\infty/c_0$, by Theorem~\ref{thm:CHandSat}.

Theorem~\ref{thm:noinjection} has many corollaries. The prototypical separable $\Cstar$-algebra with an increasing approximate identity of projections is $\mathcal K(H)$.

\begin{corollary}\label{cor.FA.EmbCalkin}
Assume $\OCA_{\infty}$ and $\MA_{\aleph_1}$ and let $\SI\subseteq\mathcal P(\NN)$ be a meager dense ideal. Let $\cstar{A}_n$ be unital nonzero $\Cstar$-algebras. Then $\prod \cstar{A}_n/\bigoplus_{\SI}\cstar{A}_n$ does not embed into the Calkin algebra or into $\prod M_{k(n)}/\bigoplus M_{k(n)}$.
\end{corollary}
\begin{proof}
That there is no unital embedding of $\prod \cstar{A}_n/\bigoplus_{\SI}\cstar{A}_n$ in $\mathcal C(H)$ is Theorem~\ref{thm:noinjection}. As every cut down by a projection of $\mathcal C(H)$ is isomorphic to $\mathcal C(H)$, this concludes the proof. The second statement follows from that every cut down by a projection of $\prod M_{k(n)}/\bigoplus M_{k(n)}$ embeds unitally into $\prod M_n/\bigoplus M_n$.
\end{proof}

Corollary~\ref{cor.FA.EmbCalkin}, together with Theorem~\ref{thm:CH.EmbMFCH}, proves this generalization of Theorem~\ref{thmi:2}.
\begin{theorem}
Let $\SI\subseteq\mathcal P(\NN)$ be a meager dense ideal.
\item That $\prod A_n/\bigoplus _{\SI}A_n$ does not embed in the Calkin algebra for any choice of $A_n$ unital nonzero $\Cstar$-algebras is consistent with $\ZFC$;
\item That $\prod A_n/\bigoplus _{\SI}A_n$ embeds in the Calkin algebra for any choice of $A_n$ unital nonzero MF $\Cstar$-algebras is independent from $\ZFC$.
\end{theorem}

Instances of the following theorem where showed to be valid in a model of set theory (obtained via the use of forcing) by Ghasemi in \cite{Ghasemi.FDD}.

\begin{theorem}
Assume $\OCA_\infty$ and $\MA_{\aleph_1}$. Let $\cstar{A}_n$ be unital infinite-dimensional $\Cstar$-algebras.
Then $\prod \cstar{A}_n/\bigoplus \cstar{A}_n$ doesn't embed into $\prod M_{k(n)}/\bigoplus M_{k(n)}$, for any choice $k(n)$.
\end{theorem}
\begin{proof}
By contradiction, let $\cstar{A}_n$ be unital and infinite-dimensional and $\Lambda\colon\cstar{B}=\prod\cstar{A}_n/\bigoplus\cstar{A}_n\to\prod M_{k(n)}/\bigoplus M_{k(n)}$ be an embedding. Since $\prod M_{k(n)}/\bigoplus M_{k(n)}$ embeds unitally into $\prod M_n/\bigoplus M_n$ we will assume that $k(n)=n$. Also, as every corner of $\prod M_n/\bigoplus M_n$ is isomorphic to $\prod M_{k(n)}/\bigoplus M_{k(n)}$ for some sequence $k(n)$,  we can assume that $\Lambda$ is unital.

Let $\ell_\infty/c_0\subseteq Z(\cstar{B})$, be the canonical copy generated by the image of $p_A$, $A\subseteq\NN$, where $(p_A)_n=1$ if $n\in A$ and $0$ otherwise.
Let $\alpha=\sum\alpha_n\colon\ell_\infty\to\prod M_n$ be the asymptotically additive map which is a lifting of $\Lambda\restriction \ell_\infty/c_0$ on a nonmeager ideal $\SI$. By the definition of asymptotically additive there are increasing sequences $n_i,m_i$ with $n_i<m_i$ such that the range of $\alpha_i$ is contained in $\prod_{n_i\leq j\leq m_i}M_j$. Note that we are not requiring, at this stage, that $m_i\leq n_{i+1}$.

As $\Lambda\restriction\ell_\infty/c_0$ is a $^*$-homomorphism by Theorem~\ref{thm:FA.ApproxStruct} we can assume that each $\alpha_i$ is a $^*$-homomorphism whose range is included in $\prod_{n_i\leq j\leq m_i}M_j$ and $\alpha=\sum\alpha_i$ is a lift for $\Lambda\restriction\ell_\infty/c_0$ on $\SI$.

Fix $\cstar{B}_i=\prod_{n_i\leq j\leq m_i} M_j$ and let $R_i$ be a natural number so large that there is no set $\{x_k\}_{k\leq R_i}\subseteq (\cstar{B}_i)$ with $\norm{x_kx_l}<\frac{1}{4}$ and $\frac{1}{2}\leq\norm{x_k}\leq2$ for all $k,l\leq R_i$. The existence of such a number is possible since each $\cstar{B}_i$ is finite dimensional.

For every $i$ fix $2R_i-1$ pairwise orthogonal positive elements of norm $1$,  $y_1,\ldots,y_{2R_i-1}$ and  let $y_{2R_i}=1-\sum_{j<2R_i} y_j$. (Choose the $y_k$'s so that $\norm{y_{2R_i-1}}=1$). Let 
\[
\gamma_i\colon M_{2R_i}\to \cstar{A}_i
\]
be the map given by \[\gamma_i(x)=\sum_{j\leq 2R_i} e_{j,j}xe_{j,j}y_j,\]
where $e_{j,j}$ is the class matrix unit for $M_{2R_i}$. Note that $e_{j,j}xe_{j,j}$ is a complex number, and so each $\gamma_i$ is a linear $^*$-preserving positive map. Fix $\Gamma=\sum\gamma_i\colon\prod M_{2R_i}\to\prod\cstar{A}_i$. Note that $\Gamma(\bigoplus M_{2R_i})\subseteq \bigoplus \cstar{A}_i$, and so $\Gamma$ induces a map $\Gamma'\colon\prod M_{2R_i}/\bigoplus M_{2R_i}\to\prod\cstar{A}_i/\bigoplus\cstar{A}_i$. Let $\Delta=\Lambda\circ\Gamma'$. It is easy to see that $\Delta$ satisfies the hypothesis of Theorem~\ref{thm:FA.liftingtheorem}, in particular Equation~\eqref{eqn:restrictions} of \S\ref{s:FA.Lift1}, as each $\gamma_i$ is unital. Since we assumed $\OCA_\infty$ and $\MA_{\aleph_1}$ there is an asymptotically additive map 
\[
\beta=\sum\beta_n\colon\prod M_{2R_i}\to\prod M_n
\]
which is a lifting of $\Delta$ on a nonmeager $\SJ$. Since the intersection of two nonmeager ideals is still a nonmeager ideals, we can assume $\SI=\SJ$.
\begin{claim}
\begin{itemize}
\item $\lim_i\sup_{j,k<2R_i}\norm{\beta_i(e_{j,j})\beta_i(e_{k,k})}=0$
\item $\lim_i\sup_{j<2R_i}\norm{\beta_i(e_{j,j})}=\lim_i\inf_{j<2R_i}\norm{\beta_i(e_{j,j})}=1$
\item $\lim_i\sup_{j<2R_i}\norm{\beta_i(e_{j,j})-1_{\cstar{B}_i} \beta_i(e_{j,j}) 1_{\cstar{B}_i}}=0$
\end{itemize}
\end{claim}
\begin{proof}
 We will only prove the first statement, as the proof for the second and third is precisely the same, and we therefore leave it to the reader.

The proof is similar to the one of Proposition~\ref{prop:FA.ApproxStruct}. Suppose that there is $\epsilon>0$ and infinite sequences $i_l,j_l,k_l$ with the property that $j_l,k_l<2R_{i_l}$ and 
\[
\norm{\beta_{i_l}(e_{j_l,j_l})\beta_{i_l}(e_{k_l,k_l})}>\epsilon
\]
Fix $X\subseteq\{i_l\}$ with $X\in\SI=\SJ$. Let $a,b\in \prod M_{2R_i}$ defined as $a_n=e_{j_l,j_l}$ if $n=i_l\in X$ and $0$ otherwise and $b_n=e_{k_l,k_l}$ if $n=i_l\in X$ and $0$ otherwise. In particular it follows that 
\[
\norm{\pi_1(\beta(a))\pi_1(\beta(b))}>\epsilon,
\]
where $\pi_1\colon\prod M_n\to\prod M_n/\bigoplus M_n$ is the canonical quotient map.

 On the other hand, since $X\in\SI$, we have that $\beta(a)$ and $\beta(b)$ are liftings for $\Delta(\pi(a))=\Lambda(\Gamma'(\pi(a)))$ and $\Delta(\pi(b))=\Lambda(\Gamma'(\pi(b)))$, $\pi$ denoting the quotient map $\pi\colon \prod M_{2R_i}\to\prod M_{2R_i}/\bigoplus M_{2R_i}$, that is
\[
\pi_1(\beta(a))=\Lambda(\Gamma'(\pi(a)))\text{ and } \pi_1(\beta(b))=\Lambda(\Gamma'(\pi(b))).
\]

On the other hand, by the definition of $\Gamma'$ we have that, for $\pi_2\colon\prod\cstar{A}_n\to\prod\cstar{A}_n/\bigoplus\cstar{A}_n$,
\[
\Gamma'(\pi(a))=\pi_2(\Gamma(a))=\pi_2(x), \text { and } \Gamma'(\pi(b))=\pi_2(\Gamma(b))=\pi_2(z)
\]
where $x=\Gamma(a)$ and $z=\Gamma(b)$ are such that $x_n=\gamma_n(e_{j_l,j_l}=y_{j_l}\in\cstar{A}_{i_l}$ and $z_n=\gamma_n(e_{j_l,j_l}=y_{k_l}\in\cstar{A}_{i_l}$ if $n=i_l\in X$ and $0$ otherwise. Since $j_l\neq k_l<2R_i$ we have that $xz=0$, and so $\pi_2(x)\pi_2(z)=0$ meaning that $\Gamma'(\pi(a))\Gamma'(\pi(b))=0$ and so 
\[
0=\Lambda(\Gamma'(\pi(a)))\Lambda(\Gamma'(\pi(b)))=\pi_1(\beta(a))\pi_1(\beta(b))
\]
a contradiction.
\end{proof}

As $i$ is large enough, $c_{i,j}= 1_{\cstar{B}_i} \beta_i(e_{j,j}) 1_{\cstar{B}_i}$, for $j<2R_i$ are $2R_i-1$ elements of $\cstar{B}_i$ of norm greater than $\frac{1}{2}$, less than $2$, and such that $\norm{c_{i,j}c_{i,l}}<\frac{1}{4}$ whenever $j,l<2R_i$, a contradiction to our choice of $R_i$.

\end{proof}

Another interesting application occurs when one considers embeddings of the form 
\[
\phi\colon\prod M_{k(n)}/\bigoplus M_{k(n)}\to\mathcal M(\cstar{A})/\cstar{A}.
\]
\begin{theorem}\label{thm:.FA.starhomolift}
Assume $\OCA_\infty$ and $\MA_{\aleph_1}$. Let $k(n)$ be a sequence of natural numbers, $\cstar{A}$ be a separable $\Cstar$-algebra with an increasing approximate identity of projections and
\[
\phi\colon\prod M_{k(n)}/\bigoplus M_{k(n)}\to\mathcal M(\cstar{A})/\cstar{A}.
\]
be a unital embedding. Then there is a $^*$-homomorphism  
\[\Phi\colon\prod M_{k(n)}\to\mathcal M(\cstar{A})
\]
 and a ccc/Fin ideal $\SI$ such that $\Phi$ is a lifting of $\phi$ on $\SI$.
\end{theorem}

\begin{proof} 
This is Theorem~\ref{thm:FA.liftingtheorem} together with Theorem~\ref{thm:FA.ApproxStruct}.
\end{proof}

In case the map $\phi$ is an isomorphism, we are able to enlarge $\SI$. 

\begin{theorem}\label{thm:isoFDD}
Assume $\OCA_\infty$ and $\MA_{\aleph_1}$. Let $\cstar{A}$ be a separable $\Cstar$-algebra with an increasing approximate identity of projections and suppose that $\prod M_{k(n)}/\bigoplus M_{k(n)}$ and $\mathcal M(\cstar{A})/\cstar{A}$ are isomorphic. Then there is a projection $q\in\mathcal M(\cstar{A})$ and $k\in\NN$ such that 
\[
1-q\in\cstar{A},\,\,  q\cstar{A}q\cong\bigoplus_{n\geq k} M_{k(n)} \text{ and  }q\mathcal M(\cstar{A})q\cong\prod_{n\geq k} M_{k(n)}.\]

Moreover, if $\Lambda\colon\prod M_{k(n)}/\bigoplus M_{k(n)}\to \mathcal M(\cstar{A})/\cstar{A}$ is an isomorphism, there is a $^*$-homomorphism $\Phi$ which a lift of $\Lambda$, 
\[
\phi\colon \prod_n M_{k(n)}\to \mathcal M(\cstar{A})
\] 
and a $k\in\NN$ for which $\phi\restriction\prod_{n\geq k}M_{k(n)}$ is an isomorphism.
\end{theorem}

\begin{proof}
We will prove the second claim, from which the first follows easily. To simplify the notation, we will assume that $k(n)=n$, but the reader will see that the proof goes through for any choice of the sequence $k(n)$. By $\pi,\pi_1$ we denote the canonical quotient maps $\pi\colon\prod M_n\to\prod M_n/\bigoplus M_n$ and $\pi_1\colon\mathcal M(\cstar{A})\to\mathcal M(\cstar{A})/\cstar{A}$.

Let $\Lambda$ be an isomorphism
\[
\Lambda\colon\prod M_n/\bigoplus M_n\to\mathcal M(\cstar{A})/\cstar{A}
\]
and $\phi\colon \prod M_n\to\mathcal M(\cstar{A})$ be a strict-strict continuous $^*$-homomorphism which is  lift on a ccc/Fin ideal $\SI$, as ensured by Theorem~\ref{thm:.FA.starhomolift}.

We argue as in Lemma~\ref{lem:liftforeverything.Borel} to get that $\SI=\mathcal P(\NN)$. We first prove that $\pi_1(\phi(1))=1$. Recall that if $A\subseteq \NN$, $\tilde p_A$ denotes the projection in $\prod M_n$ such that $(\tilde p_A)_n=1$ if $n\in A$ and $0$ otherwise. Note that if $p\in \prod M_n$ is such that $p$ is a projection dominating each $\tilde p_A$, for $A\in\SI$, then $p=1$, as $\SI$ is dense (see Proposition~\ref{prop:onedominates}).
Let $r=\pi_1(\phi(1))$. Then for all $A\in\SI$
\[
\Lambda^{-1}(r)\geq\Lambda^{-1}(\pi(\phi(\tilde p_A)))=\Lambda^{-1}(\Lambda (\tilde p_A))=\tilde p_A
\]
 and so $\Lambda^{-1}(r)=1$, that is $r=1$.

We now prove that $\SI=\mathcal P(\NN)$. Fix $x\in\prod M_n$ and $y\in\mathcal M(A)$ such that $\Lambda(\pi(x))=\pi_1(y)$, and let 
\[
\SI_x=\{A\subseteq\NN\mid \phi(\tilde p_A)(\phi(x)-y)\in\cstar{A}\}.
\]
Since $\phi$ is strictly-strictly continuous, $x$ and $y$ are fixed, and $\cstar{A}$ is Borel, $\SI_x$ is Borel. On the other hand, $\SI\subseteq\SI_x$ and so $\SI_x=\mathcal P(\NN)$. As $\NN\in\SI_x$  we have
\[
\pi_1(\phi(x))=\Lambda(\pi(x)).
\]
Consider now $q=\phi(1)$. With $p_n=\phi(1_n)$, where $1_n=\tilde p_{\{n\}}\in M_n$, and $q_n=\sum_{i\leq n} p_n$, we have that $q_n$ is an increasing sequence of projections converging (strictly) to $q$, and so is an approximate identity in $q\cstar{A}q$. Also, whenever $X\subseteq\NN$ then the projection $p_X=\sum_{n\in X} p_n$ is such that if $a\in\mathcal M(\cstar{A})$ then 
\[
p_Xa-ap_X\in\cstar{A},
\]
 as $\pi(p_X)$ is central in $\mathcal M(\cstar{A})/\cstar{A}$.

By Lemma~\ref{lem:apoxcommuting}, there is $n_0$ such that $p_n$ is central in $(q-q_{n_0})\cstar{A}(q-q_{n_0})$ for all $n\geq n_0$, therefore we have that
\[
(q-q_{n_0})\cstar{A}(q-q_{n_0})\cong \bigoplus_{n>n_0}p_n\cstar{A}p_n.
\]
\begin{claim}
There is $k\geq n_0$ such that if $n>k$ then $\phi\restriction M_n\colon M_n\to p_n\cstar{A}p_n$ is an isomorphism.
\end{claim}
\begin{proof}
Note that there is $k$ such that if $n\geq k$ then $\phi\restriction M_n$ is nonzero. If not there is an infinite sequence $n_i$ such that $\phi\restriction M_{n_i}=0$, then, with $X=\{n_i\}$, we have that $\phi(\tilde p_X)\in \cstar{A}$, a contradiction to $1=\norm{\Lambda(\pi(\tilde p_X))}$ and that $\phi$ a lifting of $\Lambda$.
Since each $M_n$ is simple, every $\phi\restriction M_n$, for $n\geq k$, is injective. 

We now want to show that $\phi\restriction M_n$ is eventually surjective. If not, there is an infinite sequence $k_i$ such that $\phi\restriction M_{k_i}$ is not surjective. In this case, the vector space $\phi\restriction M_{k_i}$ is properly contained in the vector space $p_{k_i}\cstar{A}p_{k_i}$, and so there is $a_i\in p_{k_i}\cstar{A}p_{k_i}$ with $d(a_i,\phi\restriction M_{k_i})=1$ and $\norm{a_i}=1$. Fix $\bar a=\sum a_i\in\mathcal M(\cstar{A})\setminus \cstar{A}$, and let $\bar b\in\prod M_n$ with the property that
\[
\Lambda^{-1}(\pi_1(\bar a))=\pi(\bar b).
\]
Since $\phi$ is a lifting of $\Lambda$, we have that $\norm{q_n(\phi(\bar b)-\bar a)q_n}\to 0$, a contradiction.
\end{proof}
By setting $\tilde q=q-q_{n_0}$ we have the thesis.
\end{proof}

\chapter{Some open questions}\label{ch:QandA}

In this chapter we list few open questions and discuss the future prospective developments of the research carried over in this thesis.

Regarding the study of different layers of saturation studied in Chapter~\ref{ch:CH}, we recall Question~\ref{q:intro2}:
\begin{question}
Under $\CH$, is there an infinite-dimensional countably degree-1 saturated algebra of density character $\aleph_1$ with only $\mathfrak c$-many automorphisms?
\end{question}

As every corona of a $\sigma$-unital $\Cstar$-algebra is countably degree-1 saturated, by \cite{farah2011countable}, a negative answer to this question would solve Conjecture~\ref{conj:CH}.

Another problem concerning the different layers of saturation is to understand if, and when, they can differ. In \cite{Farah-Shelah.RCQ}, the authors exhibited  many examples of corona failing to be countably saturated, showing the existence of abelian ones (if $X=\{(x,y)\in\er^2\mid x=0\text{ or } y\in\NN\}$ then $\beta X\setminus X$ has only countably many clopen sets). All such examples fail to be quantifier-free saturated.

\begin{question}
Is every countably quantifier-free saturated $\Cstar$-algebra countably saturated?
\end{question}

Examples of algebras which are not known to be, or not to be, countably saturated are $M/\mathcal K_M$, where $M$ is a II$_\infty$-factor with separable predual and $\mathcal K_M$ is its unique ideal of finite trace elements, and $C(\beta X\setminus X)$, where $X$ is $\er^n$, for $n\geq 2$. An open question, related to saturation, asks whether these spaces can be universal across continua. If $C(\beta\er^n\setminus\er^n)$ is countably saturated, then we have a positive answer to the following.
\begin{question}
Assume $\CH$. Does $\beta\er^n\setminus\er^n$ surjects onto every compact connected space of density character $\mathfrak c$?
\end{question}

Turning to show that certain coronas have large automorphisms groups under $\CH$, we analyze the construction of $X$ as in \S\ref{ss:arigidspace}. In this case,  the obstructions preventing the currently known methods to construct an homeomorphism of $\beta X\setminus X$ different from the identity (such an homeomorphism would have to be necessarily nontrivial), can be generalized to construct higher dimensional versions of $X$ with the same properties. On the other hand, it seems that, if one were able to show that $\beta X\setminus X$ has nontrivial homeomorphisms under $\CH$, then one could generalize such proof to the higher dimensional analogues of $X$.
Similarly, in the nonabelian case, it is possible to identify a $\Cstar$-algebra $A$ for which the current methods provided by Farah and Coskey do not apply in the search for a nontrivial automorphism of $\mathcal M(A)/A$. Solving this test case should give an idea to the difficulties one has to overcome to solve the general conjecture.

Again in the noncommutative setting, we can turn to analyze the structure of corona of non necessarily $\sigma$-unital algebras, in the spirit of \S\ref{s:CH.instances}.
\begin{question}
Let $M$ be a II$_\infty$-factor with separable predual. Is it independent from $\ZFC$ whether $M/\mathcal K_M$ has only inner automorphisms?
\end{question}

Regarding the applications of forcing axioms to show all homeomorphisms of a \v{C}ech-Stone remainder are trivial, a partial result was provided by Farah and Shelah. Recall that if $b\subseteq X$ is a closed set, one can canonically associate to $b$ a closed set, namely $\beta b\setminus b$. Unless $b$ is compact, $\beta b\setminus b$ is nonempty.

\begin{defin}
Let $\phi\in\Homeo(\beta X\setminus X)$. We say that $\phi$ has a representation if for all $b\subseteq X$ closed there is a closed set $c\subseteq X$ such that $\phi[\beta b\setminus b]=\beta c\setminus c$.
\end{defin}
Having a representation is weaker than being trivial, in fact, every trivial homeomorphism has a representation. In case Forcing Axioms are assumed the two definitions coincide.
\begin{theorem}[{\cite[Theorem~5.3]{Farah-Shelah.RCQ}}]
Assume PFA, let $X$ be a locally compact Polish space and let $\phi\in \Homeo(\beta X\setminus X)$. If $\phi$ has a representation then $\phi$ is trivial.
\end{theorem}

Reading through their proof, it seems it would be possible to ask only for a local representation instead.
\begin{defin}
Let $\phi\in\Homeo(\beta X\setminus X)$. We say that $\phi$ has a local representation if for all $b\subseteq X$ closed and noncompact there are $c,d\subseteq X$ closed and noncompact, with $c\subseteq b$, and such that $\phi[\beta c\setminus c]=\beta d\setminus d$. 
\end{defin}
\begin{question}
Assume PFA. Let $X$ be Polish and locally compact and $\phi\in\Homeo(\beta X\setminus X)$. If $\phi$ has a local representation, is $\phi$ trivial?
\end{question}

A test case in analyzing which homeomorphisms of $\beta X\setminus X$ have a (local) representation, under the assumption of some Forcing Axioms, is given if $X$ is the disjoint union of countably many compact sets $X_n$. In this case, $C_0(X)=\bigoplus C(X_n)$ and $C(\beta X\setminus X)\cong \prod C(X_n)/\bigoplus C(X_n)$. Ulam stability type results of S\v{e}mrl (\cite{Semrl.USAbel}) show that every trivial homeomorphism of $\beta X\setminus X$ corresponds to a $^*$-homomorphisms $\prod C(X_n)\to \prod C(X_n)$ which is a $^*$-isomorphism up to finite indexes. The (topological) study of the possible copies of $\beta \omega\setminus\omega$ inside $\beta X\setminus X$, and of which ones come from sets of the form $\beta b\setminus b$, is key in obtaining results similar to Thoerem~\ref{thm:isoFDD}.
\begin{conjecture}
Assume Forcing Axioms. Let $X_n$, $Y_n$ be compact and Polish, and $X$ (resp. $Y$) be the disjoint union of the $X_n$'s (resp, the $Y_n$'s). Then for every isomorphism $\phi\colon C(\beta X\setminus X)\to C(\beta Y\setminus Y)$ there are $k_1,k_2$ and an isomorphism \[\prod_{n\geq k_1}C(X_n)\to \prod_{n\geq k_2} C(Y_n)\] which is a lift of $\phi$.

In particular there is an $f$ which is an almost permutation of $\NN$ such that $X_n$ is homeomorphic to $Y_{f(n)}$ for all $n$ on which $f$ is defined.
\end{conjecture}

Again analyzing reduced products, and turning to the non abelian case, here is what is known:
\begin{theorem}[{\cite{McKenney.UHF}}]\label{thm:MckUHF}
Assume Forcing Axioms. Let $A_n,B_n$ be unital separable UHF algebras. Then $\prod A_n/\bigoplus A_n\cong \prod B_n/\bigoplus B_n$ if and only if there is $f$, an almost permutation of $\NN$, such that $A_n\cong B_{f(n)}$ whenever $f(n)$ is defined.
\end{theorem}
Given an isomorphism of reduced products of unital separable UHF algebras, the difficulties in always obtaining a $^*$-homomorphism which is a lifting, rely in that it is not known whether an $\epsilon$-morphism between such algebras is uniformly close to an actual morphism. With in mind the definition of Ulam stability (see Definition~\ref{defin:US.ulamstab}), and denoting by $\mathcal{AF}_s$ the class of all unital simple separable AF algebras, we state the following conjecture:
\begin{conjecture}\label{conj:UlamStable}
The following are equivalent:
\begin{enumerate}
\item\label{thm.USAppl1} Under Forcing Axioms, for every choice of $\{A_n\}\subseteq\mathcal{AF}_s$ every automorphism of $\prod A_n/\bigoplus A_n$ has a lifting which is a $^*$-homomorphism;
\item\label{thm.USAppl2} $(\mathcal{AF}_s,\mathcal{AF}_s)$ is Ulam stable;
\item\label{thm.USAppl3} for every $\epsilon>0$ there is $\delta>0$ such that for every $A\in\mathcal{AF}_s$ and two representations $\pi_1,\pi_2\colon A\to\mathcal B(H)$ with $d_{KK}(\pi_1(A),\phi_2(A))<\delta$ there is an isomorphism $\phi\colon A\to A_2$ with $\norm{\pi_1(\phi(a))-\pi_2(a)}<\epsilon$ for all $a\in A_1$.
\end{enumerate}
\end{conjecture}

If one wants to state conjectures regarding isomorphisms of reduced products, while remaining in the setting in which Theorem~\ref{thm:FA.liftingtheorem} applies, one can go beyond the classes of UHF or simple AF algebras and consider the class of all unital nuclear simple separable algebras. To generalize McKenney's Theorem~\ref{thm:MckUHF} to this setting, one would have to state and prove results of stability nature for maps which are not necessarily almost $^*$-homomorphisms, but are only almost cpc maps. This is because nuclear $\Cstar$-algebras can be seen as limits in the category of finite-dimensional $\Cstar$-algebra where the maps considered are completely positive contractions.

\begin{question}\label{ques:redprodnuc}
Is every approximately cpc map from a finite-dimensional $A$ to a $\Cstar$-algebra $B$ close to a cpc map by a factor independent from $A$ and $B$?
\end{question}
A positive answer to Question~\ref{ques:redprodnuc} would open the door to the formulation of a strong rigidity result, again in the spirit of Theorem~\ref{thm:MckUHF}, but when considering $\mathcal {N}_s$, the class of unital separable  simple nuclear algebras.
\begin{conjecture}
Assume Forcing Axiom.  If $A_n,B_n\in\mathcal N_s$, then $\prod A_n/\bigoplus A_n\cong \prod B_n/\bigoplus B_n$ if and only if there is $f$, an almost permutation of $\NN$, such that $A_n\cong B_{f(n)}$ whenever $f(n)$ is defined.
\end{conjecture}
Again as before, to obtain an actual $^*$-homomorphism which is a lift for an isomorphism between reduced products of objects in $\mathcal N_s$, it is conceivable to state conjectures relatively to whether or not $(\mathcal N_s,\mathcal N_s)$ is Ulam stable. The latter seems to be a very hard problem, and it has been open, even in the case of unital separable UHF algebras, for more than 4 decades, when the first discussion of these sort of phaenomena appeared in \cite{PR:pertAF} (see also \cite{Christensen.NI} or \cite{christensen2012perturbations}).

\bibliographystyle{amsplain}
\bibliography{dabib}
\end{document}